\documentclass[10pt]{article}

\usepackage[margin=1.53in]{geometry}

\usepackage{amsthm,amsmath}
\usepackage{amssymb,amscd}
\usepackage{mathrsfs,eucal}
\usepackage{enumerate}
\usepackage{color}
\usepackage[all,cmtip]{xy}
\usepackage{tikz}

\makeatletter
\renewcommand{\section}{\@startsection%
  {section}
  {1}
  {0em}
  {\baselineskip}
  {0.5\baselineskip}
  {\normalfont\large\bfseries}}
\renewcommand{\subsection}{\@startsection%
  {subsection}
  {2}
  {0em}
  {\baselineskip}
  {0.25\baselineskip}
  {\normalfont\bfseries}
}
\makeatother

\newtheorem{thm}{Theorem}[section]

\newtheorem{lem}[thm]{Lemma}
\newtheorem{prop}[thm]{Proposition}
\newtheorem{conj}[thm]{Conjecture}
\theoremstyle{definition}
\newtheorem{defn}[thm]{Definition}

\newtheorem{rem}[thm]{Remark}
\newtheorem{ex}[thm]{Example}
\newtheorem*{rem*}{Remark}

\newcommand\independent{\protect\mathpalette{\protect\independenT}{\perp}}
\def\independenT#1#2{\mathrel{\rlap{$#1#2$}\mkern2mu{#1#2}}}

\begin{document}

\title{\bf Phase Transitions in Nonlinear Filtering\thanks{Supported
in part by NSF grant CAREER-DMS-1148711 and by the ARO through a
PECASE award.}}

\author{Patrick Rebeschini\thanks{Yale University, 17 Hillhouse Avenue, 
New Haven, CT 06520. E-mail: patrick.rebeschini@yale.edu} \and Ramon van 
Handel\thanks{Sherrerd Hall, Princeton University,    
Princeton, NJ 08544. E-mail: rvan@princeton.edu}}

\date{}

\maketitle

\begin{abstract} 
It has been established under very general conditions that the ergodic 
properties of Markov processes are inherited by their conditional 
distributions given partial information.  While the existing theory 
provides a rather complete picture of classical filtering models, many 
infinite-dimensional problems are outside its scope.  Far from being a 
technical issue, the infinite-dimensional setting gives rise to surprising 
phenomena and new questions in filtering theory.  The aim of this paper is 
to discuss some elementary examples, conjectures, and general theory that 
arise in this setting, and to highlight connections with problems in 
statistical mechanics and ergodic theory.  In particular, we exhibit a 
simple example of a uniformly ergodic model in which ergodicity of the 
filter undergoes a phase transition, and we develop some qualitative 
understanding as to when such phenomena can and cannot occur.
We also discuss closely related problems in the setting of conditional 
Markov random fields.

\smallskip\smallskip\noindent
{\footnotesize\textbf{Keywords:} filtering in infinite dimension; 
conditional ergodicity \& mixing; phase transitions.}
\end{abstract}

\smallskip

\section{Introduction}

Let $(X_k,Y_k)_{k\ge 0}$ be a bivariate Markov chain.  Such a model
represents the setting of partial information: it is presumed that only 
$(Y_k)_{k\ge 0}$ can be observed, while $(X_k)_{k\ge 0}$ defines the 
unobserved dynamics.  In order to understand the behavior of the unobserved 
process given the observations, it is natural to `lift' the unobserved 
dynamics to the level of conditional distributions, that is, to 
investigate the nonlinear filter
$$
	\pi_k := \mathbf{P}[X_k\in\,\cdot\,|Y_1,\ldots,Y_k].
$$
Under standard assumptions on the observation structure (cf.\ section 
\ref{sec:model}), the process $(\pi_k)_{k\ge 0}$ is 
itself a measure-valued Markov chain.  The fundamental question that arises in 
this setting is to understand in what manner the probabilistic structure of 
the model $(X_k,Y_k)_{k\ge 0}$ `lifts' to the conditional distributions 
$(\pi_k)_{k\ge 0}$.

Of particular interest in this context is the behavior of ergodic properties 
under conditioning.  It is natural to suppose that the ergodic properties of 
$(X_k,Y_k)_{k\ge 0}$ will be inherited by the filter $(\pi_k)_{k\ge 0}$: for 
example, if $X_k$ forgets its initial condition as $k\to\infty$, then the 
optimal mean-square estimate of $X_k$ (and therefore the filter $\pi_k$) 
should intuitively possess the same property.  Such a conclusion was already 
conjectured by Blackwell as early as 1957 \cite{Bla57}, and a proof was 
provided by Kunita in 1971 \cite{Kun71}.  Unfortunately, both the proof and 
the conclusion are erroneous: it is elementary to construct a 
finite-state Markov chain $(X_k,Y_k)_{k\ge 0}$ that is 1-dependent 
(as strong an ergodic property as one could hope for) 
with observations of the form $Y_k=h(X_{k-1},X_k)$ such 
that the corresponding filtering process $(\pi_k)_{k\ge 0}$ is nonergodic, 
see Example \ref{ex:bla} below.\footnote{
	Surprisingly, the counterexample (intended for a different purpose)
	appears in Blackwell's own paper \cite{Bla57}.
}

Despite the appearance of counterexamples already in the most elementary 
setting, recent advances have provided a surprisingly complete picture of 
such problems in a general setting.  On the one hand, it has been shown 
under very general assumptions \cite{vH09c,TvH13} that ergodicity of the 
underlying model is inherited by the filter when the observations are 
\emph{nondegenerate}, that is, when the conditional law of each 
observation $\mathbf{P}[Y_k\in\,\cdot\,|X]$ has a positive density with 
respect to some fixed reference measure.  This is a mild condition in 
classical filtering models that serves mainly to rule out the singular 
case of noiseless observations: for example, the addition of any 
observation noise to the above counterexample would render the filter 
ergodic.  On the other hand, even in the noiseless case, ergodicity is 
inherited in the absence of certain symmetries that are closely related to 
systems-theoretic notions of \emph{observability} \cite{vH08, vH09, vH09b, 
CvH10}.  One can therefore conclude that while there exist elementary 
examples where the ergodicity of the model fails to be inherited by the 
filter, such examples must be very fragile as they require both a singular 
observation structure and the presence of unusual symmetries, either of 
which is readily broken by a small perturbation of the model.

The theory outlined above provides a satisfactory understanding of 
conditional ergodicity in classical filtering models.  Some care must be 
taken, however, in interpreting this conclusion.  The ubiquitous 
applicability of the theory hinges on the notion that most filtering 
models possess observation densities, an assumption made almost 
universally in the filtering literature (cf.\ \cite{CR11} and the 
references therein).  This assumption is largely innocuous in 
finite-dimensional systems.  The situation is entirely different in 
infinite dimension, where singularity of probability measures is the norm.  
There exists almost no mathematical literature on filtering in infinite 
dimension, despite the substantial practical importance of 
infinite-dimensional filtering models in data assimilation problems that 
arise in areas such as weather forecasting or geophysics \cite{Stu10}.  
The aim of this paper is to draw attention to the fact that, far from 
being a technical issue, the infinite-dimensional setting gives rise to 
new probabilistic phenomena and questions in filtering theory that are 
fundamentally different than those that have been studied in the 
literature to date, and whose understanding remains limited.

\begin{rem}
The singularity of measures in infinite dimension is already a significant 
hurdle in the ergodic theory of infinite-dimensional Markov processes, as 
it precludes application of standard methods based on Harris recurrence 
(e.g., \cite{Mat07}).  An infinite-dimensional unobserved process is not, 
however, the primary source of the problems considered in the sequel: as 
long as the observations are nondegenerate, it was shown in \cite{TvH13} 
that ergodicity is still inherited by the filter.  The problem is that 
nondegeneracy is a restrictive assumption on the observations in infinite 
dimension: it implies that the observations must be effectively 
finite-dimensional \cite[Remark 5.20]{TvH13}.  It is in models were the 
observations are genuinely infinite-dimensional that new phenomena appear.
\end{rem}

To model a filtering problem in infinite dimension, we suppose that 
$(X_k,Y_k)_{k\ge 0}$ is a Markov chain in the product state space 
$E^V\times F^V$, where $E,F$ are local state spaces and $V$ is a countably 
infinite set of sites (for concreteness, we fix $V=\mathbb{Z}^d$ 
throughout). Each element of $V$ should be viewed as a single dimension of 
the model.  A more practical interpretation is that $V$ defines a spatial 
degree of freedom and that $(X_k,Y_k)_{k\ge 0}$ describes the dynamics of 
a time-varying random field, as is the case in data assimilation 
applications.  In accordance with this interpretation, we will assume that 
the dynamics of the state $X_k$ and the observations $Y_k$ are local in 
nature: that is, the conditional distributions of the local state $X_k^v$ 
given the previous state $X_{k-1}$, and of the local observation $Y_k^v$ 
given the underlying process $X$, depend only on $X_{k-1}^w$ and $X_k^w$ 
for sites $w\in V$ that are neighbors of $v$.  In essence, our basic model 
therefore consists of an infinite family of local filtering models 
$(X_k^v,Y_k^v)_{k\ge 0}$ whose dynamics are locally coupled according to 
the graph structure of $V=\mathbb{Z}^d$.  The general model is described 
in detail in section \ref{sec:model}.

In section \ref{sec:phtrans} we investigate the natural 
infinite-dimensional version of Blackwell's Example \ref{ex:bla}. Recall 
that it was crucial in the finite-dimensional setting that the 
observations $Y_k=h(X_{k-1},X_k)$ are noiseless: the addition of any noise 
renders the observations nondegenerate and then ergodicity is preserved.  
This is no longer the case in infinite dimension: even if the local 
observations $Y_k^v$ are nondegenerate, the failure of the filter to 
inherit ergodicity can persist.  In fact, we observe a phase transition: 
the filter fails to be ergodic when the noise is small, but becomes 
ergodic when the noise strength exceeds a strictly positive threshold.  
The remarkable feature of this phenomenon is that no qualitative change of 
any kind occurs in the ergodic properties of the underlying model: 
$(X_k^v,Y_k^v)_{k\ge 0,v\in V}$ is a 1-dependent random field for every 
value of the noise parameter.  We are therefore in the surprising 
situation that complex ergodic behavior emerges in an otherwise trivial 
model when we consider its conditional distributions.  Such conditional 
phase transitions cannot arise in finite dimension.

The above example indicates that our intuition about inheritance of 
ergodicity, which fails in classical filtering models only in pathological 
cases, cannot be taken for granted in infinite dimension even under local 
nondegeneracy assumptions.  This raises the question as to whether there 
are situations in which the inheritance of ergodicity is guaranteed.  In 
view of the finite-dimensional theory, it is natural to conjecture that 
this might be the case under a symmetry breaking assumption. As will be 
discussed in section \ref{sec:obs}, the existing observability theory is 
not satisfactory in infinite dimension, and even in its simplest form the 
verification of this conjecture remains an open problem.  However, we will 
show in section \ref{sec:transinv} that the problem can be resolved in 
translation-invariant models for a (not entirely natural) modification of 
the usual filtering process using tools from multidimensional ergodic 
theory.  In section \ref{sec:conttime} the analogous problem will be 
considered in continuous time, where the method of section 
\ref{sec:transinv} can be extended to provide a full proof of the 
conjecture in the translation-invariant case.

In section \ref{sec:condrf} we turn our attention to the counterpart of 
the filter stability problem in the setting of Markov random fields.  
Such problems provide a simple setting for the investigation of decay of 
correlations in filtering problems, and are of interest in their own right 
as models that arise, for example, in image analysis.
Here the natural question of interest is whether the spatial mixing 
properties of random fields are inherited by conditioning on local 
observations. A direct adaptation of the example of section 
\ref{sec:phtrans} shows that the answer is negative in general.  However, 
we will show in section \ref{sec:mono} that the conditional random field 
inherits mixing properties of the underlying model when the latter 
possesses certain monotonicity properties.  This provides an entirely 
different mechanism for conditional ergodicity than in the 
observability theory of section \ref{sec:obs}.

As will rapidly become evident, the above results repeatedly exploit 
connections between filtering in infinite dimension, the statistical 
mechanics of disordered systems \cite{Bov06}, and multidimensional ergodic 
theory \cite{Con72,Geo11}.  Some further connections with other 
probabilistic problems will be discussed in the final section 
\ref{sec:disc} of this paper.

The consideration of filtering problems in infinite dimension is far from 
esoteric.  It is a long-standing problem in applied probability to develop 
efficient algorithms to compute nonlinear filters in high-dimensional 
models (current algorithms require a computational effort exponential in 
the dimension, a severe problem in data assimilation applications).  
Improved understanding of the ergodic and spatial mixing properties of the 
filter may be essential for progress on such practical problems.  In 
recent work \cite{RvH13, RvH13b}, we have shown that local filtering 
algorithms can attain dimension-free approximation errors in models that 
exhibit conditional decay of correlations.  The machinery developed in 
\cite{RvH13, RvH13b} to investigate high-dimensional filtering problems is 
complementary to the present paper: the former provides strong 
quantitative results under strong (`high-temperature') assumptions, while 
our aim here is to address fundamental issues that arise in infinite 
dimension.  Regardless of any practical implications, however, the 
investigation of emergent phenomena that arise from conditioning is 
arguably of fundamental interest to the understanding of conditional 
distributions, and provides a compelling motivation for the investigation 
of conditional phenomena in probability theory.

\section{Filtering in infinite dimension}
\label{sec:model}

The goal of this section is to set up the basic filtering problem that 
will be studied in the sequel.  We begin by defining a general 
setting for nonlinear filtering and introduce and discuss the basic 
ergodicity question in section \ref{sec:filt}.  We subsequently introduce 
our canonical infinite-dimensional filtering model in section 
\ref{sec:ihmm}.

\subsection{Nonlinear filtering and ergodicity}
\label{sec:filt}

Throughout this paper, we model dynamics with partial information as a 
Markov chain $(X_k,Y_k)_{k\ge 0}$ that has the additional property that 
its transition kernel factorizes as
$$
	\mathbf{P}[(X_k,Y_k)\in A|X_{k-1},Y_{k-1}] =
	\int \mathbf{1}_A(x,y)\,P(X_{k-1},dx)\,\Phi(X_{k-1},x,dy) 
$$ 
for given transition kernels $P$ and $\Phi$: the factorization corresponds 
to the assumption that $(X_k)_{k\ge 0}$ is a Markov chain in its own 
right, and that the observations $(Y_k)_{k\ge 0}$ are conditionally 
independent given $(X_k)_{k\ge 0}$.  Such models are frequently called 
\emph{hidden Markov models}.  While not essential for the development of 
our theory (see, e.g., \cite{TvH12} for a more general setting), the 
hidden Markov model setting is convenient mathematically and is ubiquitous 
in practice as a model of noisy observations of random dynamics.

For the time being, we assume that $X_k$ and $Y_k$ take values in an 
arbitrary Polish space (we will define a more concrete 
infinite-dimensional setting in section \ref{sec:ihmm} below).
The \emph{nonlinear filter} is defined as the regular conditional 
probability
$$
	\pi_k := \mathbf{P}[X_k\in\,\cdot\,|Y_1,\ldots,Y_k].
$$
We are interested in the question of whether $(\pi_k)_{k\ge 0}$ inherits 
the ergodic properties of the underlying dynamics $(X_k)_{k\ge 0}$.  
There are several different but closely connected ways to make this 
question precise (cf.\ Remark \ref{rem:altq} below).  For 
concreteness, we will focus attention on one particularly elementary 
formulation of this question that will serve as the guiding problem to be 
investigated throughout this paper.

We will assume in the sequel that the Markov chain $(X_k)_{k\ge 0}$ admits 
a unique invariant measure $\lambda$.  As 
$\mathbf{P}[X_k,Y_k\in\cdot\mbox{}| X_{k-1},Y_{k-1}]$ does not depend on 
$Y_{k-1}$ due to the hidden Markov structure, the invariant measure 
$\lambda$ extends uniquely to an invariant measure for the chain 
$(X_k,Y_k)_{k\ge 0}$, and we denote the unique stationary law of this 
process as $\mathbf{P}$.  By stationarity, we can assume in the 
sequel that $(X_k,Y_k)_{k\in\mathbb{Z}}$ is defined also for $k<0$.

The ergodic property of $(X_k)_{k\ge 0}$ that we 
will consider is \emph{stability} in the sense that
$$
	|\mathbf{P}[X_k\in A|X_0]-\lambda(A)| \xrightarrow{k\to\infty}0
	\quad\mbox{in }L^1
$$
for every measurable set $A$: that is, the law of $X_k$ `forgets' the 
initial condition $X_0$ as $k\to\infty$.  The analogous 
conditional property is \emph{filter stability} in the sense that
$$
	|\mathbf{P}[X_k\in A|X_0,Y_1,\ldots,Y_k]-
	\mathbf{P}[X_k\in A|Y_1,\ldots,Y_k]|\xrightarrow{k\to\infty}0
	\quad\mbox{in }L^1
$$
for every measurable set $A$: that is, the conditional distribution of 
$X_k$ given the observed data `forgets' the initial condition $X_0$ as 
$k\to\infty$.  It is natural to suppose that stability of the underlying 
dynamics will imply stability of the filter.  This conclusion is 
incorrect, however, as is illustrated by the following classical example 
\cite{Bla57}.

\begin{ex}
\label{ex:bla}
Let $(X_k)_{k\ge 0}$ be an i.i.d.\ sequence of random variables with
$\mathbf{P}[X_k=1]=\mathbf{P}[X_k=-1]=1/2$, and let $Y_k=X_kX_{k-1}$
for $k\ge 1$.  This evidently defines a stationary hidden Markov model 
with $P(x,\cdot)=(\delta_1+\delta_{-1})/2$ and $\Phi(x',x,\cdot)=
\delta_{xx'}$.  Note that
$$
	X_k = X_0Y_1Y_2\cdots Y_k.
$$
We can therefore easily compute for every $k\ge 0$
\begin{align*}
	&\mathbf{P}[X_k=1|X_0,Y_1,\ldots,Y_k] = \mathbf{1}_{X_k=1},
	\\	
	&\mathbf{P}[X_k=1|Y_1,\ldots,Y_k] = 1/2.
\end{align*}
Thus the filter is certainly not stable.  On the other hand, underlying 
dynamics $(X_k)_{k\ge 0}$ is an i.i.d.\ sequence, and is therefore stable 
in the strongest possible sense:
$$
	\mathbf{P}[X_k\in A|X_0]=\lambda(A)\quad\mbox{for all }k\ge 1.
$$
Moreover, even the process $(X_k,Y_k)_{k\ge 0}$ is stable in the 
strongest possible sense: it is a 1-dependent sequence, so that
$\mathbf{P}[(X_k,Y_k)\in A|X_0,Y_0] = \mathbf{P}[(X_k,Y_k)\in A]$
for all $k\ge 2$.
\end{ex}

Example \ref{ex:bla} shows that the inheritance of ergodicity under 
conditioning cannot be taken for granted.  Nonetheless, the phenomenon 
exhibited here is very fragile: if the observations are perturbed by any 
noise (for example, if we set $Y_k=X_kX_{k-1}\xi_k$ with 
$\mathbf{P}[\xi_k=-1]=1-\mathbf{P}[\xi_k=1]=p$ and any $0<p<1$), the 
filter will become stable.  The inheritance of ergodicity is therefore 
apparently obstructed by the singularity of the observation kernel $\Phi$.  
To rule out such singular behavior, it is natural to require that the 
observation kernel $\Phi$ possesses a positive density with respect to 
some reference measure $\varphi$.  A model with this property is said to 
possess \emph{nondegenerate observations}.  One might now expect that 
nondegeneracy of the observations removes the obstruction to inheritance 
of ergodicity observed in Example \ref{ex:bla}. Unfortunately, this is 
still not the case in complete generality, as is demonstrated by an 
esoteric counterexample in \cite{vH12}.  However, the conclusion does hold 
if we use a stronger \emph{uniform} notion of stability.

\begin{thm}[\cite{vH09c}]
\label{thm:vH09c}
Suppose that the following hold.
\begin{enumerate}
\item The underlying dynamics is \emph{uniformly stable} in the sense
$$
	\sup_A
	|\mathbf{P}[X_k\in A|X_0]-\lambda(A)| \xrightarrow{k\to\infty}0
	\quad\mbox{in }L^1.
$$
\item The observations are \emph{nondegenerate} in the sense
$$
	\Phi(x',x,dy) = g(x',x,y)\,\varphi(dy),\qquad
	g(x',x,y)>0\mbox{ for all }x,x',y.
$$
\end{enumerate}
Then the filter is \emph{uniformly stable} in the sense
$$
	\sup_A
	|\mathbf{P}[X_k\in A|X_0,Y_1,\ldots,Y_k]-
	\mathbf{P}[X_k\in A|Y_1,\ldots,Y_k]|\xrightarrow{k\to\infty}0
	\quad\mbox{in }L^1.
$$
\end{thm}

This result, together with the mathematical theory behind its proof (cf.\ 
section \ref{sec:meas}), provides a very general qualitative understanding 
of the inheritance of ergodicity in classical filtering models.  However, 
as will be explained below, this theory breaks down completely in 
infinite-dimensional models.  In the remainder of this paper, we will see 
that new phenomena arise in the infinite-dimensional setting.

\begin{rem}
\label{rem:altq}
The question of inheritance of ergodic properties under conditioning can 
be formulated in a number of different ways.  For concreteness, we focus 
our attention in this paper on the elementary formulation introduced 
above.  As the choice of problem is somewhat arbitrary, let us briefly 
describe a number of alternative formulations.

In the setting of stability of the filter, we have considered `forgetting' 
of the initial condition $X_0$ under the stationary measure.  Similar 
problems can be formulated, however, in a more general setting. 
Denote by $\mathbf{P}^\mu$ the law of the process
$(X_k,Y_{k})_{k\ge 0}$ with the initial distribution $X_0\sim\mu$.
A natural notion of stability is to require that 
$$
	\mathbf{P}^\mu[X_k\in\cdot\mbox{}]\xrightarrow{k\to\infty}\lambda
	\quad\mbox{for every }\mu
$$
in a suitable topology on probability measures.
If we define the filter started at $\mu$ as
$\pi_k^\mu := \mathbf{P}^\mu[X_k\in\cdot\mbox{}|Y_1,\ldots,Y_k]$,
we can now investigate the general filter stability problem
$$
	|\pi_k^\mu(f)-\pi_k^\nu(f)|\xrightarrow{k\to\infty}0
	\quad\mbox{in }L^1(\mathbf{P}^\gamma)
$$
for a suitable class of measures $\mu,\nu,\gamma$ and functions $f$.  The 
formulation that we consider in this paper corresponds to the special case 
$\nu=\lambda$ and $\mu=\gamma=\delta_x$ for $x$ outside a $\lambda$-null 
set.  Nonetheless, our formulation proves to be equivalent in a rather 
general setting to stability for general initial measures
$\mu,\nu,\gamma$, cf.\ \cite[Chapter 12]{CR11} and \cite{vH09c,TvH13}.

A different and perhaps more natural formulation dates back to Blackwell 
\cite{Bla57} and Kunita \cite{Kun71}.  Using the Markov property of the 
underlying model, it is not difficult to show that the measure-valued 
stochastic process $(\pi_k)_{k\ge 0}$ is itself a Markov chain, cf.\ 
\cite[Appendix A]{vH12}.  One can now ask whether the ergodic properties 
of the Markov chain $(X_k)_{k\ge 0}$ `lift' to ergodic properties of the 
Markov chain $(\pi_k)_{k\ge 0}$.  For example, if $(X_k)_{k\ge 0}$ admits 
a unique stationary measure, does $(\pi_k)_{k\ge 0}$ admit a unique 
stationary measure also?  Similarly, if $(X_k)_{k\ge 0}$ converges to its 
stationary measure starting from any initial condition, does the same 
property hold for $(\pi_k)_{k\ge 0}$?  Remarkably, while these questions 
appear in first instance to be quite distinct from the question of filter 
stability, such properties again prove to be equivalent in a very general 
setting to the notion of filter stability that we consider in this paper, 
cf.\ \cite{Kun71,Ste89,CvH10,CR11,vH12}.

A third formulation of inheritance of ergodicity under conditioning is 
obtained when we consider, rather than the filter, the conditional 
distribution of the entire process $X=(X_k)_{k\in\mathbb{Z}}$ given the 
infinite observation sequence $Y=(Y_k)_{k\in\mathbb{Z}}$.  Using the 
Markov property of the underlying model, it is not difficult to establish 
that $X$ is still a Markov process under the conditional distribution 
$\mathbf{P}[\,\cdot\,|Y]$, albeit time-inhomogeneous and with transition 
probabilities that depend on the realized observation sequence $Y$: that 
is, the conditional process is a Markov chain in a random environment.  
One can now ask whether the process $X$ inherits its ergodic properties 
under $\mathbf{P}$ when it is considered under the conditional 
distribution $\mathbf{P}[\,\cdot\,|Y]$.  Once again, this apparently 
distinct formulation proves to be equivalent in a general setting the 
formulation considered in this paper, a fact that is exploited heavily in 
the theory of \cite{vH09c,TvH13}.

It is now well understood that the properties described above 
are equivalent in classical filtering models.  While some of these 
arguments extend directly to the infinite-dimensional setting, others do 
not, and it remains to be investigated to what extent these 
equivalences remain valid in infinite dimension.  Nonetheless, the problem 
formulation considered here is arguably the most elementary one, and 
provides a natural starting point for the investigation of conditional 
phenomena in infinite dimension.
\end{rem}

\begin{rem} 
Even when the underlying dynamics $(X_k)_{k\ge 0}$ is not stable, it may 
be the case that the filter is stable.  For example, using the trivial 
observation model $Y_k=X_k$, the filter is stable regardless of any 
properties of the underlying model.  More generally, the filter is 
expected to be stable when the observations are `sufficiently 
informative,' which is made precise in \cite{vH08, vH09, vH09b} in terms 
of nonlinear notions of \emph{observability}.  Such results are in some 
sense the opposite of Theorem \ref{thm:vH09c}: the latter shows that 
ergodicity is inherited by the filter, while the former show that the 
filter can be ergodic regardless of ergodicity of the underlying model 
(even without nondegeneracy).  None of these results prove to be 
satisfactory in infinite dimension: it appears that a general theory 
for ergodicity of the filter will require both ergodicity of the 
underlying model and some form of observability, as will become evident
in the following sections.
\end{rem}

\subsection{The infinite-dimensional model}
\label{sec:ihmm}

The aim of this paper is to investigate conditional phenomena that arise 
in infinite dimension.  So far, no assumptions have been made on the model 
dimension: we have set up our theory in any Polish state space.  
Nonetheless, while no explicit dimensionality requirements appear, for 
example, in Theorem \ref{thm:vH09c}, the assumptions of previous results 
can typically hold only in finite-dimensional situations.  To understand 
the problems that arise in infinite dimension, and to provide a concrete 
setting for the investigation of conditional phenomena in infinite 
dimension, we presently introduce a canonical infinite-dimensional 
filtering model that will be used in the sequel.

The practical interest in infinite-dimensional filtering models stems from 
problems that have spatial in addition to dynamical structure. To model 
this situation, let us assume for concreteness that the spatial degrees of 
freedom are indexed by the infinite lattice $\mathbb{Z}^d$.  We also 
define Polish spaces $E$ and $F$ that describe the state of the model at 
each spatial location.  We now assume that $X_k$ and $Y_k$ are random 
fields that are indexed by $\mathbb{Z}^d$ and take values locally in $E$ 
and $F$, respectively, for every time $k$: that is,
$$
	X_k=(X_k^v)_{v\in\mathbb{Z}^d}
	\in E^{\mathbb{Z}^d}
	\qquad\mbox{and}\qquad
	Y_k=(Y_k^v)_{v\in\mathbb{Z}^d}\in 
	F^{\mathbb{Z}^d}.
$$
Each $v\in\mathbb{Z}^d$ should be viewed as a single `dimension' of the 
model.\footnote{
	The present setting is easily extended to the setting of more 
	general locally finite graphs and to the setting where each 
	location $v$ may possess a different local state space $E^v$.  
	Such an extension does not illuminate significantly the phenomena 
	that will be investigated in the sequel.  On the other hand, 
	a nontrivial extension of substantial interest in
        applications is to
	continuous infinite-dimensional models such as stochastic partial 
	differential equations, cf.\ \cite{Stu10}.  While we do not 
	investigate such models here, the present setting may be viewed as 
	prototypical for more complicated models of this type, see
	\cite{RvH13} for further discussion.
}
We now define a hidden Markov model that respects the spatial structure of 
the problem by assuming that both the underlying dynamics and the 
observations are \emph{local}: that is, we assume that the transition and 
observation kernels $P$ and $\Phi$ factorize as
$$
	P(x,dz) = \prod_{v\in\mathbb{Z}^d} P^v(x,dz^v),\qquad
	\Phi(x,z,dy) = \prod_{v\in\mathbb{Z}^d}\Phi^v(x,z,dy^v),
$$
where 
$$
	P^v(x,A)\quad\mbox{and}\quad
	\Phi^v(x,z,B)\quad
	\mbox{depend only on }x^w,z^w\mbox{ for }\|w-v\|\le 1.
$$
Such a model should be viewed as a hidden Markov model counterpart of 
probabilistic cellular automata \cite{LMS90} or interacting particle 
systems \cite{Lig05} that have been widely investigated in the literature
as natural models of space-time dynamics.
Alternatively, one might view such a model as an infinite collection 
$(X_k^v,Y_k^v)_{k\ge 0}$ of hidden Markov models whose dynamics and 
observations are locally coupled to their neighbors in $\mathbb{Z}^d$.

While problems of this type have been rarely considered in filtering 
theory, the infinite-dimensional model that we have formulated is in 
principle a special case of the general model described in the previous 
section.  However, its structure is such that the assumptions of a result 
such as Theorem \ref{thm:vH09c} typically cannot hold.  Let us consider, 
for example, the setting where each local observation $Y^v$ has a positive 
density of the form $\Phi^v(x,z,dy^v) = g(z^v,y^v)\,\varphi(dy^v)$, so 
that the observations are \emph{locally nondegenerate}.  Choose two values 
$e,e'\in E$ such that $g(e,\cdot)\ne g(e',\cdot)$, and define the constant 
configurations $z,z'$ as $z^v=e$ and $z^{\prime v}=e'$ for all 
$v\in\mathbb{Z}^d$. Then the measures $\Phi(x,z,\cdot)$ and 
$\Phi(x,z',\cdot)$ are two distinct laws of an infinite number of i.i.d.\ 
random variables, and are therefore mutually singular.  This immediately 
rules out the possibility that the observations are nondegenerate in the 
sense of Theorem \ref{thm:vH09c}. It is precisely this problem that lies 
at the heart of the difficulties in infinite-dimensional models: 
probability measures in infinite dimension are typically mutually 
singular, even when they admit densities locally (that is, for any 
finite-dimensional marginal).  In the absence of densities, classical 
results in filtering theory cannot be taken for granted, and the study of 
filtering in infinite dimension gives rise to fundamentally different 
problems than have been studied in the literature to date.  We initiate 
the investigation of such problems in the sequel.

\begin{rem}
The singularity of measures in infinite dimension is problematic not 
only for the nondegeneracy of observations, but also for the ergodic 
theory of Markov chains.  For example, the uniform stability property
in Theorem \ref{thm:vH09c} will rarely hold 
in infinite dimension: it is often the case that the law of $X_k$ is 
singular with respect to $\lambda$ for all $k<\infty$, which rules 
out total variation convergence (see \cite[Example 2.3]{TvH13} for a 
simple illustration).  However, this issue is surmounted in \cite{TvH13}
using a form of localization: by performing the analysis of Theorem 
\ref{thm:vH09c} locally (that is, to finite-dimensional projections of the 
original model), we can avoid the singularity of the full 
infinite-dimensional problem.  This allows to extend the conclusion of 
Theorem \ref{thm:vH09c} to a wide range of infinite-dimensional models 
with nondegenerate observations.  In practice, this implies that much of 
the classical filtering theory extends, at least in spirit, to models 
where $X_k$ is infinite-dimensional but $Y_k$ is (effectively) 
finite-dimensional.  It is only when the observations $Y_k$ are also
infinite-dimensional that new phenomena arise.
\end{rem}

\begin{rem}
Let us note that we have used the term `infinite-dimensional' to denote 
the situation where there are infinitely many independent degrees of 
freedom, which is the key issue in our setting.  The problem of dimension 
is unrelated to the linear algebraic or metric dimension of the state 
space: indeed, even each of the local state spaces $E$ and $F$ in our 
model can itself be an arbitrary Polish space.  Conversely, it is possible 
to have infinite-dimensional systems that are `effectively 
finite-dimensional' in the sense that only finitely many degrees of 
freedom carry significant information.  This is common, for example, in 
stochastic partial differential equations (see, e.g., \cite{TvH13}).

At the same time, it should be noted that even in finite-dimensional 
systems where results such as Theorem \ref{thm:vH09c} technically apply, 
the qualitative information contained in such statements may be misleading 
from the practical point of view: in finite but high-dimensional systems, 
phenomena that arise qualitatively in infinite dimension are still 
manifested in a quantitative fashion (see \cite{RvH13} for quantitative 
results and discussion on filtering in high dimension).  For example, 
if the filter is not stable for the infinite-dimensional model, it 
will often still be the case that the filter is stable for every 
finite-dimensional truncation of the model; however, the quantitative rate 
of stability will vanish rapidly as the dimension is increased.  
Conversely, if the filter is stable for the infinite-dimensional model, 
then the rate of stability of the filter for the finite-dimensional models 
will be dimension-free.  As it is ultimately the quantitative behavior of 
filtering algorithms that is of importance in practice,
the qualitative phenomena investigated here in infinite dimension can 
still provide more insight into the behavior of practical filtering 
problems in high dimension than classical results in filtering theory. 
\end{rem}

\section{A conditional phase transition}
\label{sec:phtrans}

\subsection{Model and results}

The goal of this section is to develop a simple example of the general 
infinite-dimensional setting of section \ref{sec:ihmm} where we observe 
nontrivial behavior of the inheritance of ergodicity.  This model, to be 
described presently, is a natural infinite-dimensional variation on 
Blackwell's counterexample (Example \ref{ex:bla} above).

Throughout this section, 
$$
	X_k=(X_k^v)_{v\in\mathbb{Z}}\in \{-1,1\}^\mathbb{Z}
	\qquad\mbox{and}\qquad
	Y_k=(\bar Y_k^v,\hat Y_k^v)_{v\in\mathbb{Z}}\in 
	(\{-1,1\}\times\{-1,1\})^\mathbb{Z}
$$
are binary random fields in one spatial dimension.  We let
$$
	(X_k^v)_{k,v\in\mathbb{Z}}\text{ are i.i.d.\ with }
	\mathbf{P}[X_k^v=1]=1/2,
$$
and we let
$$
	\bar Y_k^v = X_k^vX_{k-1}^v\bar\xi_k^v,\qquad\quad
	\hat Y_k^v = X_k^vX_k^{v+1}\hat\xi_k^v,	
$$
where
$$
	(\bar\xi_k^v)_{k,v\in\mathbb{Z}},~(\hat\xi_k^v)_{k,v\in\mathbb{Z}}
	\text{ are i.i.d.\ with }
	\mathbf{P}[\bar\xi_k^v=-1]=p
$$
and $(\bar\xi_k^v)_{k,v\in\mathbb{Z}},~(\hat\xi_k^v)_{k,v\in\mathbb{Z}}$
are independent of $(X_k^v)_{k,v\in\mathbb{Z}}$.

This evidently corresponds to a model of the form discussed in section 
\ref{sec:ihmm}. In words, the underlying dynamics is of the simplest 
possible type: each time and each spatial location is an independent 
random variable.  When $p=0$, the observations reveal for each site 
whether its current state differs from its state at the previous time and 
from the states of its two neighbors at the present time.  When $p>0$, 
each observation is subject to additional noise that inverts the outcome 
with probability $p$.  By symmetry, it will suffice to consider the case 
$p\le 1/2$, which we will do from now on.

The model that we have constructed is evidently a direct extension of 
Example \ref{ex:bla} to infinite dimension.  As in Example \ref{ex:bla}, 
the process $(X_k,Y_k)_{k\in\mathbb{Z}}$ is ergodic in the strongest 
sense, so that even the uniform stability assumption of Theorem 
\ref{thm:vH09c} is satisfied.  When $p=0$, it is easily seen by the same 
reasoning as in Example \ref{ex:bla} that the filter is not stable.  
However, in Example \ref{ex:bla} the addition of observation noise with 
error probability $p>0$ would yield nondegenerate observations, and thus 
filter stability by Theorem \ref{thm:vH09c}.  In the present setting, on 
the other hand, nondegeneracy fails for any $p$. Nonetheless, the 
observations are \emph{locally nondegenerate} when $p>0$, and one might 
conjecture that this suffices to ensure inheritance of ergodicity.  This 
is not the case.

\begin{thm}
\label{thm:phtrans}
For the model of this section, there exist constants $0<p_\star\le 
p^\star<1/2$ such that the filter is stable for $p^\star<p\le 1/2$ and is 
not stable for $0\le p<p_\star$.
\end{thm}

\begin{rem}
We naturally believe that one can choose $p_\star=p^\star$ in Theorem 
\ref{thm:phtrans}, but we did not succeed in proving that.
The proof yields some explicit bounds on $p_\star$ and $p^\star$.
\end{rem}

Theorem \ref{thm:phtrans} shows that local nondegeneracy does not suffice 
to ensure inheritance of ergodicity in infinite dimension: ergodicity of 
the filter undergoes a \emph{phase transition} at a strictly positive 
signal to noise ratio of the observations.  Remarkably, the underlying 
model does not seem to exhibit any qualitative change in behavior:
$(X_k^v,Y_k^v)_{k,v\in\mathbb{Z}}$ is a one-dependent random field for 
every value of the error probability $p$.  Thus it is evidently possible 
in infinite dimension that complex ergodic behavior emerges in an 
otherwise trivial model when we consider its conditional distributions.

The remainder of this section is devoted to the proof of Theorem 
\ref{thm:phtrans}.  The proof relies on standard tools from statistical 
mechanics \cite{Bov06,Geo11}: a Peierls argument for the low noise regime 
and a Dobrushin contraction method for the high noise regime.

\subsection{Proof of Theorem \ref{thm:phtrans}: low noise}
\label{sec:lownoise}

We begin by noting that as $(X_k^v,Y_k^v)_{k,v\in\mathbb{Z}}$ and 
$(-X_k^v,Y_k^v)_{k,v\in\mathbb{Z}}$ have the same law, it follows that
$\mathbf{E}[X_k^0|Y_1,\ldots,Y_k]=\mathbf{E}[-X_k^0|Y_1,\ldots,Y_k]$,
and we therefore have
$$
	\mathbf{E}[X_k^0|Y_1,\ldots,Y_k]=0\quad\mbox{for all }k\ge 1.
$$
To prove that the filter is not stable, it therefore suffices to show that
$$
	\inf_{k\ge 1}
	\mathbf{E}|\mathbf{E}[X_k^0|X_0,Y_1,\ldots,Y_k]|>0.
$$
To show this, we begin by reducing the problem to finite dimension.

\begin{lem}
\label{lem:findimred}
Suppose that $0<p\le 1/2$.  Then
$$
	\mathbf{E}[X_k^0|X_0,Y_1,\ldots,Y_k,
	\{X_1^v,\ldots,X_k^v:|v|>m\}]
	\xrightarrow{m\to\infty}
	\mathbf{E}[X_k^0|X_0,Y_1,\ldots,Y_k]\quad\mbox{a.s.}
$$
\end{lem}

\begin{proof}
Let $\beta := \log\sqrt{(1-p)/p}>0$. We begin by noting that
$$
	\mathbf{P}[\hat Y_\ell^v=y|X_0,\ldots,X_k] =
	\sqrt{p(1-p)}\, e^{\beta yX_\ell^vX_\ell^{v+1}}
$$
for $1\le\ell\le k$ and $y\in\{-1,1\}$.
Define the probability measure $\mathbf{Q}$ such that
$$
	\mathbf{P}(A) = \mathbf{E_Q}\bigg[
	\mathbf{1}_A
	\prod_{\ell=1}^k 4p(1-p)\, 
	e^{\beta \hat Y_\ell^m X_\ell^mX_\ell^{m+1}}\,
	e^{\beta \hat Y_{\ell}^{-m-1}X_\ell^{-m-1}X_\ell^{-m}}
	\bigg].
$$
Then under $\mathbf{Q}$, the observations $\hat Y_\ell^m$ and
$\hat Y_\ell^{-m-1}$, $1\le\ell\le k$ are symmetric Bernoulli and
independent from all the remaining variables in the model, while the 
remainder of the model is the same as defined above.  In particular, this
implies that
$$
	\{X_0^v,X_\ell^v,Y_\ell^v:1\le \ell\le k,|v|>m\}
	\independent
	\{X_0^v,X_\ell^v,Y_\ell^v:1\le \ell\le k,|v|\le m\}
	\quad\mbox{under }\mathbf{Q}.
$$
We therefore obtain using the Bayes formula
\begin{align*}
	&\mathbf{P}[A|X_0,Y_1,\ldots,Y_k,
	\{X_1^v,\ldots,X_k^v:|v|>m\}] = \\
	&\quad\frac{
	\mathbf{E_Q}[\mathbf{1}_A
	\frac{d\mathbf{P}}{d\mathbf{Q}}|X_0,Y_1,\ldots,Y_k,
	\{X_1^v,\ldots,X_k^v:|v|>m\}]
	}{
	\mathbf{E_Q}[\frac{d\mathbf{P}}{d\mathbf{Q}}|X_0,Y_1,\ldots,Y_k,
	\{X_1^v,\ldots,X_k^v:|v|>m\}]} \\
	&
	\quad\quad\ge e^{-4\beta k}\,
	\mathbf{Q}[A|X_0,Y_1,\ldots,Y_k,
        \{X_1^v,\ldots,X_k^v:|v|>m\}]
	\\
	&
	\quad\quad= e^{-4\beta k}\,
	\mathbf{Q}[A|X_0,Y_1,\ldots,Y_k]
\end{align*}
for any $A\in\sigma\{X_0,Y_1,\ldots,Y_k,X_1^v,\ldots,X_k^v:|v|\le m\}$.

Define $Z^0:=(X_1^0,\ldots,X_k^0)$ and
$Z^{-m}:=(X_1^m,\ldots,X_k^m,X_1^{-m},\ldots,X_k^{-m})$ for $m\ge 1$.
Due to the conditional independence
structure of the infinite-dimensional filtering model,
$$
	\mathbf{E}[f(Z^{-m})|X_0,Y_1,\ldots,Y_k,
	Z^{-m-1},Z^{-m-2},\ldots] =
	\mathbf{E}[f(Z^{-m})|X_0,Y_1,\ldots,Y_k,Z^{-m-1}]
$$
for every $m\ge 0$.  Thus $(Z^m)_{m\le 0}$ is a 
Markov chain under any regular version of the conditional distribution 
$\mathbf{P}[\,\cdot\,|X_0,Y_1,\ldots,Y_k]$ (almost surely with respect to 
the realization of $X_0,Y_1,\ldots,Y_k$).  Moreover, the above estimate 
shows that the (random) transition kernels of this Markov chain satisfy
the Doeblin condition \cite[Theorem 16.2.4]{MT93}, so
$$
	|\mathbf{E}[X_k^0|X_0,Y_1,\ldots,Y_k,
	\{X_1^v,\ldots,X_k^v:|v|>m\}]-
	\mathbf{E}[X_k^0|X_0,Y_1,\ldots,Y_k]| \le
	2(1-e^{-4\beta k})^{m+1}
$$
for all $m\ge 0$.  This completes the proof.
\end{proof}

Lemma \ref{lem:findimred} reduces our problem to a finite-dimensional one. 
Indeed, it is clear that the filter is not stable for $p=0$ (for precisely 
the same reason as in Example \ref{ex:bla}), so we will assume 
without loss of generality in the sequel that $0<p\le 1/2$.  
Applying Lemma \ref{lem:findimred}, it follows that in order to prove that 
the filter is not stable, it suffices to show that
$$
	\inf_{k,m\ge 1}\mathbf{E}|\mathbf{E}[X_k^0|X_0,Y_1,\ldots,Y_k,
	\{X_1^v,\ldots,X_k^v:|v|>m\}]|>0.
$$
But the conditional independence structure of the 
infinite-dimensional filtering model implies that the conditional 
expectation inside 
this expression depends only on $X_\ell^v$ and $Y_\ell^v$ for $0\le\ell\le 
k$ and $|v|\le m+1$.  We are thus faced with the problem of obtaining
a lower bound on this finite-dimensional quantity that is uniform in 
$k,m$.

To lighten the notation, it will be convenient to view 
$(X_k^v)_{k,v\in\mathbb{Z}}$ not as a sequence of spatial random fields on 
$\mathbb{Z}$, but rather as a single space-time random field on 
$\mathbb{Z}^2$.  To this end, we will write $X^q:=X_k^v$ for 
$q=(k,v)\in\mathbb{Z}^2$.  We will similarly write $Y^{qr}:=\bar Y_k^v$ 
and $\xi^{qr}:=\bar\xi_k^v$ if $q=(k-1,v)$ and $r=(k,v)$, and 
$Y^{qr}:=\hat Y_k^v$ and $\xi^{qr}:=\hat\xi_k^v$ if $q=(k,v)$ and 
$r=(k,v+1)$ (the order of the indices $q,r$ is irrelevant, that is, 
$Y^{qr}:=Y^{rq}$ etc.)  In this manner, we can view 
$X=(X^q)_{q\in\mathbb{Z}^2}$ as a random field on the lattice 
$\mathbb{Z}^2$, with observations $Y^{qr}$ attached to each edge 
$\{q,r\}\subset\mathbb{Z}^2$ with $\|q-r\|=1$.

\begin{lem}
\label{lem:rbim}
Suppose that $0<p\le 1/2$, and let $k,m\ge 1$. 
Define $\beta := \log\sqrt{(1-p)/p}$,
$J:=[1,k]\times[-m,m]$, and
$\partial J:=\{0\}\times[-m,m]\cup[1,k]\times\{-m-1,m+1\}$.
For any given configuration $x\in\{-1,1\}^{\mathbb{Z}^2}$, we define the
random measure $\Sigma$ on $\{-1,1\}^J$ as
$$
	\Sigma^x(\{z\}) := \frac{1}{Z}
	\exp\Bigg(\beta\,\Bigg\{
	\sum_{\{q,r\}\subseteq J:\|q-r\|=1}
	\xi^{qr}x^qx^rz^qz^r
	+
	\sum_{q\in J,r\in\partial J:\|q-r\|=1}
	\xi^{qr}x^qz^q
	\Bigg\}\Bigg),
$$
where $Z$ is the normalization such that $\Sigma^x(J)=1$.
Then
$$
	\mathbf{E}[(X^q)_{q\in J}\in A|X_0,Y_1,\ldots,Y_k,
        \{X_1^v,\ldots,X_k^v:|v|>m\}] =
	\Sigma^X(A).
$$
\end{lem}

\begin{proof}
By the conditional independence
structure of the filtering model, we have
\begin{align*}
	&\mathbf{E}[(X^q)_{q\in J}\in A|X_0,Y_1,\ldots,Y_k,
        \{X_1^v,\ldots,X_k^v:|v|>m\}] =\\
	&\qquad\mathbf{E}[(X^q)_{q\in J}\in A|
	(X^q)_{q\in\partial J},
	(Y^{qr})_{q\in J,
	r\in J\cup\partial J,\|q-r\|=1}].
\end{align*}
The joint distribution of the random variables that appear
in this expression is
\begin{multline*}
	\mathbf{P}[(X^q)_{q\in J\cup\partial J}=z,
	(Y^{qr})_{q\in J,
        r\in J\cup\partial J,\|q-r\|=1}=y] = 
	2^{-|J\cup\partial J|}\times\mbox{}\\
	\prod_{\{q,r\}\subseteq J:\|q-r\|=1}
	\sqrt{p(1-p)}\, e^{\beta y^{qr}z^qz^r}
	\prod_{q\in J,
        r\in \partial J:\|q-r\|=1}
	\sqrt{p(1-p)}\, e^{\beta y^{qr}z^qz^r},
\end{multline*}
where $|A|$ denotes the cardinality of a set $A$.  The result
now follows readily from the Bayes formula and the fact that
$Y^{qr}=X^qX^r\xi^{qr}$ by construction.
\end{proof}

Lemma \ref{lem:rbim} shows that the conditional distribution 
$\mathbf{P}[\,\cdot\,|X_0,Y_1,\ldots,Y_k, \{X_1^v,\ldots,X_k^v:|v|>m\}]$ 
has a familiar form in statistical mechanics: it is (up to the change of 
variables or \emph{gauge transformation} $\sigma^q = x^qz^q$) an Ising 
model with random interactions, also known as a \emph{random bond Ising 
model} or an \emph{Ising spin glass}, with inverse temperature
$\beta=\log\sqrt{(1-p)/p}$.  The failure of stability of the filter for 
large $\beta$ can now be addressed using a standard method in statistical 
mechanics \cite[section 6.4]{Bov06}.  For concreteness, we include the 
requisite arguments in the present setting, which completes the proof.

\begin{rem}
For the benefit of readers who are not familiar with ideas from 
statistical mechanics, we provide a brief description of the 
Peierls argument that is used below.

In Lemma \ref{lem:rbim}, one should interpret $x$ as the true 
configuration of the underlying system, and $z$ as the configuration in 
the set $J$ that is reconstructed from the noisy observations.  The random 
variable $\Sigma^x(\{z\})$ (which depends only on the realization of the 
observation noise once $x,z$ are fixed) determines the probability that we 
reconstruct $z$ when the true configuration is $x$.  As we have 
conditioned on the true value of the signal at $X_0,X_1^v,\ldots,X_k^v$ 
for $|v|>m$, we can assume that the true configuration is correctly 
reconstructed on the boundary $z^r=x^r$ for $r\in\partial J$.  We want to 
show that the probability of correctly reconstructing the true 
configuration at a given site inside $J$ (e.g., $X_k^0$ which 
corresponds to the point $\mathbf{0}:=(k,0)\in J$) remains bounded away 
from $\frac{1}{2}$ with high probability, no matter how far this site is
from the boundary $\partial J$.

Now suppose that $z^\mathbf{0}\ne x^\mathbf{0}$, that is, that the 
configuration at $\mathbf{0}$ is incorrectly decoded.  Then there must 
exist an interface between the set of incorrectly decoded vertices (that 
contains $\mathbf{0}$) and the set of correctly decoded vertices (that 
contains the boundary $\partial J$).  This is illustrated in the following 
figure, where correctly decoded vertices are indicated by white circles 
and incorrectly decoded vertices are indicated by black circles:
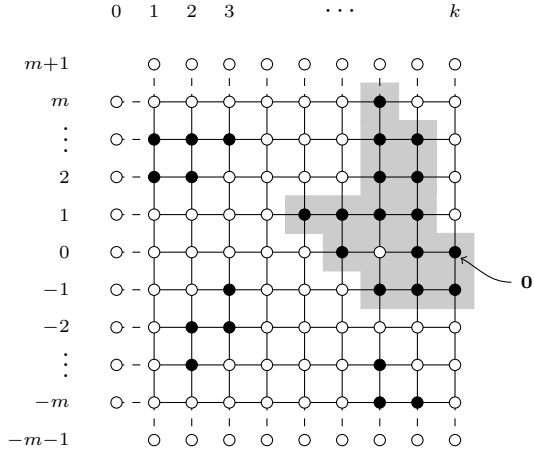
\begin{center}
\begin{tikzpicture}

\filldraw[color=black!20] (4.75,-.75) rectangle (3.25,.25);
\filldraw[color=black!20] (3.25,.25) rectangle (2.75,-.25);
\filldraw[color=black!20] (2.25,.25) rectangle (4.25,.75);
\filldraw[color=black!20] (3.25,1.75) rectangle (4.25,.75);
\filldraw[color=black!20] (3.25,1.75) rectangle (3.75,2.25);

\draw (-.5,0) node[left] {$\scriptstyle 0$};
\draw (-.5,.5) node[left] {$\scriptstyle 1$};
\draw (-.5,1) node[left] {$\scriptstyle 2$};
\draw (-.5,1.6) node[left] {$\vdots$};
\draw (-.5,2) node[left] {$\scriptstyle m$};
\draw (-.5,2.5) node[left] {$\scriptstyle m+1$};
\draw (-.5,-.5) node[left] {$\scriptstyle -1$};
\draw (-.5,-1) node[left] {$\scriptstyle -2$};
\draw (-.5,-1.4) node[left] {$\vdots$};
\draw (-.5,-2) node[left] {$\scriptstyle -m$};
\draw (-.5,-2.5) node[left] {$\scriptstyle -m-1$};

\draw(0,3) node[above] {$\scriptstyle 0$};
\draw(.5,3) node[above] {$\scriptstyle 1$};
\draw(1,3) node[above] {$\scriptstyle 2$};
\draw(1.5,3) node[above] {$\scriptstyle 3$};
\draw(3,3) node[above] {$\cdots$};
\draw(4.5,3) node[above] {$\scriptstyle k$};

\draw (.5,0) to (4.5,0);
\draw[<-] (4.575,-.075) to[out=-30,in=180] (5.25,-.4) node[right] {$\scriptstyle\mathbf{0}$};
\draw (.5,.5) to (4.5,.5);
\draw (.5,1) to (4.5,1);
\draw (.5,1.5) to (4.5,1.5);
\draw (.5,2) to (4.5,2);
\draw (.5,-.5) to (4.5,-.5);
\draw (.5,-1) to (4.5,-1);
\draw (.5,-1.5) to (4.5,-1.5);
\draw (.5,-2) to (4.5,-2);

\draw[dashed] (0,0) to (.5,0);
\draw[dashed] (0,.5) to (.5,.5);
\draw[dashed] (0,1) to (.5,1);
\draw[dashed] (0,1.5) to (.5,1.5);
\draw[dashed] (0,2) to (.5,2);
\draw[dashed] (0,-.5) to (.5,-.5);
\draw[dashed] (0,-1) to (.5,-1);
\draw[dashed] (0,-1.5) to (.5,-1.5);
\draw[dashed] (0,-2) to (.5,-2);

\filldraw[draw=black,fill=white] (0,0) circle (.075);
\filldraw[draw=black,fill=white] (0,.5) circle (.075);
\filldraw[draw=black,fill=white] (0,1) circle (.075);
\filldraw[draw=black,fill=white] (0,1.5) circle (.075);
\filldraw[draw=black,fill=white] (0,2) circle (.075);
\filldraw[draw=black,fill=white] (0,-.5) circle (.075);
\filldraw[draw=black,fill=white] (0,-1) circle (.075);
\filldraw[draw=black,fill=white] (0,-1.5) circle (.075);
\filldraw[draw=black,fill=white] (0,-2) circle (.075);

\draw (.5,2) to (.5,-2);
\draw (1,2) to (1,-2);
\draw (1.5,2) to (1.5,-2);
\draw (2,2) to (2,-2);
\draw (2.5,2) to (2.5,-2);
\draw (3,2) to (3,-2);
\draw (3.5,2) to (3.5,-2);
\draw (4,2) to (4,-2);
\draw (4.5,2) to (4.5,-2);

\draw[dashed] (.5,2) to (.5,2.5);
\draw[dashed] (1,2) to (1,2.5);
\draw[dashed] (1.5,2) to (1.5,2.5);
\draw[dashed] (2,2) to (2,2.5);
\draw[dashed] (2.5,2) to (2.5,2.5);
\draw[dashed] (3,2) to (3,2.5);
\draw[dashed] (3.5,2) to (3.5,2.5);
\draw[dashed] (4,2) to (4,2.5);
\draw[dashed] (4.5,2) to (4.5,2.5);

\draw[dashed] (.5,-2) to (.5,-2.5);
\draw[dashed] (1,-2) to (1,-2.5);
\draw[dashed] (1.5,-2) to (1.5,-2.5);
\draw[dashed] (2,-2) to (2,-2.5);
\draw[dashed] (2.5,-2) to (2.5,-2.5);
\draw[dashed] (3,-2) to (3,-2.5);
\draw[dashed] (3.5,-2) to (3.5,-2.5);
\draw[dashed] (4,-2) to (4,-2.5);
\draw[dashed] (4.5,-2) to (4.5,-2.5);

\filldraw[draw=black,fill=white] (.5,2.5) circle (.075);
\filldraw[draw=black,fill=white] (1,2.5) circle (.075);
\filldraw[draw=black,fill=white] (1.5,2.5) circle (.075);
\filldraw[draw=black,fill=white] (2,2.5) circle (.075);
\filldraw[draw=black,fill=white] (2.5,2.5) circle (.075);
\filldraw[draw=black,fill=white] (3,2.5) circle (.075);
\filldraw[draw=black,fill=white] (3.5,2.5) circle (.075);
\filldraw[draw=black,fill=white] (4,2.5) circle (.075);
\filldraw[draw=black,fill=white] (4.5,2.5) circle (.075);

\filldraw[draw=black,fill=white] (.5,2) circle (.075);
\filldraw[draw=black,fill=white] (1,2) circle (.075);
\filldraw[draw=black,fill=white] (1.5,2) circle (.075);
\filldraw[draw=black,fill=white] (2,2) circle (.075);
\filldraw[draw=black,fill=white] (2.5,2) circle (.075);
\filldraw[draw=black,fill=white] (3,2) circle (.075);
\filldraw[draw=black,fill=black] (3.5,2) circle (.075);
\filldraw[draw=black,fill=white] (4,2) circle (.075);
\filldraw[draw=black,fill=white] (4.5,2) circle (.075);

\filldraw[draw=black,fill=black] (.5,1.5) circle (.075);
\filldraw[draw=black,fill=black] (1,1.5) circle (.075);
\filldraw[draw=black,fill=black] (1.5,1.5) circle (.075);
\filldraw[draw=black,fill=white] (2,1.5) circle (.075);
\filldraw[draw=black,fill=white] (2.5,1.5) circle (.075);
\filldraw[draw=black,fill=white] (3,1.5) circle (.075);
\filldraw[draw=black,fill=black] (3.5,1.5) circle (.075);
\filldraw[draw=black,fill=black] (4,1.5) circle (.075);
\filldraw[draw=black,fill=white] (4.5,1.5) circle (.075);

\filldraw[draw=black,fill=black] (.5,1) circle (.075);
\filldraw[draw=black,fill=black] (1,1) circle (.075);
\filldraw[draw=black,fill=white] (1.5,1) circle (.075);
\filldraw[draw=black,fill=white] (2,1) circle (.075);
\filldraw[draw=black,fill=white] (2.5,1) circle (.075);
\filldraw[draw=black,fill=white] (3,1) circle (.075);
\filldraw[draw=black,fill=black] (3.5,1) circle (.075);
\filldraw[draw=black,fill=black] (4,1) circle (.075);
\filldraw[draw=black,fill=white] (4.5,1) circle (.075);

\filldraw[draw=black,fill=white] (.5,.5) circle (.075);
\filldraw[draw=black,fill=white] (1,.5) circle (.075);
\filldraw[draw=black,fill=white] (1.5,.5) circle (.075);
\filldraw[draw=black,fill=white] (2,.5) circle (.075);
\filldraw[draw=black,fill=black] (2.5,.5) circle (.075);
\filldraw[draw=black,fill=black] (3,.5) circle (.075);
\filldraw[draw=black,fill=black] (3.5,.5) circle (.075);
\filldraw[draw=black,fill=black] (4,.5) circle (.075);
\filldraw[draw=black,fill=white] (4.5,.5) circle (.075);

\filldraw[draw=black,fill=white] (.5,0) circle (.075);
\filldraw[draw=black,fill=white] (1,0) circle (.075);
\filldraw[draw=black,fill=white] (1.5,0) circle (.075);
\filldraw[draw=black,fill=white] (2,0) circle (.075);
\filldraw[draw=black,fill=white] (2.5,0) circle (.075);
\filldraw[draw=black,fill=black] (3,0) circle (.075);
\filldraw[draw=black,fill=white] (3.5,0) circle (.075);
\filldraw[draw=black,fill=black] (4,0) circle (.075);
\filldraw[draw=black,fill=black] (4.5,0) circle (.075);

\filldraw[draw=black,fill=white] (.5,-.5) circle (.075);
\filldraw[draw=black,fill=white] (1,-.5) circle (.075);
\filldraw[draw=black,fill=black] (1.5,-.5) circle (.075);
\filldraw[draw=black,fill=white] (2,-.5) circle (.075);
\filldraw[draw=black,fill=white] (2.5,-.5) circle (.075);
\filldraw[draw=black,fill=white] (3,-.5) circle (.075);
\filldraw[draw=black,fill=black] (3.5,-.5) circle (.075);
\filldraw[draw=black,fill=black] (4,-.5) circle (.075);
\filldraw[draw=black,fill=black] (4.5,-.5) circle (.075);

\filldraw[draw=black,fill=white] (.5,-1) circle (.075);
\filldraw[draw=black,fill=black] (1,-1) circle (.075);
\filldraw[draw=black,fill=black] (1.5,-1) circle (.075);
\filldraw[draw=black,fill=white] (2,-1) circle (.075);
\filldraw[draw=black,fill=white] (2.5,-1) circle (.075);
\filldraw[draw=black,fill=white] (3,-1) circle (.075);
\filldraw[draw=black,fill=white] (3.5,-1) circle (.075);
\filldraw[draw=black,fill=white] (4,-1) circle (.075);
\filldraw[draw=black,fill=white] (4.5,-1) circle (.075);

\filldraw[draw=black,fill=white] (.5,-1.5) circle (.075);
\filldraw[draw=black,fill=black] (1,-1.5) circle (.075);
\filldraw[draw=black,fill=white] (1.5,-1.5) circle (.075);
\filldraw[draw=black,fill=white] (2,-1.5) circle (.075);
\filldraw[draw=black,fill=white] (2.5,-1.5) circle (.075);
\filldraw[draw=black,fill=white] (3,-1.5) circle (.075);
\filldraw[draw=black,fill=black] (3.5,-1.5) circle (.075);
\filldraw[draw=black,fill=white] (4,-1.5) circle (.075);
\filldraw[draw=black,fill=white] (4.5,-1.5) circle (.075);

\filldraw[draw=black,fill=white] (.5,-2) circle (.075);
\filldraw[draw=black,fill=white] (1,-2) circle (.075);
\filldraw[draw=black,fill=white] (1.5,-2) circle (.075);
\filldraw[draw=black,fill=white] (2,-2) circle (.075);
\filldraw[draw=black,fill=white] (2.5,-2) circle (.075);
\filldraw[draw=black,fill=white] (3,-2) circle (.075);
\filldraw[draw=black,fill=black] (3.5,-2) circle (.075);
\filldraw[draw=black,fill=black] (4,-2) circle (.075);
\filldraw[draw=black,fill=white] (4.5,-2) circle (.075);

\filldraw[draw=black,fill=white] (.5,-2.5) circle (.075);
\filldraw[draw=black,fill=white] (1,-2.5) circle (.075);
\filldraw[draw=black,fill=white] (1.5,-2.5) circle (.075);
\filldraw[draw=black,fill=white] (2,-2.5) circle (.075);
\filldraw[draw=black,fill=white] (2.5,-2.5) circle (.075);
\filldraw[draw=black,fill=white] (3,-2.5) circle (.075);
\filldraw[draw=black,fill=white] (3.5,-2.5) circle (.075);
\filldraw[draw=black,fill=white] (4,-2.5) circle (.075);
\filldraw[draw=black,fill=white] (4.5,-2.5) circle (.075);

\end{tikzpicture}
\end{center}
We call the shaded region a \emph{contour}: it has the key property that 
every edge that crosses its boundary connects an incorrectly decoded site 
inside with a correctly decoded site outside (it 
also proves to be convenient to assume contours are simply connected, 
i.e., that they have no holes; this is why the white vertex is included in 
the shaded region).

The main idea of the proof is that when the error probability $p$ is close 
to zero, it is very unlikely that a decoded configuration $z$ contains 
many edges that connect a correctly decoded site $r$ with an incorrectly 
decoded site $q$.  Intuitively, if the error probability $p$ is small, 
then the observations reveal with a high degree of confidence whether 
$x^q$ and $x^r$ coincide or differ; therefore, if we reconstruct one 
(in)correctly, then it is very likely that the other will be reconstructed 
(in)correctly as well.  More precisely, the $\Sigma^x$-probability of 
reconstructing a configuration that contains a given contour $J'\subseteq 
J$ is exponentially small in the length of the boundary of $J'$
(Lemma \ref{lem:peierls} below).  It then follows from a simple union 
bound that the probability that the decoded configuration contains 
\emph{any} contour that includes $\mathbf{0}$ is bounded away from 
$\frac{1}{2}$.  But this completes the proof, as at least one such contour 
must be present if $\mathbf{0}$ is incorrectly decoded.  (Of course, the 
crucial insight behind this approach is that the probability 
of containing a given contour and the number of contours of a given length
do not depend on the size of the domain, so that all the estimates that 
appear in this argument are dimension-free.)
\end{rem}

We now proceed to make these ideas precise.

\begin{prop}
There exists an absolute constant $0<p_\star<1/2$ such that
$$
	\mathbf{E}|\mathbf{E}[X_k^0|X_0,Y_1,\ldots,Y_k,
	\{X_1^v,\ldots,X_k^v:|v|>m\}]|\ge \tfrac{1}{4}
$$
for every $k,m\ge 1$ whenever $0<p<p_\star$.
\end{prop}

\begin{proof}
Let us fix $k,m\ge 1$ throughout the proof, and define 
$\mathbf{0}:=(k,0)\in J$.  We will prove below the following claim:
there exists an absolute constant $0<p_\star<1/2$ such that
$$
	\mathbf{P}\Bigg[
	\Sigma^x(\{z:z^{\mathbf{0}}=x^{\mathbf{0}}\}) \ge \frac{3}{4}
	\Bigg] \ge \frac{1}{2}
$$
whenever $0<p<p_\star$: that is, when the noise is sufficiently small, the 
conditional distribution $\mathbf{P}[X_k^0\in\,\cdot\,|X_0,Y_1,\ldots,Y_k,
\{X_1^v,\ldots,X_k^v:|v|>m\}]$ assigns a large probability to the actually 
realized value of $X_k^0$ at least half of the time (recall
Lemma \ref{lem:rbim}).  Let us complete
the proof assuming this claim.  Note that
$\Sigma^x(\{z:z^{\mathbf{0}}=x^{\mathbf{0}}\}) \ge 3/4$
implies $|\Sigma^x(\{z:z^{\mathbf{0}}=1\}) -
\Sigma^x(\{z:z^{\mathbf{0}}=-1\})| \ge 1/2$.
Thus the above claim implies that
$$
	\mathbf{P}\bigg[|\mathbf{E}[X_k^0|X_0,Y_1,\ldots,Y_k,
	\{X_1^v,\ldots,X_k^v:|v|>m\}]| \ge \frac{1}{2}
	\bigg|X_0,\ldots,X_k\bigg] \ge \frac{1}{2},
$$
where we have used Lemma \ref{lem:rbim} and the fact that $\{X^q\}$ and 
$\{\xi^{qr}\}$ are independent. The proof is now completed by a 
straightforward estimate.

It remains to prove the claim.  To this end, we use a Peierls 
argument.  Fix for the time being configurations $x,z\in\{-1,1\}^J$. 
For any $J'\subseteq J$, define the boundary edges
$$
	\mathfrak{E}J':=
	\{\{q,r\}:q\in J',~r\in (J\backslash J')\cup\partial
        J,~\|q-r\|=1\}.
$$
A subset $J'\subseteq J$ is called a \emph{contour} if it is simply 
connected, $z^q=-x^q$ for all $\{q,r\}\in\mathfrak{E}J'$ with $q\in J'$, 
and $z^r=x^r$ if in addition $r\in J\backslash J'$.  We will denote the 
set of contours as $\mathfrak{C}_{z,x}$ (note that the definition of a 
contour depends on the given configurations $z$ and $x$).
If $z^{\mathbf{0}}=-x^{\mathbf{0}}$, then there must exist a 
contour $J'\in\mathfrak{C}_{z,x}$ such that $\mathbf{0}\in J'$: indeed, 
construct $J'$ by choosing the maximal connected subset of $J$ 
such that $\mathbf{0}\in J'$ and $z^q=-x^q$ for all $q\in J'$, and then 
`fill in the holes' to make $J'$ simply connected.  Thus
$$
	\Sigma^x(\{z:z^{\mathbf{0}}=-x^{\mathbf{0}}\}) \le
	\Sigma^x(\{z:\exists\,J'\in\mathfrak{C}_{z,x},~
	\mathbf{0}\in J'\}) \le
	\sum_{J'\ni\mathbf{0}}
	\Sigma^x(\{z:J'\in\mathfrak{C}_{z,x}\}).
$$
Now note that, by the definition of a contour, $x^qz^q=-1$
whenever $\{q,r\}\in\mathfrak{E}J'$ with $q\in J'$, and $x^qx^rz^qz^r=-1$ 
if in addition $r\in J\backslash J'$.  Thus the existence of a contour 
implies the presence of many such edges.  The basic idea of the proof
is that the probability that this occurs is small under $\Sigma^x$
due to Lemma \ref{lem:rbim}.  Let us make this precise.

\begin{lem}
\label{lem:peierls}
For any $J'\subseteq J$, we have
$$
	\Sigma^x(\{z:J'\in\mathfrak{C}_{z,x}\}) \le
	\exp\bigg(
	-2\beta\sum_{\{q,r\}\in\mathfrak{E}J'}\xi^{qr}\bigg).
$$
\end{lem}

\begin{proof}
Assume without loss of generality that $J'$ is simply connected.
Let us use for simplicity the convention that $z^r=x^r$ 
for $r\in\partial J$.  Define the events
\begin{align*}
	A &= \{z:z^q=-x^q\mbox{ and }z^r=x^r\mbox{ for }
	\{q,r\}\in\mathfrak{E}J',~q\in J'\} ,\\
	B &= \{z:z^q=x^q\mbox{ and }z^r=x^r\mbox{ for }
	\{q,r\}\in\mathfrak{E}J',~q\in J'\}.
\end{align*}
Then we evidently have by Lemma \ref{lem:rbim}
$$
	\Sigma^x(\{z:J'\in\mathfrak{C}_{z,x}\}) =
	\Sigma^x(A) \le
	\frac{\Sigma^x(A)}{\Sigma^x(B)}.	
$$
An elementary computation shows that
$$
	\frac{\Sigma^x(A)}{\Sigma^x(B)} = 
	\exp\bigg(
	-2\beta\sum_{\{q,r\}\in\mathfrak{E}J'}\xi^{qr}\bigg)\,
	\frac{\sum_{z}\mathbf{1}_A(z)\exp(\beta\sum_{\{q,r\}\subseteq J'
	:\|q-r\|=1}
	\xi^{qr}x^qx^rz^qz^r)}{
	\sum_{z}\mathbf{1}_B(z)\exp(\beta\sum_{\{q,r\}\subseteq J'
        :\|q-r\|=1}
        \xi^{qr}x^qx^rz^qz^r)
	}.
$$
But the ratio in this expression is unity, as the exponential term inside
the sums is invariant under the transformation $z^q\mapsto -z^q$ for all
$q\in J'$.  The proof is complete.
\end{proof}

Lemma \ref{lem:peierls} allows us to estimate
\begin{align*}
	&\mathbf{P}\Bigg[
	\Sigma^x(\{z:z^{\mathbf{0}}=-x^{\mathbf{0}}\})] \ge 
	\sum_{J'\ni\mathbf{0}\mbox{ \small{simply connected}}}
	e^{-\beta|\mathfrak{E}J'|}
	\Bigg] \\
	&
	\le 
	\mathbf{P}\Bigg[
	\sum_{J'\ni\mathbf{0}\mbox{ \small{simp.\ conn.}}}
	\exp\bigg(
	-2\beta\sum_{\{q,r\}\in\mathfrak{E}J'}\xi^{qr}\bigg)
	\ge 
	\sum_{J'\ni\mathbf{0}\mbox{ \small{simp.\ conn.}}}
	e^{-\beta|\mathfrak{E}J'|}\Bigg] \\
	&\le 
	\mathbf{P}\Bigg[\exists\,J'\ni\mathbf{0}\mbox{ simply connected with }
	\sum_{\{q,r\}\in\mathfrak{E}J'}\xi^{qr}\le
	\frac{|\mathfrak{E}J'|}{2}
	\Bigg] \\
	&\le 
	\sum_{J'\ni\mathbf{0}\mbox{ \small{simply connected}}}
	\mathbf{P}\Bigg[
	\sum_{\{q,r\}\in\mathfrak{E}J'}\xi^{qr}\le
	\frac{|\mathfrak{E}J'|}{2}
	\Bigg].
\end{align*}
Using a standard combinatorial result \cite[Lemma 6.13]{Geo11}
$$
	|\{J'\subseteq J\mbox{ simply connected}
	:\mathbf{0}\in J',~|\mathfrak{E}J'|=l\}| \le l3^{l-1},
$$
as well as the simple bound
$$
	\mathbf{P}\Bigg[
	\sum_{\{q,r\}\in\mathfrak{E}J'}\xi^{qr}\le
	\frac{|\mathfrak{E}J'|}{2}
	\Bigg] =
	\mathbf{P}\Bigg[
	\mathrm{Bin}(|\mathfrak{E}J'|,1-p)
	\le
	\frac{3}{4}\,
	|\mathfrak{E}J'|
	\Bigg] \le
	2^{|\mathfrak{E}J'|} p^{|\mathfrak{E}J'|/4},
$$
we can conclude that
$$
	\mathbf{P}\Bigg[
	\Sigma^x(\{z:z^{\mathbf{0}}=-x^{\mathbf{0}}\})] \ge c_1
	\Bigg] \le c_2,\quad
	c_1 =
	\sum_{l=3}^\infty
	l3^{l-1}\bigg(\frac{p}{1-p}\bigg)^{l/2},\quad
	c_2 =
	\sum_{l=3}^\infty
	l3^{l-1}
	2^{l} p^{l/4}.
$$
But we can now evidently choose $p_\star>0$ sufficiently small such
that $c_1\le 1/4$ and $c_2\le 1/2$ whenever $p\le p_\star$,
which readily yields the desired estimate.
\end{proof}

\subsection{Proof of Theorem \ref{thm:phtrans}: high noise}

We now turn to proving that the filter is stable when the noise is strong.  
We begin by noting that it suffices to prove stability of 
finite-dimensional marginals of the filter.

\begin{lem}
\label{lem:finapprox}
Suppose that
$$
	\mathbf{E}|
	\mathbf{E}[f(X_k^{-m},\ldots,X_k^m)|X_0,Y_1,\ldots,Y_k]-
	\mathbf{E}[f(X_k^{-m},\ldots,X_k^m)|Y_1,\ldots,Y_k]
	|\xrightarrow{k\to\infty}0
$$
for every function $f$ and every $m\ge 1$.  Then the filter is stable.
\end{lem}

\begin{proof}
Fix any measurable subset $A$ of $\{-1,1\}^{\mathbb{Z}}$ and define
$$
	F_m = f_m(X_0^{-m},\ldots,X_0^m) := \mathbf{P}[X_0\in A|
	X_0^{-m},\ldots,X_0^m].
$$
We can estimate
\begin{align*}
	&\mathbf{E}|
	\mathbf{P}[X_k\in A|X_0,Y_1,\ldots,Y_k]-
	\mathbf{P}[X_k\in A|Y_1,\ldots,Y_k]
	| \le
	2\,\mathbf{E}|f_m(X_k^{-m},\ldots,X_k^m)-\mathbf{1}_A(X_k)| 
	\\ &\qquad\quad\mbox{}+
	\mathbf{E}|
	\mathbf{E}[f_m(X_k^{-m},\ldots,X_k^m)|X_0,Y_1,\ldots,Y_k]-
	\mathbf{E}[f_m(X_k^{-m},\ldots,X_k^m)|Y_1,\ldots,Y_k]
	|.
\end{align*}
By stationarity the first term does not depend on $k$, and the assumption 
gives
$$
	\limsup_{k\to\infty}
	\mathbf{E}|
	\mathbf{P}[X_k\in A|X_0,Y_1,\ldots,Y_k]-
	\mathbf{P}[X_k\in A|Y_1,\ldots,Y_k]
	| \le
	2\,\mathbf{E}|F_m-\mathbf{1}_A(X_0)|.
$$
Letting $m\to\infty$ and using the martingale convergence theorem
concludes the proof.
\end{proof}

We will in fact prove a much stronger \emph{pathwise} bound than is 
required by the above lemma. The basic tool we will use for this purpose 
is the Dobrushin comparison theorem \cite[Theorem 8.20]{Geo11}, which we 
state here in a convenient form.

\begin{thm}[Dobrushin comparison theorem]
\label{thm:dobr}
Let $\mu$ and $\nu$ be probability measures on 
$\{-1,1\}^I$ for some countable set $I$, and choose measurable
functions $m_i,n_i$ such that
$$
	m_i(X) = \mu(X^i=1|\{X^j:j\ne i\}),\qquad
	n_i(X) = \nu(X^i=1|\{X^j:j\ne i\}).
$$
Define
$$
	b_i := \sup_x|m_i(x)-n_i(x)|,\qquad
	C_{ji} := \sup_{x,z:x^v=z^v\mathrm{~for~}v\ne i}
	|m_j(x)-m_j(z)|,
$$
and assume that
$$
	\sup_{j\in I}\sum_{i\in I} C_{ji}<1.
$$
Then $D:=\sum_{n=0}^\infty C^n$ exists (in the sense of matrix algebra), 
and
$$
	|\mu(f)-\nu(f)| \le \sum_{j\in J}\sum_{i\in I} D_{ji}b_i
$$
whenever $J$ is a finite set, $f(x)$ depends only on $\{x^j:j\in J\}$, and 
$0\le f\le 1$.
\end{thm}

We will apply this result pathwise to compare the 
filters with and without conditioning on the initial condition.  To this 
end, we must compute the quantities that arise in the Dobrushin comparison 
theorem for suitably chosen regular conditional probabilities.

\begin{lem}
\label{lem:spec}
Fix any version of the regular conditional probabilities
$$
	\mu_{X,Y} := \mathbf{P}[X_0,\ldots,X_k\in\,\cdot\,|
	X_0,Y_1,\ldots,Y_k],\qquad
	\nu_{Y} := \mathbf{P}[X_0,\ldots,X_k\in\,\cdot\,|
	Y_1,\ldots,Y_k].
$$
Then there is a set $A$ with $\mathbf{P}[(X,Y)\in A]=1$ such that
for every $(x,y)\in A$
\begin{multline*}
	\mu_{x,y}(X_\ell^v=1|\{X_r^w:(r,w)\ne(\ell,v)\}) =
	\nu_{y}(X_\ell^v=1|\{X_r^w:(r,w)\ne(\ell,v)\}) = \\
	\frac{
	e^{\beta 
	\{\bar y_\ell^v X_{\ell-1}^v +
	\hat y_\ell^v X_\ell^{v+1} +
	\bar y_{\ell+1}^v X_{\ell+1}^v +
	\hat y_\ell^{v-1} X_\ell^{v-1}\}}
	}{
	e^{\beta 
	\{\bar y_\ell^v X_{\ell-1}^v +
	\hat y_\ell^v X_\ell^{v+1} +
	\bar y_{\ell+1}^v X_{\ell+1}^v +
	\hat y_\ell^{v-1} X_\ell^{v-1}\}}
	+
	e^{-\beta 
	\{\bar y_\ell^v X_{\ell-1}^v +
	\hat y_\ell^v X_\ell^{v+1} +
	\bar y_{\ell+1}^v X_{\ell+1}^v +
	\hat y_\ell^{v-1} X_\ell^{v-1}\}}
	}
\end{multline*}
for $1\le\ell<k$ and $v\in\mathbb{Z}$,
\begin{multline*}
	\mu_{x,y}(X_k^v=1|\{X_r^w:(r,w)\ne(k,v)\}) =
	\nu_{y}(X_k^v=1|\{X_r^w:(r,w)\ne(k,v)\}) = \\
	\frac{
	e^{\beta 
	\{\bar y_k^v X_{k-1}^v +
	\hat y_k^v X_k^{v+1} +
	\hat y_k^{v-1} X_k^{v-1}\}}
	}{
	e^{\beta 
	\{\bar y_k^v X_{k-1}^v +
	\hat y_k^v X_k^{v+1} +
	\hat y_k^{v-1} X_k^{v-1}\}}
	+
	e^{-\beta 
	\{\bar y_k^v X_{k-1}^v +
	\hat y_k^v X_k^{v+1} +
	\hat y_k^{v-1} X_k^{v-1}\}}
	}
\end{multline*}
for $v\in\mathbb{Z}$, and $\mu_{x,y}(X_0^v=1)=\mathbf{1}_{x_0^v=1}$
for $v\in\mathbb{Z}$, where $\beta:=\log\sqrt{(1-p)/p}$.
\end{lem}

\begin{proof}
It is an elementary fact that
\begin{align*}
	\mu_{X,Y}(X_\ell^v=1|\{X_r^w:(r,w)\ne(\ell,v)\}) &=
	\mathbf{P}[X_\ell^v=1|X_0,Y_1,\ldots,Y_k,\{X_r^w:(r,w)\ne(\ell,v)\}],\\
	\nu_{Y}(X_\ell^v=1|\{X_r^w:(r,w)\ne(\ell,v)\}) &=
	\mathbf{P}[X_\ell^v=1|Y_1,\ldots,Y_k,\{X_r^w:(r,w)\ne(\ell,v)\}],
\end{align*}
see \cite[p.\ 95--96]{vW83} or \cite[Lemma 3.4]{TvH13}.  That each 
statement in the Lemma holds for $\mathbf{P}$-a.e.\ $(x,y)$ can therefore 
be read off from Lemma \ref{lem:rbim}.  As there are countably many 
statements, they can be assumed to hold simultaneously on a set $A$ of 
unit measure.
\end{proof}

We can now complete the proof of filter stability for $p>p^\star$.

\begin{prop}
\label{prop:dob}
There exists an absolute constant $0<p^\star<1/2$ such that
$$
	|\mathbf{E}[f(X_k^{-m},\ldots,X_k^m)|X_0,Y_1,\ldots,Y_k]-
	\mathbf{E}[f(X_k^{-m},\ldots,X_k^m)|Y_1,\ldots,Y_k]|
	\le (8m+4)\|f\|_\infty e^{-k}
$$
a.s.\ for every $k,m\ge 1$ and function $f$ whenever $p^\star<p\le 1/2$.
\end{prop}

\begin{proof}
We apply Theorem \ref{thm:dobr} with $I=\{0,\ldots,k\}\times\mathbb{Z}$ 
and $\mu=\mu_{x,y}$, $\nu=\nu_y$ as defined in Lemma \ref{lem:spec}.
Evidently $b_{(0,v)}\le 1$ and $b_{(\ell,v)}=0$ for $1\le\ell\le k$ and 
$v\in\mathbb{Z}$, so we have 
$$
	|\mu_{x,y}(f(X_k^{-m},\ldots,X_k^m))-
	\nu_y(f(X_k^{-m},\ldots,X_k^m))| \le
	2\|f\|_\infty\sum_{w=-m}^m\sum_{v\in\mathbb{Z}}D_{(k,w)(0,v)}
$$
by Theorem \ref{thm:dobr} provided that the condition on the matrix $C$ is 
satisfied.

We proceed to estimate the matrix $C$ using Lemma \ref{lem:spec}.  
Evidently
$$
	C_{(\ell',v')(\ell,v)}=0\quad\mbox{if}\quad
	\ell'=0\quad\mbox{or}\quad|\ell'-\ell|+|v'-v|>1
	\quad\mbox{or}\quad\ell=\ell',~v=v'.
$$
On the other hand, note that by Lemma \ref{lem:spec}
$$
	\frac{e^{-4\beta}}{e^{4\beta}+e^{-4\beta}} \le
	\mu_{x,y}(X_\ell^v=1|\{X_r^w:(r,w)\ne(\ell,v)\})
	\le \frac{e^{4\beta}}{e^{4\beta}+e^{-4\beta}},
$$
so we can estimate
$$
	C_{ji} \le \tanh(4\beta) < 1\quad\mbox{for all }i,j\in I.
$$
It follows readily that
$$
	\|C\|_* :=
	\sup_{j\in I}\sum_{i\in I}e^{\|j-i\|}C_{ji}
	\le
	4e\tanh(4\beta).
$$
We can now evidently choose $0<p^\star<1/2$ such that 
$4e\tanh(4\beta)<1/2$ for $p^\star<p\le 1/2$.  Then the condition of
Theorem \ref{thm:dobr} is satisfied.  Moreover, as $\|\cdot\|_*$ is a 
matrix norm
$$
	\|D\|_* \le
	\sum_{n=0}^\infty \|C\|_*^n \le 2.
$$
Thus we obtain
\begin{align*}
	&|\mu_{x,y}(f(X_k^{-m},\ldots,X_k^m))-
	\nu_y(f(X_k^{-m},\ldots,X_k^m))| \\ &\quad\le
	(4m+2)\|f\|_\infty e^{-k}
	\max_{w=-m,\ldots,m}
	\sum_{v\in\mathbb{Z}}
	e^{\|(k,w)-(0,v)\|}
	D_{(k,w)(0,v)} \\
	&\quad
	\le
	(4m+2)\|D\|_*\|f\|_\infty e^{-k}
	\le
	(8m+4)\|f\|_\infty e^{-k}.
\end{align*}
As our estimates are valid for $\mathbf{P}$-a.e.\ $(x,y)$,
the proof is complete.
\end{proof}

\begin{rem}
It is natural to conjecture that one can choose $p_\star=p^\star$ in 
Theorem \ref{thm:phtrans}.  While we certainly believe this to be true, we 
were not able to prove this fact using standard methods.  The difficulty 
can be seen in Lemma \ref{lem:rbim}, as we presently explain.

Lemma \ref{lem:rbim} shows that that the conditional distribution of 
$X_1,\ldots,X_n$ given $Y_1,\ldots,Y_n$ can be viewed as an Ising model in 
the spin variables $\sigma^q:=x^qz^q$ with independent random interactions 
$\xi^{qr}$.  An Ising model is called ferromagnetic if all the 
interactions are positive.  In the ferromagnetic case, it is standard to 
establish the existence of a unique phase transition point by monotonicity 
arguments \cite[p.\ 100]{Geo11}.  Unfortunately, while our model is 
`ferromagnetic on average' as 
$\mathbf{P}[\xi^{qr}=1]>\mathbf{P}[\xi^{qr}=-1]$, there are always 
infinitely many interactions of either sign.  Thus correlation 
inequalities cannot be used, and in their absence it is not clear how to 
prove the existence of a simple phase boundary.  Monotonicity arguments 
will play a central role in section \ref{sec:mono} below to 
rule out conditional phase transitions in a somewhat different context.

While we have made no attempt to optimize the estimates for $p_\star$ and 
$p^\star$ that can be extracted from the proof of Theorem 
\ref{thm:phtrans}, the methods used here are not expected to yield 
realistic values of these constants.  On the other hand, the strong 
control provided by the Dobrushin comparison theorem has proved to be a 
powerful tool for the quantitative analysis of filtering algorithms in 
high dimension, cf.\ \cite{RvH13,RvH13b}.  
\end{rem}

\section{Symmetry breaking and observability}
\label{sec:obs}

Theorem \ref{thm:phtrans} shows that inheritance of ergodicity under 
conditioning cannot be taken for granted in infinite dimension even when 
the model is locally nondegenerate.  Are such phenomena prevalent in 
infinite dimension, or are they restricted to some carefully constructed 
examples?  We would like to understand in what situations 
such phenomena can be ruled out, both from the mathematical perspective 
and in view of the importance of filter stability (as well as 
spatial decay of correlations in infinite dimension) for the performance 
of practical filtering algorithms \cite{RvH13}.

It is not difficult to understand the mechanism that causes the filter to 
be unstable in Theorem \ref{thm:phtrans}.  In this model, the observations 
possess a global symmetry: the conditional law of $Y$ is unchanged under 
the transformation $X\mapsto -X$.  This symmetry renders the filter 
trivially unstable in the absence of observation noise, in precise analogy 
with Example \ref{ex:bla}.  In the finite-dimensional case, however, 
Theorem \ref{thm:vH09c} shows that the addition of any observation noise 
suffices to ensure that ergodicity of the underlying model is not broken 
by the additional symmetry introduced by conditioning.  The surprise in 
infinite dimension is that the qualitative effect of the added symmetry 
still persists in the presence of observation noise.  Thus local 
nondegeneracy in itself does not suffice to ensure the inheritance of 
ergodicity under conditioning.

On the other hand, the phenomenon exhibited in Theorem \ref{thm:phtrans} 
evidently cannot arise in models that do not possess observation 
symmetries.  It seems natural to conjecture that the presence of such 
symmetries is the \emph{only} possible obstruction to inheritance of 
ergodicity under conditioning: that is, inheritance of ergodicity is 
ensured once observation symmetries are ruled out.  It is not entirely 
obvious, however, how such a principle can be rigorously formulated.  On 
the other hand, even in the absence of a general definition, this 
intuitive notion should certainly be satisfied in many elementary 
observation models.  For example, let us state the following simple 
conjecture, which encapsulates the essence of the above intuition in 
the simplest possible setting.

\begin{conj}
\label{conj:obs}
Let $(X_k,Y_k)_{k\in\mathbb{Z}}$ be a stationary infinite-dimensional 
hidden Markov model as in section \ref{sec:ihmm} with 
$X_k\in\{-1,1\}^{\mathbb{Z}}$ and with $Y_k\in\{-1,1\}^{\mathbb{Z}}$ 
of the form
$$
	Y_k^v = X_k^v\xi_k^v,\qquad
	(\xi_k^v)_{k,v\in\mathbb{Z}}\mbox{ are i.i.d.}
	\independent X \mbox{ with } \mathbf{P}[\xi_k^v=-1]=p.
$$
If the underlying process $(X_k)_{k\in\mathbb{Z}}$ is stable, 
then the filter is stable.
\end{conj}

The idea behind this conjecture is that the direct observation structure 
$Y_k^v=X_k^v\xi_k^v$ is evidently devoid of symmetries for any
$p\ne\frac{1}{2}$: every configuration $x\in\{-1,1\}^{\mathbb{Z}}$ gives 
rise to a distinct observation law $\mathbf{P}[Y_k\in\,\cdot\,|X_k=x]$
(the case $p=\frac{1}{2}$ is trivial as then $Y\independent X$; we 
will therefore assume $p\ne\frac{1}{2}$ in the sequel).
Thus any mechanism 
of the type exhibited by Theorem \ref{thm:phtrans} is ruled out, and it 
seems hard to imagine another mechanism by which ergodicity of the 
underlying process could be obstructed due to conditioning on such 
informative observations.  Despite the seemingly obvious nature of this 
conjecture, we were not able to prove such a result in a general setting.  

The idea that stability of the filter is related to the absence of 
symmetries is not new in the infinite-dimensional setting.  It arises 
already in classical filtering models for a somewhat different reason: it 
may happen that the filter is stable even when the underlying model is 
\emph{not} ergodic.  In such situations, stability properties can emerge 
under the conditional distribution due to the informative nature of the 
observations; in essence, the filter will `forget' its initial 
distribution as the information contained therein is superseded by the 
information in the observations.\footnote{
	This is not the only mechanism that gives rise to stability of 
	nonlinear filters in classical nonergodic models; it is also 
	possible to exploit the local ergodicity of the model, cf.\ 
	\cite{DGLM11} and the references therein.  However, this approach 
	is unrelated to the symmetric breaking mechanism that is developed
	in this section.
} This phenomenon was made precise in the 
papers \cite{vH08, vH09, vH09b}.  While the theory developed in these 
papers is closely related to the symmetry breaking properties that we aim 
to exploit here, these results are not satisfactory in infinite dimension 
as will be explained below.

In this section, we will aim to extend such observability arguments to 
\emph{translation-invariant} systems in infinite dimension by exploiting a 
technique from multidimensional ergodic theory \cite{Con72}.  Somewhat 
surprisingly, the problem proves to be more tractable in the 
continuous-time setting, for which will establish validity of the natural 
analogue of Conjecture \ref{conj:obs}.  In its original discrete time 
formulation, however, our ultimate result falls short of establishing 
Conjecture \ref{conj:obs} even for translation-invariant models.  
Nonetheless, the theory developed here provides one possible mechanism for 
symmetry breaking in conditional ergodic theory. An entirely different 
mechanism will be discussed in the context of conditional random fields in 
section \ref{sec:mono} below.

The remainder of this section in organized as follows.  In section 
\ref{sec:genobs} we will recall some basic ideas from the general 
observability theory developed in \cite{vH08, vH09, vH09b}, and explain 
why these results are not satisfactory in infinite dimension.  We also
outline a slightly stronger entropic formulation that will provide the 
basis for a partial extension to translation-invariant systems in infinite 
dimension, which is developed in section \ref{sec:transinv}.  In section 
\ref{sec:conttime} we consider the continuous-time analogue of Conjecture 
\ref{conj:obs}, and provide a complete proof in this setting for the
translation-invariant case.

\subsection{General observability theory}
\label{sec:genobs}

Let $(X_k,Y_k)_{k\ge 0}$ be a general hidden Markov model as in section 
\ref{sec:filt}, and denote by $\mathbf{P}^\mu$ the law of the model with 
initial distribution $X_0\sim\mu$.  Define the \emph{prediction filter}
$$
	\tilde\pi_k^\mu:=\mathbf{P}^\mu[X_{k}\in\cdot|Y_1,\ldots,Y_{k-1}]
$$
(note that we are conditioning the state at time $k$ only on observations 
prior to time $k$).  The basic observation behind the observability theory 
of \cite{vH08, vH09, vH09b} is as follows.

\begin{thm}[\cite{vH09,vH09b}]
\label{thm:obs}
Suppose that $X_k$ takes values in a compact metric space and that 
$(X_k,Y_k)_{k\ge 0}$ is Feller.  Suppose that the following
\emph{observability} assumption holds:
$$
	\mathbf{P}^\mu[(Y_k)_{k\ge 0}\in\mbox{}\cdot\mbox{}]=
	\mathbf{P}^\nu[(Y_k)_{k\ge 0}\in\mbox{}\cdot\mbox{}]
	\quad\mbox{if and only if}\quad
	\mu=\nu.
$$
Then we have for every bounded and continuous function $f$
$$
	|\tilde\pi_k^\mu(f)-\tilde\pi_k^\nu(f)|
	\xrightarrow{k\to\infty}0\quad\mathbf{P}^\mu\mbox{-a.s.}
$$
whenever $\mathbf{P}^\mu[(Y_k)_{k\ge 0}\in\mbox{}\cdot\mbox{}]\ll
\mathbf{P}^\nu[(Y_k)_{k\ge 0}\in\mbox{}\cdot\mbox{}]$.
\end{thm}

Note that no ergodicity or nondegeneracy assumptions are imposed in this 
result: the only assumption is that of observability, that is, that 
distinct initial laws give rise to distinct observation laws.  The 
latter can evidently be viewed as a very mild type of symmetry breaking 
assumption.  A more general form of the result that does not require 
compactness or Feller assumptions can be found in \cite{vH09b}.

There is nothing in Theorem \ref{thm:obs} that prohibits its application 
in infinite-dimensional systems.  In fact, it is readily verified that the 
observability assumption of Theorem \ref{thm:obs} is automatically 
satisfied in the setting of Conjecture \ref{conj:obs} (provided $p\ne 
\frac{1}{2}$).  Nonetheless, unlike in most finite-dimensional problems, 
the conclusion of Theorem \ref{thm:obs} is not satisfactory in the 
infinite-dimensional setting, as we will presently explain.

There are two ways in which the conclusion of Theorem \ref{thm:obs} falls 
short of the filter stability property as it was introduced in section 
\ref{sec:filt}.  On the one hand, the conclusion applies to the prediction 
filter $\tilde\pi_k$ rather than to the filter $\pi_k$.  This is, however, 
not a major issue: one could argue that the prediction filter is of equal 
interest in practice as the filter itself, and thus we would be quite 
content in first instance to resolve a variant of Conjecture 
\ref{conj:obs} for the prediction filter.\footnote{
	In finite dimension, stability of the 
	filter can often be deduced from that of the prediction filter 
	and vice versa (e.g., \cite{vH08,TvH13}), so that the 
	conditional ergodic theory is not expected to be substantially 
	different in these two formulations.  The situation in infinite
	dimension is not entirely clear, however, cf.\ section
	\ref{sec:transinv} below.
}
Much more serious, however, is the fact that stability is only obtained 
for initial measures $\mu,\nu$ that give rise to \emph{absolutely 
continuous} observation laws.  If the state space of $X_k$ is countable, 
this is not a restriction: then
$\delta_x\ll\lambda$ whenever $\lambda(\{x\})>0$, so choosing 
$\mu=\delta_{X_0}$ and $\nu=\lambda$ in Theorem \ref{thm:obs} yields
$$
	|\mathbf{E}[f(X_k)|X_0,Y_1,\ldots,Y_{k-1}]-
	\mathbf{E}[f(X_k)|Y_1,\ldots,Y_{k-1}]|
	\xrightarrow{k\to\infty}0
	\quad\mbox{in }L^1,
$$
the natural counterpart of the stability property of section 
\ref{sec:filt} for the prediction filter.  In a continuous state space, 
this argument does not work, as a point mass is singular with respect to 
any nonatomic measure.  Nonetheless, this proves to be a minor problem in 
most finite-dimensional models: when the hidden Markov model possesses 
transition and observation densities, one can generally deduce absolute 
continuity of the observation laws $\mathbf{P}^\mu[(Y_k)_{k\ge 
0}\in\mbox{}\cdot\mbox{}]\ll\mathbf{P}^\nu[(Y_k)_{k\ge 0}\in\mbox{}\cdot\mbox{}]$
even when $\mu$ and $\nu$ are mutually singular.  This is discussed in 
detail in \cite[section 11.5]{CR11}.  In infinite dimension, on the other 
hand, the existence of transition and observation densities is out of the 
question, and thus the conclusion of Theorem \ref{thm:obs} is severely 
limited.  In particular, there is no hope to deduce the very natural 
filter stability property as introduced in section \ref{sec:filt}, or its 
prediction filter counterpart, by using the result or the method of proof 
of Theorem \ref{thm:obs}.

\begin{rem} While the result of Theorem \ref{thm:obs} falls short of the 
natural stability property as introduced in section \ref{sec:filt}, it 
does ensure a minimal form of stability for absolutely continuous initial 
conditions in the setting of Conjecture \ref{conj:obs}. Unfortunately, in 
infinite dimension, absolutely continuous measures are of little practical 
interest.  Indeed, if $\mu\ll\nu$ are absolutely continuous probability 
measures on an infinite product space $E^{\mathbb{Z}}$, then the density 
$\frac{d\mu}{d\nu}$ can be approximated arbitrarily well in $L^1(\nu)$ by 
densities that depend on a finite number of coordinates.  This implies 
that absolutely continuous probability measures in infinite dimension 
differ effectively only on a finite number of coordinates (indeed, if 
$\frac{d\mu}{d\nu}$ is a function of $\{x^v\}_{v\in I}$ for a finite set 
$I\subset\mathbb{Z}$, then $\mu(\{x^v\}_{v\not\in 
I}\in\cdot\,|\{x^v\}_{v\in I})= \nu(\{x^v\}_{v\not\in 
I}\in\cdot\,|\{x^v\}_{v\in I})$ by the Bayes formula, and thus the 
difference between $\mu$ and $\nu$ is entirely determined by the 
marginal on $I$). Thus the stability property for absolutely continuous 
initial measures does not capture the dissipation of an initial error in 
infinitely many coordinates, and is therefore of little relevance to the 
analysis of any reasonable approximation method that might arise in 
practice. \end{rem}

The proof of Theorem \ref{thm:obs} in \cite{vH09,vH09b} relies on a 
classical martingale argument due to Blackwell and Dubins \cite{BD62}.  
Instead of explaining this approach, let us present an alternative proof 
of stability, using the notion of entropy, in a special case: the 
finite-dimensional counterpart of Conjecture \ref{conj:obs}.  The key 
mechanism behind the proof is equivalent to that of 
\cite{vH09,vH09b}.
Nonetheless, the entropic formulation 
will be crucial to developing a nontrivial infinite-dimensional extension 
in section \ref{sec:transinv} below.

\begin{prop}
\label{prop:babyobs}
Let $(X_k,Y_k)_{k\in\mathbb{Z}}$ be a hidden Markov model as in 
section \ref{sec:filt} with $X_k\in\{-1,1\}^r$ (where $r<\infty$)
and with observations $Y_k\in\{-1,1\}^r$ of the form
$$
	Y_k^v = X_k^v\xi_k^v,\qquad
	(\xi_k^v)_{k\in\mathbb{Z},v\in\{1,\ldots,r\}}\mbox{ are i.i.d.}
	\independent X \mbox{ with } \mathbf{P}[\xi_k^v=-1]=p.
$$
Then the prediction filter is stable in the sense that
$$
	|\mathbf{P}[X_k\in A|X_0,Y_1,\ldots,Y_{k-1}]-
	\mathbf{P}[X_k\in A|Y_1,\ldots,Y_{k-1}]|
	\xrightarrow{k\to\infty}0\quad\mbox{in }L^1
$$
for every set $A$, provided that $p\ne\frac{1}{2}$.
\end{prop}

\begin{proof}
We use standard ideas from information theory \cite{CT06}.
Let $H(X)$ denote entropy and $H(X|Y)$
denote conditional entropy of discrete random variables $X,Y$.  Then
\begin{align*}
	&\sum_{k=1}^n\{H(Y_k|Y_1,\ldots,Y_{k-1})-
	H(Y_k|X_0,Y_1,\ldots,Y_{k-1})\} \\ &=
	H(Y_1,\ldots,Y_n)-H(Y_1,\ldots,Y_n|X_0)
	\phantom{\sum_{k=1}^n} \\
	&= H(X_0)-H(X_0|Y_1,\ldots,Y_n) \le H(X_0),
\end{align*}
where we have used the chain rule for entropy and
symmetry of mutual information $H(X)-H(X|Y)=H(Y)-H(Y|X)=:I(X;Y)$.
As $H(X_0)\le r$ is independent of $n$,
$$
	H(Y_k|Y_1,\ldots,Y_{k-1})-
        H(Y_k|X_0,Y_1,\ldots,Y_{k-1})
	\xrightarrow{k\to\infty}0.
$$
But note that 
\begin{multline*}
	H(Y_k|Y_1,\ldots,Y_{k-1})-
        H(Y_k|X_0,Y_1,\ldots,Y_{k-1}) = \mbox{}\\ 
	\mathbf{E}\Bigg[\sum_{y\in\{-1,1\}^r} 
	\mathbf{P}[Y_k=y|X_0,Y_1,\ldots,Y_{k-1}]
	\log_2\frac{\mathbf{P}[Y_k=y|X_0,Y_1,\ldots,Y_{k-1}]}{
	\mathbf{P}[Y_k=y|Y_1,\ldots,Y_{k-1}]}\Bigg],
\end{multline*}
which is precisely the expected relative entropy between the 
conditional 
distributions
$\mathbf{P}[Y_k\in\,\cdot\,|X_0,Y_1,\ldots,Y_{k-1}]$ and
$\mathbf{P}[Y_k\in\,\cdot\,|Y_1,\ldots,Y_{k-1}]$. Thus by Pinsker's 
inequality
$$
	|\mathbf{E}[g(Y_k)|X_0,Y_1,\ldots,Y_{k-1}]-
	\mathbf{E}[g(Y_k)|Y_1,\ldots,Y_{k-1}]|
	\xrightarrow{k\to\infty}0\quad\mbox{in }L^1
$$
for every function $g$.  Now note that, by the conditional independence 
structure of the observations, we have 
$\mathbf{E}[g(Y_k)|X_0,Y_1,\ldots,Y_{k-1},X_k]=
\mathbf{E}[g(Y_k)|Y_1,\ldots,Y_{k-1},X_k]=
\mathbf{E}[g(Y_k)|X_k]$.  Thus we can write using the tower property
\begin{align*}
	\mathbf{E}[g(Y_k)|X_0,Y_1,\ldots,Y_{k-1}]&=
	\mathbf{E}[f(X_k)|X_0,Y_1,\ldots,Y_{k-1}],\\
	\mathbf{E}[g(Y_k)|Y_1,\ldots,Y_{k-1}]&=
	\mathbf{E}[f(X_k)|Y_1,\ldots,Y_{k-1}],
\end{align*}
where 
$$
	f(x)=\mathbf{E}[g(Y_k)|X_k=x] = 
	(T_1\cdots T_rg)(x),\qquad
	(T_ig)(x):=(1-p)g(x)+p g(x^{-i})
$$
and $x^{-i}:=(x^1,\ldots,x^{i-1},-x^i,x^{i+1},\ldots,x^r)$.  But it 
is readily seen that if $p\ne\frac{1}{2}$, then each operator $T_i$ is 
invertible. Thus \emph{every} function $f$ is of the form 
$f(x)=\mathbf{E}[g(Y_k)|X_k=x]$ 
for some function $g$, and the proof is evidently complete.
\end{proof}

Let us emphasize the two key ideas of the proof.  We first establish that 
the conditional distribution of the \emph{next observation} $Y_k$ given 
the observation history $Y_1,\ldots,Y_{k-1}$ `forgets' the initial 
condition $X_0$ as $k\to\infty$.  That is, unlike the 
prediction filter which aims to predict the \emph{next hidden state} 
$X_k$, stability of predicting the next observation holds regardless of 
any properties of the model: it is a simple consequence of the finite 
entropy property $H(X_0)<\infty$ of the initial condition (this fails in 
the continuous case, and for this reason the absolute continuity 
assumption $\mathbf{P}^\mu[(Y_k)_{k\ge 0}\in\mbox{}\cdot\mbox{}]\ll 
\mathbf{P}^\nu[(Y_k)_{k\ge 0}\in\mbox{}\cdot\mbox{}]$ is essential in the 
general setting of Theorem \ref{thm:obs}).  However, due to the absence of 
symmetry in the observations, predicting the next observation $Y_k$ is 
\emph{equivalent} to predicting the hidden state $X_k$.  It is here that 
observability enters the picture.

\subsection{Translation-invariant systems}
\label{sec:transinv}

The proof of Proposition \ref{prop:babyobs} relies crucially on the fact 
that $X_k$ and $Y_k$ take values in a finite state space $\{-1,1\}^r$ with 
$r<\infty$, so that the entropies $H(Y_1,\ldots,Y_n)$ and $H(X_0)$ are 
finite.  In the infinite-dimensional setting of Conjecture \ref{conj:obs} 
where $X_k$ and $Y_k$ take values in $\{-1,1\}^{\mathbb{Z}}$, this is no 
longer the case.  Nonetheless, entropy arguments are ubiquitous in the 
study of infinite-dimensional systems that are 
\emph{translation-invariant}: in such systems, the role of entropy is 
replaced by that of the entropy rate (or `specific entropy'), which 
plays an central role in the ergodic theory of measurable dynamical 
systems, cf.\ \cite[Part 2]{Gla03}, and in statistical mechanics where it 
gives rise to thermodynamic formalism, cf.\ \cite[Chapters 15--16]{Geo11}.  
The aim of this section is to develop an infinite-dimensional counterpart 
to Proposition \ref{prop:babyobs} using this formalism.

Let $(X_k,Y_k)_{k\in\mathbb{Z}}$ be a 
stationary infinite-dimensional hidden Markov model as in section 
\ref{sec:ihmm} with $X_k\in\{-1,1\}^{\mathbb{Z}}$ and with 
$Y_k\in\{-1,1\}^{\mathbb{Z}}$ 
of the form
$$
	Y_k^v = X_k^v\xi_k^v,\qquad
	(\xi_k^v)_{k,v\in\mathbb{Z}}\mbox{ are i.i.d.}\independent X
	\mbox{ with } \mathbf{P}[\xi_k^v=-1]=p.
$$
This model is the same as in Conjecture \ref{conj:obs}.  In addition, 
however, we will assume throughout this section that the model is 
translation-invariant: that is, that the local transition kernels
$P^v(x,dz)=P^{w}(x,dz)$ for all $v,w\in\mathbb{Z}$ (cf.\ section 
\ref{sec:ihmm}), and that 
$$
	(X_k^v,Y_k^v)_{k,v\in\mathbb{Z}}
	\,\mathop{\stackrel{\mathrm{law}}{=}}\,
	(X_{k+\ell}^{v+w},Y_{k+\ell}^{v+w})_{k,v\in\mathbb{Z}}
	\quad\mbox{for all }\ell,w\in\mathbb{Z}.
$$
We now state the main result of this section.

\begin{thm}
\label{thm:iobs}
For the model of this section, if $p\ne\frac{1}{2}$,
we have
$$
	|\mathbf{P}[X_k\in A|X_0,Y_1,\ldots,Y_{k-1},Y_k^{<v}]-
	\mathbf{P}[X_k\in A|Y_1,\ldots,Y_{k-1},Y_k^{<v}]|
	\xrightarrow{k\to\infty}0\quad\mbox{in }L^1
$$
for every measurable set $A$ and $v\in\mathbb{Z}$,
where $Y_k^{<v}:=\{Y_k^w:w<v\}$.
\end{thm}

Theorem \ref{thm:iobs} establishes stability of a quantity
$$
	\mathbf{P}[X_k\in A|Y_1,\ldots,Y_{k-1},Y_k^{<v}]
$$
that is intermediate between the prediction filter (section 
\ref{sec:genobs})
$$
	\mathbf{P}[X_k\in A|Y_1,\ldots,Y_{k-1}]	
$$
and the filter (section \ref{sec:filt})
$$
	\mathbf{P}[X_k\in A|Y_1,\ldots,Y_{k}]. 
$$
It will become evident in the proof that the intermediate filter is the 
natural infinite-dimensional object that arises in extending the 
observability theory of section \ref{sec:genobs} to translation-invariant 
systems. On the other hand, from the practical point of view, the 
intermediate filter is a somewhat strange object.  Unlike the prediction 
filter, which arises when predicting the next state given the observations 
to date, and the filter, which arises when tracking the current state 
given the observation history, the intermediate filter uses an unnatural 
subset of the observations in the current time step. In particular, 
Theorem \ref{thm:iobs} falls short of establishing either Conjecture 
\ref{conj:obs} or its natural prediction filter counterpart. Nonetheless, 
Theorem \ref{thm:iobs} could provide a step towards establishing 
Conjecture \ref{conj:obs} for translation-invariant models (cf.\ Remark 
\ref{rem:weirdexchg} below).  Moreover, in section \ref{sec:conttime} 
below, we will see that this shortcoming of Theorem \ref{thm:iobs} can be 
resolved for the natural continuous-time analogue of Conjecture 
\ref{conj:obs}.

\begin{rem}
Note that Theorem \ref{thm:iobs} does not impose any ergodicity assumption 
on the underlying hidden Markov model: only observability was used to 
establish the result.  In this sense, this result goes beyond the spirit 
of Conjecture \ref{conj:obs}, which states that ergodicity of the 
underlying model is inherited by the filter in models with informative 
observations.  
Indeed, the result of Theorem \ref{thm:iobs} cannot be interpreted as 
establishing the inheritance of ergodicity, as ergodicity plays no role in 
the argument; rather, the intermediate filter is rendered stable here 
entirely due to the the absence of observation symmetries, even when the 
underlying model is not ergodic.

One might therefore expect that ergodicity should play no role in 
Conjecture \ref{conj:obs} either: after all, the mechanism by which we are 
exploiting the absence of observation symmetries appears to be independent 
of the ergodic properties of the model.  However, it is possible that 
ergodicity must nonetheless enter the picture in order to extend the 
conclusion of Theorem \ref{thm:iobs} from the intermediate filter to the 
filter, so that neither ergodicity nor observability suffices by itself to 
ensure stability of the filter in infinite dimensional models. Some 
evidence for this possibility will be discussed in section 
\ref{sec:graphs}.  On the other hand, the continuous-time results of 
section \ref{sec:conttime} below could be viewed as evidence to the 
contrary.  New ideas appear to be needed to resolve these questions.
\end{rem}

We now turn to the proof of Theorem \ref{thm:iobs}.  To this end we begin 
by proving the following result, which replaces the key step in the proof 
of Proposition \ref{prop:babyobs}.

\begin{prop}
\label{prop:pinsker}
For the model of this section, we have
$$
	H(Y_k^v|Y_1,\ldots,Y_{k-1},Y_k^{<v})-
	H(Y_k^v|X_0,Y_1,\ldots,Y_{k-1},Y_k^{<v})
	\xrightarrow{k\to\infty}0
$$
for every $v\in\mathbb{Z}$.
\end{prop}

\begin{rem}
\label{rem:conze}
Consider a stationary stochastic process $(Z_k)_{k\in\mathbb{Z}}$
such that $Z_k$ takes values in a finite set.  The entropy rate $h(Z)$ of 
the process can be expressed as
$$
	h(Z) := \lim_{n\to\infty} \frac{H(Z_1,\ldots,Z_n)}{n} =
	\lim_{n\to\infty}\frac{1}{n}\sum_{k=1}^n
	H(Z_k|Z_1,\ldots,Z_{k-1}) =
	H(Z_1|Z_0,Z_{-1},\ldots),
$$
where we used the chain rule and stationarity, respectively.
This can be used to derive nontrivial identities for conditional 
entropies.  Of particular importance in the present context is that if 
$Z_k=(X_k,Y_k)$ (still taking values in a finite set), 
then \cite[Lemma 18.2]{Gla03}
$$
	\lim_{n\to\infty}
	H(Y_1|Y_0,Y_{-1},\ldots;X_{-n},X_{-n-1},\ldots)=
	H(Y_1|Y_0,Y_{-1},\ldots).
$$
This result, which could be employed directly in Proposition 
\ref{prop:babyobs} instead of the simpler argument given there,
arises in the proof of the Rokhlin-Sinai characterization of Kolmogorov 
automorphisms in ergodic theory, cf.\ \cite[Chapter 18]{Gla03}.

The key to the proof of Proposition \ref{prop:pinsker} is that
the above ideas admit a multidimensional extension. Let 
$(Z^q)_{q\in\mathbb{Z}^2}$ be a translation-invariant random field
such that $Z^q$ takes values in a finite set.  In this setting, the 
entropy rate (or `specific entropy') $h(Z)$ can be expressed in terms 
of the lexicographic order $\prec$ on $\mathbb{Z}^2$ \cite{Fol73,Con72}:
$$
	h(Z) := \lim_{n\to\infty} \frac{H(Z^{B_n})}{|B_n|} =
	H(Z^q|Z^{\prec q}),
$$
where $B_n$ is the centered box in $\mathbb{Z}^2$ with radius
$n$ and $Z^{B_n}=\{Z^q:q\in B_n\}$, $Z^{\prec q}=\{Z^u:u\prec q\}$.
The random field analogue to the above entropy identity was obtained by 
Conze \cite[eq.\ (20)]{Con72}: if $Z^q=(X_k^v,Y_k^v)$ for 
$q=(k,v)\in\mathbb{Z}^2$ take values in a finite set, then
$$
	\lim_{n\to\infty}
	H(Y_1^v|Y_1^{<v},Y_0,Y_{-1},\ldots;
	X_{-n}^{<v},X_{-n-1},X_{-n-2},\ldots)=
	H(Y_1^v|Y_1^{<v},Y_0,Y_{-1},\ldots)
$$
for any $v\in\mathbb{Z}$.  The proof of Proposition \ref{prop:pinsker} 
follows from this identity.
\end{rem}

\begin{rem}
By arguing as in the previous remark, the result of Proposition 
\ref{prop:pinsker} can be rewritten in terms
of \emph{conditional entropy rates}. If we define
\begin{align*}
	h(Y_k|Y_1,\ldots,Y_{k-1})&:=
	\lim_{n\to\infty}\frac{H(Y_k^1,\ldots,Y_k^n|Y_1,\ldots,Y_{k-1})}{n},
	\\
	h(Y_k|X_0,Y_1,\ldots,Y_{k-1})&:=
	\lim_{n\to\infty}\frac{H(Y_k^1,\ldots,Y_k^n|X_0,Y_1,\ldots,Y_{k-1})}{n},
\end{align*}
then the conclusion of
Proposition \ref{prop:pinsker} can be expressed as follows:
$$
	h(Y_k|Y_1,\ldots,Y_{k-1})-
	h(Y_k|X_0,Y_1,\ldots,Y_{k-1})\xrightarrow{k\to\infty}0.
$$
This could be interpreted as a direct extension of the key step in the 
proof of Proposition \ref{prop:babyobs} to translation-invariant systems 
in infinite dimension; the finite-dimensional notion of entropy is simply 
replaced by its infinite-dimensional counterpart, the entropy 
rate.  
\end{rem}

\begin{proof}[Proof of Proposition \ref{prop:pinsker}]
By stationarity, it suffices to show that
$$
	H(Y_0^v|Y_{-k+1},\ldots,Y_{-1},Y_0^{<v})-
	H(Y_0^v|X_{-k},Y_{-k+1},\ldots,Y_{-1},Y_0^{<v})
	\xrightarrow{k\to\infty}0.
$$
First, we note that
\begin{align*}
	\lim_{k\to\infty}H(Y_0^v|Y_{-k+1},\ldots,Y_{-1},Y_0^{<v})&=
	H(Y_0^v|Y_0^{<v},Y_{-1},Y_{-2},\ldots) \\ &=
	\lim_{k\to\infty}
	H(Y_0^v|Y_0^{<v},Y_{-1},\ldots;
	X_{-k}^{<v},X_{-k-1},\ldots),
\end{align*}
where we have used the identity of Conze \cite[eq.\ (20)]{Con72} (cf.\ 
Remark \ref{rem:conze} above).  As
\begin{align*}
	H(Y_0^v|Y_0^{<v},Y_{-1},\ldots;
	X_{-k+1}^{<v},X_{-k},\ldots)
	&\le
	H(Y_0^v|Y_0^{<v},Y_{-1},\ldots;
	X_{-k},X_{-k-1},\ldots) 
	\\ &\le
	H(Y_0^v|Y_0^{<v},Y_{-1},\ldots;
	X_{-k}^{<v},X_{-k-1},\ldots),
\end{align*}
we obtain
$$
	\lim_{k\to\infty}
	H(Y_0^v|Y_0^{<v},Y_{-1},\ldots;
	X_{-k}^{<v},X_{-k-1},\ldots) =
	\lim_{k\to\infty}
	H(Y_0^v|Y_0^{<v},Y_{-1},\ldots;
	X_{-k},X_{-k-1},\ldots).
$$
But by the 
hidden Markov model structure, $\{Y_{-k},Y_{-k-1},\ldots;
X_{-k-1},X_{-k-2},\ldots\}$ and
$\{Y_0,Y_{-1},\ldots,Y_{-k+1}\}$ are conditionally independent given 
$X_{-k}$.  Thus
$$
	H(Y_0^v|Y_0^{<v},Y_{-1},\ldots;
	X_{-k},X_{-k-1},\ldots) =
	H(Y_0^v|Y_0^{<v},Y_{-1},\ldots,Y_{-k+1},X_{-k}),
$$
and the proof is complete.
\end{proof}

Proposition \ref{prop:pinsker} concerns the stability of prediction of the 
next observation.  As in the proof of Proposition \ref{prop:babyobs}, we 
will transform such properties into stability properties of the filter by 
using the informative nature of the observations.

\begin{proof}[Proof of Theorem \ref{thm:iobs}]
By translation-invariance and Lemma \ref{lem:finapprox}, it suffices
to show that
$$
	\mathbf{P}[(X_k^1,\ldots,X_k^m)\in B|X_0,Y_1,\ldots,Y_{k-1},Y_k^{<v}]-
	\mathbf{P}[(X_k^1,\ldots,X_k^m)\in B|Y_1,\ldots,Y_{k-1},Y_k^{<v}]
	\to 0
$$
in $L^1$ as $k\to\infty$ for every set $B$, $m\ge 1$ and $v\in\mathbb{Z}$.

Suppose first that $v\le 1$.  By the chain rule for entropy and
Proposition \ref{prop:pinsker}, we have
\begin{multline*}
	H(Y_k^v,\ldots,Y_k^m|Y_1,\ldots,Y_{k-1},Y_k^{<v})-
	H(Y_k^v,\ldots,Y_k^m|X_0,Y_1,\ldots,Y_{k-1},Y_k^{<v}) = \\
	\sum_{\ell=v}^m\{
	H(Y_k^\ell|Y_1,\ldots,Y_{k-1},Y_k^{<\ell})-
	H(Y_k^\ell|X_0,Y_1,\ldots,Y_{k-1},Y_k^{<\ell})\} 
	\xrightarrow{k\to\infty}0.
\end{multline*}
Following verbatim the second part of the proof of Proposition
\ref{prop:babyobs} yields
$$
	\mathbf{P}[(X_k^v,\ldots,X_k^m)\in C|X_0,Y_1,\ldots,Y_{k-1},Y_k^{<v}]-
	\mathbf{P}[(X_k^v,\ldots,X_k^m)\in C|Y_1,\ldots,Y_{k-1},Y_k^{<v}]
	\to 0
$$
in $L^1$ as $k\to\infty$ for every set $C$, and thus the result follows.

Now suppose that $v>1$.  We can assume without loss of generality that
$v\le m$ (otherwise the conclusion follows from the result for $m=v$).
We may also assume that $0<p<1$ (otherwise the 
conclusion is trivial).  By the Bayes formula,
\begin{multline*}
	\mathbf{P}[(X_k^1,\ldots,X_k^m)=x|Y_1,\ldots,Y_{k-1},Y_k^{<v}] = 
\\	\frac{
	\mathbf{P}[(X_k^1,\ldots,X_k^m)=x|Y_1,\ldots,Y_{k-1},Y_k^{<1}]
	\prod_{i=1}^v g(x^i,Y_k^i)}{
	\sum_{z\in\{-1,1\}^m}
	\mathbf{P}[(X_k^1,\ldots,X_k^m)=z|Y_1,\ldots,Y_{k-1},Y_k^{<1}]
	\prod_{i=1}^v g(z^i,Y_k^i)}
\end{multline*}
where $g(u,y)=\mathbf{P}[Y_k^i=y|X_k^i=u]$,
and $\mathbf{P}[(X_k^1,\ldots,X_k^m)=x|X_0,Y_1,\ldots,Y_{k-1},Y_k^{<v}]$
satisfies the analogous expression.
As $0<\inf g\le \sup g<\infty$, it follows readily that
$$
	\mathbf{P}[(X_k^1,\ldots,X_k^m)=x|X_0,Y_1,\ldots,Y_{k-1},Y_k^{<1}]-
	\mathbf{P}[(X_k^1,\ldots,X_k^m)=x|Y_1,\ldots,Y_{k-1},Y_k^{<1}]\to 0
$$
for all $x$ implies
$$
	\mathbf{P}[(X_k^1,\ldots,X_k^m)=x|X_0,Y_1,\ldots,Y_{k-1},Y_k^{<v}]-
	\mathbf{P}[(X_k^1,\ldots,X_k^m)=x|Y_1,\ldots,Y_{k-1},Y_k^{<v}]\to 0
$$
for all $x$, and thus the proof is complete.
\end{proof}

\begin{rem}
\label{rem:weirdexchg}
It is evident that the intermediate filter reduces to the filter if we let 
$v\to\infty$; thus Conjecture \ref{conj:obs} would be established for 
translation-invariant models if the limits as $k\to\infty$ and 
$v\to\infty$ could be exchanged in Theorem \ref{thm:iobs}.  Similarly, 
we could aim to obtain the conclusion of Conjecture \ref{conj:obs} for the 
prediction filter by letting $v\to-\infty$.
However, we do not know how to establish the validity of Theorem 
\ref{thm:iobs} in either limit.

To obtain some insight into this idea, let us rewrite 
Theorem \ref{thm:iobs} in a measure-theoretic manner.
Using the Markov property and translation invariance, we obtain
\begin{multline*}
	\mathbf{E}|\mathbf{P}[X_k\in A|X_0,Y_1,\ldots,Y_{k-1},Y_k^{<v}]-
	\mathbf{P}[X_k\in A|Y_1,\ldots,Y_{k-1},Y_k^{<v}]| = \\
	\mathbf{E}|\mathbf{P}[X_0\in 
	A|Y_0^{<v},Y_{-1},Y_{-2},\ldots;X_{-k},X_{-k-1},\ldots] -
	\mathbf{P}[X_0\in A|Y_0^{<v},Y_{-1},\ldots,Y_{-k+1}]|
\end{multline*}
as in the proof of Proposition \ref{prop:pinsker}.  Letting $k\to\infty$
and using Theorem \ref{thm:iobs} yields
\begin{align*}
	0&=\lim_{k\to\infty}
	\mathbf{E}|\mathbf{P}[X_k\in A|X_0,Y_1,\ldots,Y_{k-1},Y_k^{<v}]-
	\mathbf{P}[X_k\in A|Y_1,\ldots,Y_{k-1},Y_k^{<v}]| \\ &= 
	\mathbf{E}|\mathbf{P}[X_0\in A|\textstyle{\bigcap_k}
	(\mathcal{Y}_-^v\vee\mathcal{X}_{-k})]-
	\mathbf{P}[X_0\in A|\mathcal{Y}_-^v]|,
\end{align*}
where we defined the $\sigma$-fields 
$\mathcal{X}_{k}=\sigma\{X_k,X_{k-1},\ldots\}$ and
$\mathcal{Y}_-^v=\sigma\{Y_0^{<v},Y_{-1},Y_{-2},\ldots\}$.

Now let us attempt to take the limit as $v\to-\infty$.  This yields
$$
	\mathbf{E}|\mathbf{P}[X_0\in A|\textstyle{\bigcap_{k,v}}
	(\mathcal{Y}_-^v\vee\mathcal{X}_{-k})]-
	\mathbf{P}[X_0\in A|
	\textstyle{\bigcap_{v}}\mathcal{Y}_-^v]| = 0.
$$
This does not suffice to establish Conjecture \ref{conj:obs} for the 
prediction filter.  In order to deduce the latter, we would need to 
establish the identity
$$
	\bigcap_v\mathcal{Y}_-^v = \mathcal{Y}_- := \sigma\{Y_{-1},Y_{-2},
	\ldots\}\quad\mathop{\mathrm{mod}}\mathbf{P}
$$
(the notation $\mathcal{F}=\mathcal{G}$ $\mathrm{mod}\,\mathbf{P}$
indicates that the $\mathbf{P}$-completions of the $\sigma$-fields 
$\mathcal{F}$ and $\mathcal{G}$ coincide).
Indeed, if this is the case, then we obtain by Jensen's inequality
\begin{align*}
	&\lim_{k\to\infty}
	\mathbf{E}|\mathbf{P}[X_k\in A|X_0,Y_1,\ldots,Y_{k-1}]-
	\mathbf{P}[X_k\in A|Y_1,\ldots,Y_{k-1}]| 
	\\ &= 
	\mathbf{E}|\mathbf{P}[X_0\in A|\textstyle{\bigcap_k}
	(\mathcal{Y}_-\vee\mathcal{X}_{-k})]-
	\mathbf{P}[X_0\in A|\mathcal{Y}_-]| 
	\\ &\le
	\mathbf{E}|\mathbf{P}[X_0\in A|\textstyle{\bigcap_{k,v}}
	(\mathcal{Y}_-^v\vee\mathcal{X}_{-k})]-
	\mathbf{P}[X_0\in A|\mathcal{Y}_-]| = 0,
\end{align*}
where we used $\bigcap_{k,v}(\mathcal{Y}_-^v\vee\mathcal{X}_{-k})\supseteq
\bigcap_{k}(\mathcal{Y}_-\vee\mathcal{X}_{-k})\supseteq
\mathcal{Y}_-$.

Similarly, let us take $v\to\infty$.  This yields
$$
	\mathbf{E}|\mathbf{P}[X_0\in A|\textstyle{\bigvee_{v}\bigcap_{k}}
	(\mathcal{Y}_-^v\vee\mathcal{X}_{-k})]-
	\mathbf{P}[X_0\in A|
	\textstyle{\bigvee_{v}}\mathcal{Y}_-^v]| = 0.
$$
In order to establish Conjecture \ref{conj:obs}, we would now need the
identity
$$
	\bigvee_{v}\bigcap_{k}  
        (\mathcal{Y}_-^v\vee\mathcal{X}_{-k})=
	\bigcap_{k}  
	\bigvee_{v}
        (\mathcal{Y}_-^v\vee\mathcal{X}_{-k})
	\quad\mathop{\mathrm{mod}}\mathbf{P},
$$
the remainder of the argument proceeding in the same manner
as for $v\to-\infty$.

Neither of the above measure-theoretic identities appears to be obvious;
indeed, the problem of establishing such identities is closely related to 
the filter stability problem itself (cf.\ section \ref{sec:meas}).  
Nonetheless, the conclusion of Theorem \ref{thm:iobs} appears to be 
tantalizingly close to establishing Conjecture \ref{conj:obs} for 
translation-invariant systems, and the fact that the latter does not 
appear to follow directly from the former provides one more indication of 
the delicacy of the filter stability problem in infinite dimension.  A 
very similar argument will be used in the following section to resolve the 
continuous-time analogue of the problem; the key distinction in this case 
is that an appropriate measure-theoretic identity can in fact be 
established (Lemma \ref{lem:blum} below).
\end{rem}

\subsection{Continuous time}
\label{sec:conttime}

In the previous section we have developed a partial result on filter 
stability in the case of translation-invariant models with direct 
observation structure.  Unfortunately, that result concerns a quantity 
intermediate between the filter and prediction filter, and therefore 
falls short of resolving Conjecture \ref{conj:obs} even in the 
translation-invariant setting.  Surprisingly, however, it turns out that 
this problem can be resolved if we consider the natural \emph{continuous 
time} analogue of Conjecture \ref{conj:obs}, providing a complete proof 
of filter stability in the translation-invariant setting for 
continuous-time models with direct observations.  This idea will be developed
in the remainder of this section.

To define the continuous-time counterpart of the filtering model of 
Conjecture \ref{conj:obs}, we begin by considering a stationary Markov 
process $X=(X_t)_{t\in\mathbb{R}}$ with c\`adl\`ag paths with values in 
$\{-1,1\}^\mathbb{Z}$. We will assume that $X$ satisfies the Feller 
property, that is, that $x\mapsto\mathbf{E}^x[f(X_t)]$ is a quasilocal 
function for every $t$ and bounded quasilocal function $f$ (a function is 
called quasilocal if it is the uniform limit of functions that depend on a 
finite number of coordinates).\footnote{
	We recall that a function on $\{-1,1\}^{\mathbb{Z}}$ is quasilocal 
	if and only if it is continuous \cite[Remark 2.21]{Geo11}.
	Thus the definition given here is equivalent to the usual 
	definition of the Feller property of a Markov process.
} This mild 
condition ensures that the dynamics are local in a very weak sense (as 
compared to the much stronger local structure of the model in section 
\ref{sec:ihmm}, which was however not used in the previous section).
Markov processes of this type arise broadly in the literature on interacting
particle systems, cf.\ \cite{Lig05}.  To define the local 
observations $(Y_t)_{t\in\mathbb{R}}$, we introduce a `white noise' 
model of the form
$$
	dY_t^v = X_t^v\,dt+\sigma\,dW_t^v,\qquad
	Y_0^v=0,\qquad
	v\in\mathbb{Z},
$$
where $(W_t^v)_{t\in\mathbb{R}}$ are i.i.d.\ two-sided Brownian 
motions independent of $X$,
and $\sigma>0$ denotes the noise strength.  The process
$(X_t,Y_t)_{t\in\mathbb{R}}$ is a natural continuous-time analogue
of the model of Conjecture \ref{conj:obs}, and will be used in the
remainder of this section.

\begin{rem}
The details of the present model are not essential for our results.
The proof is easily extended to random fields on $\mathbb{Z}^d$ with
values in a finite state space, and to observation models other than
the usual white noise model (cf.\ \cite{vH09}).  For concreteness and
to avoid additional notation, we will work here in the simplest 
setting defined above.
\end{rem}

As in section \ref{sec:transinv}, we will further assume 
that the model is translation-invariant
$$
	(X_t^v,Y_t^v)_{t\in\mathbb{R},v\in\mathbb{Z}}
	\,\mathop{\stackrel{\mathrm{law}}{=}}\,
	(X_{t+s}^{v+w},Y_{t+s}^{v+w}-Y_{s}^{v+w}
	)_{t\in\mathbb{R},v\in\mathbb{Z}}
	\quad\mbox{for all }s\in\mathbb{R},~w\in\mathbb{Z}
$$
(note that, due to the additive nature of the observations, it is the
increments of $Y$ that are translation invariant and not $Y$ itself).
In this setting, we obtain the following result.

\begin{thm}
\label{thm:conttime}
For the model of this section, we have
$$
	|\mathbf{P}[X_t\in A|X_0,\{Y_s\}_{0\le s\le t}] -
	\mathbf{P}[X_t\in A|\{Y_s\}_{0\le s\le t}]|
	\xrightarrow{t\to\infty}0\quad\mbox{in }L^1
$$
for every measurable set $A$.
\end{thm}

This result evidently resolves, in the translation-invariant case, the 
continuous-time analogue of Conjecture \ref{conj:obs}.  We remark once 
more that stability of $X$ is not assumed.

To prove Theorem \ref{thm:conttime} we require a sharper version of the 
entropy identity of Conze that was used in the proof of Theorem 
\ref{thm:iobs} (cf.\ Remark \ref{rem:conze}).  The proof of the 
requisite identity, which we state presently, can be found in \cite[p.\ 
17, \S 8]{Con72}.

\begin{lem}[\cite{Con72}]
\label{lem:conze}
Let $(Z^q)_{q\in\mathbb{Z}^2}$ be a translation-invariant random field
where $Z^q = (\tilde X^v_k,\tilde Y^v_k)$ for $q=(k,v)\in\mathbb{Z}^2$
takes values in a finite set.  Then for any $n\ge 0$
\begin{align*}
	& H(\tilde Y_0^v|\tilde Y_0^{<v},\tilde Y_{<0}) -
	H(\tilde Y_0^v|\tilde Y_0^{<v},\tilde Y_{<0};
	\tilde X_{-n}^{<v},\tilde X_{<-n}) = \\
	&\qquad\qquad H(\tilde X_0^v|\tilde X_0^{<v},\tilde X_{<0};
	\tilde Y_{n}^v,\tilde Y_{n}^{<v},\tilde Y_{<n})-
	H(\tilde X_0^v|\tilde X_0^{<v},\tilde X_{<0};\tilde Y),
\end{align*}
where $\tilde Y_n^{<v}:=(\tilde Y_n^w)_{w<v}$, $\tilde Y_{<k}:=
(\tilde Y_n)_{n<k}$, and 
$\tilde Y:=(\tilde Y_n)_{n\in\mathbb{Z}}$.
\end{lem}

Let us note that the identity in Remark \ref{rem:conze} follows immediately
as $n\to\infty$.

The problem in continuous time is that we no longer have a discrete 
random field as in Lemma \ref{lem:conze}.  We address this by an 
appropriate discretization method, which yields the following
continuous-time counterpart of Proposition \ref{prop:pinsker}.

\begin{prop}
\label{prop:contpinsker}
Let $\delta>0$ and $v\in\mathbb{Z}$, $m\in\mathbb{N}$ be arbitrary, and define
\begin{align*}
	\rho_s &:= 
	\mathbf{P}[\{Y_t^v,\ldots,Y_t^{v+m}\}_{t\in[0,\delta]}\in\,\cdot\,|
	Y_{\le 0},\{Y_t^{<v}\}_{t\in[0,\delta]},
	X_{\le s}],\\
	\rho &:=
	\mathbf{P}[\{Y_t^v,\ldots,Y_t^{v+m}\}_{t\in[0,\delta]}\in\,\cdot\,|
	Y_{\le 0},\{Y_t^{<v}\}_{t\in[0,\delta]}],
\end{align*}
where $X_{\le s}:=(X_t)_{t\le s}$ and $Y_{\le 0}:=(Y_t)_{t\le 0}$.
Then
$$
	\mathbf{E}[D(\rho_s||\rho)]
	\xrightarrow{s\to-\infty}0,
$$
where $D(\mu||\nu)$ denotes relative entropy.
\end{prop}

\begin{proof}
Let us choose $v=0$ for simplicity.  The result for arbitrary $v$ follows
immediately by translation invariance.  In the following, we fix
$\delta>0$ and $m\in\mathbb{N}$.

Define the random field $Z=(Z^q)_{q\in\mathbb{Z}^2}$ with $Z^q=(\tilde 
X_k^r,\tilde Y_k^r)$ for $q=(k,r)\in\mathbb{Z}^2$ as
\begin{align*}
	\tilde X_k^r &:= 
	(X_{k\delta}^{r(m+1)},\cdots,X_{k\delta}^{r(m+1)+m}),\\	
	\tilde Y_k^r &:= 
	(Y_{k\delta+t}^{r(m+1)}-Y_{k\delta}^{r(m+1)},\cdots,
	Y_{k\delta+t}^{r(m+1)+m}-Y_{k\delta}^{r(m+1)+m}
	)_{t\in[0,\delta]}.
\end{align*}
Then evidently $Z$ is translation-invariant and $\tilde X_k^r$ is 
finite-valued, but $\tilde Y_k^r$ takes values in the space
$C_0([0,\delta];\mathbb{R}^{m+1})$ of continuous paths 
$\omega:[0,\delta]\to\mathbb{R}^{m+1}$ with $\omega(0)=0$ 
(which is Polish when endowed
with the topology of uniform convergence and the associated Borel
$\sigma$-field).  Thus we cannot directly apply Lemma \ref{lem:conze}.

To surmount this problem, we employ a straightforward discretization 
procedure.  Let $\{A_j\}_{j\ge 1}$ be a countable generating class for
the Borel $\sigma$-field $\mathcal{F}$ of 
$C_0([0,\delta];\mathbb{R}^{m+1})$, and define 
the functions $\kappa_j:C_0([0,\delta];\mathbb{R}^{m+1})\to
\{0,1\}^j$ as $\kappa_j:=(\mathbf{1}_{A_1},\ldots,\mathbf{1}_{A_j})$.
Then $\mathcal{F}_j:=\sigma\{\kappa_j\}$ is an increasing family of 
$\sigma$-fields such that $\bigvee_j\mathcal{F}_j=\mathcal{F}$.
Now define $\tilde Y_k^r(j) := \kappa_j(\tilde Y_k^r)$.  Then the random 
field $(\tilde X_k^r,\tilde Y_k^r(j))_{k,r\in\mathbb{Z}}$ is 
translation-invariant and finite-valued for every $j\ge 1$.  Thus we can 
apply Lemma \ref{lem:conze} to obtain
\begin{align*}
	& H(\tilde Y_0^0(j)|\tilde Y_0^{<0}(j),\tilde Y_{<0}(j)) -
	H(\tilde Y_0^0(j)|\tilde Y_0^{<0}(j),\tilde Y_{<0}(j);
	\tilde X_{-n}^{<0},\tilde X_{<-n}) = \\
	&\qquad\qquad H(\tilde X_0^0|\tilde X_0^{<0},\tilde X_{<0};
	\tilde Y_{n}^0(j),\tilde Y_{n}^{<0}(j),\tilde Y_{<n}(j))-
	H(\tilde X_0^0|\tilde X_0^{<0},\tilde X_{<0};\tilde Y(j)).
\end{align*}
In particular, as the left-hand side of this expression is an expected 
relative entropy (see the proof of Proposition \ref{prop:babyobs}), we 
can estimate
\begin{align*}
	& H(\tilde Y_0^0(i)|\tilde Y_0^{<0}(j),\tilde Y_{<0}(j)) -
	H(\tilde Y_0^0(i)|\tilde Y_0^{<0}(j),\tilde Y_{<0}(j);
	\tilde X_{-n}^{<0},\tilde X_{<-n}) \\
	&\qquad\qquad \mbox{}\le
	H(\tilde X_0^0|\tilde X_0^{<0},\tilde X_{<0};
	\tilde Y_{n}^0(j),\tilde Y_{n}^{<0}(j),\tilde Y_{<n}(j))-
	H(\tilde X_0^0|\tilde X_0^{<0},\tilde X_{<0};\tilde Y(j))
\end{align*}
for every $i\le j$.  Letting $j\to\infty$ and using that conditioning
reduces entropy gives
\begin{align*}
	& \sup_{i\ge 1}\,\{H(\tilde Y_0^0(i)|\tilde Y_0^{<0},\tilde Y_{<0}) 
	-
	H(\tilde Y_0^0(i)|\tilde Y_0^{<0},\tilde Y_{<0};
	\tilde X_{<-n})\} \\
	&\qquad\qquad \mbox{}\le
	H(\tilde X_0^0|\tilde X_0^{<0},\tilde X_{<0};
	\tilde Y_{n}^0,\tilde Y_{n}^{<0},\tilde Y_{<n})-
	H(\tilde X_0^0|\tilde X_0^{<0},\tilde X_{<0};\tilde Y).
\end{align*}
To proceed, we write the left hand side as an expected relative entropy 
as in the proof of Proposition \ref{prop:babyobs}. Using the continuity of 
the relative entropy in information (e.g., \cite[Lemma 4.4.15]{DS89}) and 
monotone convergence, this yields
\begin{align*}
	& 
	\mathbf{E}[\,D(\,
	\mathbf{P}[\tilde Y_0^0\in\cdot\,|
	\tilde Y_0^{<0},\tilde Y_{<0},\tilde X_{<-n}]
	\,||\,
	\mathbf{P}[\tilde Y_0^0\in\cdot\,|
	\tilde Y_0^{<0},\tilde Y_{<0}]
	\,)\,] \\
	&\qquad\qquad \mbox{}\le
	H(\tilde X_0^0|\tilde X_0^{<0},\tilde X_{<0};
	\tilde Y_{n}^0,\tilde Y_{n}^{<0},\tilde Y_{<n})-
	H(\tilde X_0^0|\tilde X_0^{<0},\tilde X_{<0};\tilde Y).
\end{align*}
It therefore follows immediately that
$$
	\mathbf{E}[\,D(\,
	\mathbf{P}[\tilde Y_0^0\in\cdot\,|
	\tilde Y_0^{<0},\tilde Y_{<0},\tilde X_{<-n}]
	\,||\,
	\mathbf{P}[\tilde Y_0^0\in\cdot\,|
	\tilde Y_0^{<0},\tilde Y_{<0}]
	\,)\,] \xrightarrow{n\to\infty}0.
$$
It remains to note that
$$
	\mathbf{P}[\tilde Y_0^0\in\cdot\,|
        \tilde Y_0^{<0},\tilde Y_{<0}] =
	\rho,\qquad
	\mathbf{P}[\tilde Y_0^0\in\cdot\,|
        \tilde Y_0^{<0},\tilde Y_{<0},\tilde X_{<-n}] =
	\rho_{-(n+1)\delta},
$$
where we have used the Markov property of $(X_t,Y_t)_{t\ge 0}$ to 
obtain the latter equality.  This establishes that
$\mathbf{E}[D(\rho_s||\rho)]\to 0$ along the subsequence
$s=-(n+1)\delta$, and therefore as $s\to-\infty$ as 
$\mathbf{E}[D(\rho_s||\rho)]$ is decreasing in $s$ (as is easily
verified by Jensen's inequality).
\end{proof}

We can now complete the proof of Theorem \ref{thm:conttime}.

\begin{proof}[Proof of Theorem \ref{thm:conttime}]
By translation-invariance and Lemma \ref{lem:finapprox}, it suffices
to show
$$
	|\mathbf{P}[(X_0^0,\ldots,X_0^m)\in A|X_{-t},\{Y_s\}_{-t\le s\le 0}] -
	 \mathbf{P}[(X_0^0,\ldots,X_0^m)\in A|\{Y_s\}_{-t\le s\le 0}]|
	\xrightarrow{t\to\infty}0\mbox{ in }L^1
$$
for every set $A$ and $m\ge 1$.  But by the martingale convergence 
theorem, and as $\{X_{s},Y_{s}\}_{s\le -t}$ is conditionally independent
of $\{X_s,Y_s\}_{s\ge -t}$ given $X_{-t},Y_{-t}$ by the Markov property,
it suffices to show for every set $A$ and $m\ge 1$ that
$$
	|\mathbf{P}[(X_0^0,\ldots,X_0^m)\in 
	A|Y_{\le 0},X_{\le -t}] -
	 \mathbf{P}[(X_0^0,\ldots,X_0^m)\in A|Y_{\le 0}]|
	\xrightarrow{t\to\infty}0\mbox{ in }L^1.
$$
By the martingale convergence theorem, this can be formulated equivalently
as
$$
	\mathbf{P}[(X_0^0,\ldots,X_0^m)\in 
	A|\textstyle{\bigcap}_t\sigma\{Y_{\le 0},X_{\le -t}\}] =
	\mathbf{P}[(X_0^0,\ldots,X_0^m)\in A|Y_{\le 0}].
$$
The key distinction between the continuous- and discrete-time settings
is the following.

\begin{lem}
\label{lem:blum}
We have
$$
	\bigcap_{\delta>0}
	\sigma\{Y_{\le 0},
	\{Y_s^{<0}\}_{s\in[0,\delta]},X_{\le -t}\}=
	\sigma\{Y_{\le 0},X_{\le -t}\}
	\quad\mathrm{mod}\,\mathbf{P}
$$
and
$$
	\bigcap_{\delta>0}\sigma\{Y_{\le 
	0},\{Y_s^{<0}\}_{s\in[0,\delta]}\}=
	\sigma\{Y_{\le 0}\}\quad\mathrm{mod}\,\mathbf{P}.
$$
\end{lem}

\begin{proof}
Let us prove the first statement (the second statement follows readily in 
the same manner).  It evidently suffices to prove the stronger identity
$$
	\bigcap_{\delta>0}
	\sigma\{Y_{\le \delta},X_{\le -t}\}=
	\sigma\{Y_{\le 0},X_{\le -t}\}
	\quad\mathrm{mod}\,\mathbf{P}.
$$
To this end, it suffices to show that
$$
	\mathbf{E}[Z|Y_{\le \delta},X_{\le -t}]
	\xrightarrow{\delta\downarrow 0}
	\mathbf{E}[Z|Y_{\le 0},X_{\le -t}]\quad\mbox{in }L^2
$$
for every bounded random variable $Z$.  By a standard approximation 
argument, it suffices to consider $Z$ of the form
$Z_s=f(X_{s+t_1},Y_{s+t_1},\ldots,X_{s+t_n},Y_{s+t_n})$ for 
$n\in\mathbb{N}$, $t_1,\ldots,t_n\in\mathbb{R}$, and $f$ bounded and 
local.  Now note that by stationarity,
$$
	\mathbf{E}[
	\mathbf{E}[Z_s
	|Y_{\le \delta},X_{\le -t}]^2] =
	\mathbf{E}[
	\mathbf{E}[
	Z_{s-\delta}|
	Y_{\le 0},X_{\le -t-\delta}]^2].
$$
As $X$ is Feller, it is quasi-left continuous \cite[p.\ 101]{RY99},
and therefore $X_{t-\delta}\to X_t$ a.s.\ as $\delta\downarrow 0$.
In particular, as $f$ is local and $Y$ is continuous by construction, we 
have $Z_{s-\delta}\to Z_s$ a.s.\ as $\delta\downarrow 0$.
On the other hand, as $f$ is bounded, we obtain
$$
	\mathbf{E}[
	\mathbf{E}[Z_s|Y_{\le \delta},X_{\le -t}]^2]
	\xrightarrow{\delta\downarrow 0}
	\mathbf{E}[
	\mathbf{E}[Z_s|
	Y_{\le 0},X_{<-t}]^2] =
	\mathbf{E}[
	\mathbf{E}[Z_s|
	Y_{\le 0},X_{\le -t}]^2],
$$
using Hunt's lemma \cite[Corollary II.2.4]{RY99},
where the last equality follows again by quasi-left continuity.
Thus $\mathbf{E}[Z_s|Y_{\le \delta},X_{\le -t}]\to
\mathbf{E}[Z_s|Y_{\le 0},X_{\le -t}]$ in $L^2$, completing the proof.
\end{proof}

By virtue of Lemma \ref{lem:blum}, to complete the proof of Theorem 
\ref{thm:conttime}, it suffices to show
\begin{multline*}
	\mathbf{P}[(X_0^0,\ldots,X_0^m)\in 
	A|\textstyle{\bigcap}_t\sigma\{Y_{\le 0},\{Y_s^{<0}\}_{s\in[0,\delta]},
	X_{\le -t}\}] = \\
	\mathbf{P}[(X_0^0,\ldots,X_0^m)\in A|Y_{\le 0},
	\{Y_s^{<0}\}_{s\in[0,\delta]}]
\end{multline*}
for every set $A$, $m\ge 1$, and $\delta>0$: indeed, letting
$\delta\downarrow 0$ the yields the expression before the
statement of Lemma \ref{lem:blum}.  We will deduce this fact from
Proposition \ref{prop:contpinsker}.  To this end we require 
a lemma that replaces the analogous argument in 
Proposition \ref{prop:babyobs}.

\begin{lem}
\label{lem:wiener}
Let $f:\mathbb{R}^{m+1}\to\mathbb{R}$ be a bounded continuous function.
Then there exists a sequence of bounded continuous functions
$g_n:\mathbb{R}^{m+1}\to\mathbb{R}$ such that
$$
	|\mathbf{E}[g_n(nY_{1/n}^0,\ldots,nY_{1/n}^m)|X]
	-f(X_0^0,\ldots,X_0^m)|\xrightarrow{n\to\infty}0
	\quad\mbox{in }L^1.
$$
\end{lem}

\begin{proof}
By the definition of the observations,
$$
	\mathbf{E}[g(nY_{1/n}^0,\ldots,nY_{1/n}^m)|X] =
	\textstyle{(g*\xi_{n})(n\int_0^{1/n}X_s^0\,ds,\ldots,
	n\int_0^{1/n}X_s^m\,ds)},
$$
where $\xi_{n}$ denotes the centered Gaussian measure on
$\mathbb{R}^{m+1}$ with covariance $n\,\mathrm{Id}$ and $*$ denotes
convolution.  Note the trivial estimate $|\int_0^tX_s^vds|\le t$, so the 
argument of the function $g*\xi_{n}$ above takes values in the compact
set $C=[-1,1]^{m+1}$.

We now recall that as $C$ is compact, every continuous function on $C$ 
is contained in the closure of 
$\{(g*\xi_{n})|_{C} :
g\in \mathcal{C}_b(\mathbb{R}^{m+1})\}$ with respect to the uniform
convergence topology on $C$ (here
$\mathcal{C}_b(\mathbb{R}^{n+1})$ is the family of bounded continuous 
functions on $\mathbb{R}^{n+1}$).  This follows from an elementary 
Hahn-Banach argument, cf.\ \cite[Remark 5]{vH09}.  We can therefore
choose, for each $n$, a bounded continuous function $g_n$ such that
$$
	|(g_n*\xi_{n})(x)-f(x)|\le
	1/n\quad\mbox{for all }x\in C.
$$
Then we evidently have
\begin{multline*}
	|\mathbf{E}[g_n(nY_{1/n}^0,\ldots,nY_{1/n}^m)|X]
	-f(X_0^0,\ldots,X_0^m)| \\ \mbox{}\le 
	1/n + 
	|\textstyle{f(n\int_0^{1/n}X_s^0\,ds,\ldots,
        n\int_0^{1/n}X_s^m\,ds)}-f(X_0^0,\ldots,X_0^m)|.
\end{multline*}
As $f$ is bounded and continuous, and as the paths of $X$ are 
right-continuous, this expression
converges to zero as $n\to\infty$ a.s.\ and in $L^1$.
\end{proof}

Note that, by the definition of our model, $Y_{\le 
0},\{Y_s^{<0}\}_{s\in[0,\delta]}$ is conditionally independent
of $\{Y_s^{0},\ldots,Y_s^m\}_{s\in[0,\delta]}$ given $X$, so that
for every bounded continuous function $h$
$$
	\mathbf{E}[h(Y_{1/n}^0,\ldots,Y_{1/n}^m)|
	Y_{\le 0},\{Y_s^{<0}\}_{s\in[0,\delta]},X] =
	\mathbf{E}[h(Y_{1/n}^0,\ldots,Y_{1/n}^m)|X].
$$
In particular, in view of Lemma \ref{lem:wiener}, it now suffices
to show that
$$
	\mathbf{E}[h(Y_s^0,\ldots,Y_s^m)|
	\textstyle{\bigcap}_t\sigma\{Y_{\le 0},\{Y_s^{<0}\}_{s\in[0,\delta]},
	X_{\le -t}\}] = 
	\mathbf{E}[h(Y_s^0,\ldots,Y_s^m)|Y_{\le 0},
	\{Y_s^{<0}\}_{s\in[0,\delta]}]
$$
for every bounded continuous function $h$ and $s\in[0,\delta]$.
But this follows readily from Proposition \ref{prop:contpinsker} using
Pinsker's inequality and martingale convergence.
\end{proof}

\section{Conditional random fields}
\label{sec:condrf}

Thus far we have considered infinite-dimensional counterparts of classical 
stability problems in nonlinear filtering.  However, new questions arise 
in infinite dimension beyond stability that are of interest in their own 
right.  In particular, it is of 
significant interest (cf.\ \cite{RvH13}) to understand the spatial mixing 
and decay of correlations properties of conditional distributions in 
infinite dimension, which could be viewed as spatial counterparts to the 
filter stability property.  Such questions already arise in the absence of 
dynamics, and thus we proceed in this section to introduce such problems 
in the most basic setting of conditional random fields (that is, in models 
with only spatial degrees of freedom).  Our motivations 
for such questions are threefold:
\begin{enumerate}
\item Random fields provide the simplest possible setting to investigate 
the spatial mixing properties of conditional distributions.
\item Conditional random fields are of practical interest in their own 
right, for example, in Bayesian image analysis applications 
\cite{Win03,FMS97}.
\item Even in the more classical setting of the previous sections, the 
random field viewpoint proves to be fundamental to the understanding of 
filter stability in infinite dimension: indeed, the proofs 
in both sections \ref{sec:phtrans} and \ref{sec:obs} above and in
\cite{RvH13,RvH13b} exploit the idea that 
$(X_k^v,Y_k^v)_{k\in\mathbb{Z},v\in\mathbb{Z}^d}$ can be viewed as a 
space-time random field.
\end{enumerate}
The remainder of this section is organized as follows.  In section 
\ref{sec:mrf}, we recall some basic notions from the theory of Markov 
random fields.  In section \ref{sec:crf}, we develop basic properties of 
conditional random fields and introduce some of the relevant questions.
Finally, in section \ref{sec:mono}, we develop a general result that 
ensures the inheritance of ergodicity under conditioning in random fields 
that possess certain monotonicity properties.  The latter provides a 
mechanism for the resolution of the random field counterpart of Conjecture 
\ref{conj:obs} that is quite distinct from the observability theory of 
section \ref{sec:obs}.

\subsection{Markov random fields}
\label{sec:mrf}

A random field is a collection of random variables $X_v$ that are 
indexed by the spatial degree of freedom $v$. For simplicity, we will 
assume in the sequel that $v\in\mathbb{Z}^d$ (but see section 
\ref{sec:graphs} below) and that each $X_v$ takes values in a finite set 
$E$.

In the following, we define for any $V\subseteq\mathbb{Z}^d$
$$
	V^c:=\mathbb{Z}^d\backslash V,
	\qquad
	\partial V:=\{w\in V^c:\|v-w\|=1\mbox{ for some }v\in V\},
	\qquad
	X_V:=(X_v)_{v\in V}.
$$
If $V$ is a finite subset of $\mathbb{Z}^d$, we will write 
$V\subset\subset\mathbb{Z}^d$.  We now recall a basic definition.

\begin{defn}
\label{defn:mrf}
$X=(X_v)_{v\in\mathbb{Z}^d}$ is called a \emph{Markov random field} if it 
possesses the (local) Markov property, that is, 
$\mathbf{P}[X_V\in\cdot|X_{V^c}]$ depends only on $X_{\partial 
V}$ for every $V\subset\subset\mathbb{Z}^d$.
\end{defn}

Just as Markov chains are defined by transition probabilities, Markov 
random fields are defined by a family of local transition kernels called a 
\emph{specification} \cite[Chapter 1]{Geo11}.

\begin{defn}
\label{defn:spec}
A family $\gamma=(\gamma_V)_{V\subset\subset\mathbb{Z}^d}$ of
transition kernels on $E^{\mathbb{Z}^d}$ such that
\begin{enumerate}
\item $\gamma_V(x,A)$ is a function of $x_{\partial V}$ for every
$A\in\sigma\{X_{V}\}$ and $V\subset\subset\mathbb{Z}^d$,
\item $\gamma_V(x,A)=\mathbf{1}_A(x)$ for every
$A\in\sigma\{X_{V^c}\}$ and $V\subset\subset
\mathbb{Z}^d$,
\item $\gamma_V\gamma_W=\gamma_V$ for every $W\subset 
V\subset\subset\mathbb{Z}^d$,
\end{enumerate}
is called a \emph{specification}.  A Markov random field 
$X$ is said to be \emph{specified} by $\gamma$ if we have
$\mathbf{P}(X\in A|X_{V^c})=\gamma_V(X,A)$ for every measurable set $A$ 
and $V\subset\subset\mathbb{Z}^d$.  The family of all laws of Markov 
random fields specified by $\gamma$ is denoted $\mathscr{G}(\gamma)$.
\end{defn}

\begin{ex}
\label{ex:ising}
Standard constructions of Markov random fields arise in statistical 
mechanics in the following manner.  Let $\psi_v:E\to\mathbb{R}$
and $\varphi_{\{v,w\}}:E\times E\to\mathbb{R}$ for $v,w\in\mathbb{Z}^d$ 
with $\|v-w\|=1$ be given potential functions, and let
$$
	\gamma_V(x,A) = 
	\frac{1}{Z}\sum_{x_V\in E^V}\mathbf{1}_A(x)\, \exp\Bigg(
	\sum_{\{v,w\}\subset V\cup\partial V
	:\|v-w\|=1}\varphi_{\{v,w\}}(x_v,x_w) +
	\sum_{v\in V}\psi_v(x_v)
	\Bigg)
$$
where $Z$ is the appropriate normalization factor.  It is easily 
verified that $\gamma=(\gamma_V)_{V\subset\subset\mathbb{Z}^d}$ defines a 
specification.  The potentials $\psi_v$ and 
$\varphi_{\{v,w\}}$ describe the local 
external and interaction forces between different sites, and are defined 
directly in terms of the physical parameters of the problem.  For example, 
if $E=\{-1,1\}$, $\varphi_{\{v,w\}}(\sigma,\sigma')=\beta J\sigma\sigma'$, 
and $\psi_v(\sigma)=\beta\mu\sigma$ with $\beta,J>0$ and $\mu\in\mathbb{R}$, 
this is the well known ferromagnetic Ising model with inverse temperature 
$\beta$, interaction strength $J$ and magnetic field strength $\mu$.
The construction in terms of potentials will be inessential in the sequel, 
however.
\end{ex}

Given a specification $\gamma$, there always exists a random field in 
$\mathscr{G}(\gamma)$ under our assumptions.  However, just as a Markov 
chain with given transition probabilities may admit more than one 
stationary distribution, the random field associated to a given 
specification need not be unique.  In fact, the structure of the set 
$\mathscr{G}(\gamma)$ is closely related to the spatial mixing properties 
of the associated random fields, as is shown by the following result 
\cite[section 4.4, Proposition 7.11, Theorem 7.7]{Geo11}.  To interpret 
the notion of extremality that arises here, note that if $\mathbf{P}$ and 
$\mathbf{Q}$ are the laws of two random fields in $\mathscr{G}(\gamma)$, 
then $\lambda\mathbf{P}+(1-\lambda)\mathbf{Q}$ is also in 
$\mathscr{G}(\gamma)$ for $0\le\lambda\le 1$ \cite[Chapter 7]{Geo11}; thus 
$\mathscr{G}(\gamma)$ is a convex set, and a random field is called 
\emph{extremal} if it is an extreme point of this set.

\begin{thm}
\label{thm:gibbs}
For a given specification $\gamma$, the following hold.
\begin{enumerate}
\item Existence of a random field: $\mathscr{G}(\gamma)\ne\varnothing$.
\item Uniqueness $\Leftrightarrow$ uniform mixing: 
$|\mathscr{G}(\gamma)|=1$ iff a random field in $\mathscr{G}(\gamma)$
satisfies\footnote{
	Here we used the suggestive notation
	$\mathbf{P}[X\in C|X_{W^c}=x_{W^c}]:=\gamma_W(x,C)$ to emphasize
	the significance of the mixing property.  Note that
	$\mathbf{P}[X\in C|X_{W^c}]=\gamma_W(X,C)$ holds a.s.\ by the
	definition of $\mathscr{G}(\gamma)$, but the equivalence between
	uniqueness and uniform mixing is false if a null set is omitted in 
	the supremum over $x$.
}$^,$\footnote{
	The notation $\lim_{W}a_W$ denotes
	the limit of the net $\{a_W\}$,
	where $\{W\subset\subset\mathbb{Z}^d\}$ is directed by inclusion.
}
$$
	\lim_{W\subset\subset\mathbb{Z}^d}
	\sup_x
	|\mathbf{P}[X_V\in A|X_{W^c}=x_{W^c}]-\mathbf{P}[X_V\in A]|=0
$$
for every set $A$ and $V\subset\subset\mathbb{Z}^d$.
\item Extremality $\Leftrightarrow$ mixing: the random field $X$ is an 
extreme point of $\mathscr{G}(\gamma)$ iff
$$
	\lim_{W\subset\subset\mathbb{Z}^d}\mathbf{E}|
	\mathbf{P}[X_V\in A|X_{W^c}]-\mathbf{P}[X_V\in A]|=0
$$
for every set $A$ and $V\subset\subset\mathbb{Z}^d$.
\end{enumerate}
\end{thm}

The mixing property in Theorem \ref{thm:gibbs} is a direct spatial 
analogue of the stability property of a Markov chain introduced in section 
\ref{sec:filt}.  Indeed, a Markov chain is stable if it forgets its 
initial condition after a long time: that is, the Markov chain has a 
`finite memory.' Similarly, a random field is mixing if the distribution 
of any finite set of sites $V$ is insensitive to knowledge of the 
configuration of the field outside a larger set $W$ when the distance 
between $V$ and $W^c$ is large.  This implies in particular that distant 
sites are nearly independent, that is, the field has `finite correlation 
length.' The \emph{uniform} mixing property is a strictly stronger notion, 
where the forgetting property holds uniformly in the boundary 
configuration $x_{\partial W}$ (recall that by the Markov property of the 
random field, $\mathbf{P}[X\in C|X_{W^c}=x_{W^c}]$ depends on $x_{\partial 
W}$ only).

\subsection{Conditional random fields and conditional mixing}
\label{sec:crf}

In the following, let us fix a specification $\gamma$ and a Markov random 
field $X=(X_v)_{v\in\mathbb{Z}^d}$ that is specified by $\gamma$.  In 
order to investigate the conditional distributions of random fields, we 
must introduce a suitable observation structure.  To this end, in analogy 
with section \ref{sec:ihmm}, let us fix for each $v\in\mathbb{Z}^d$ a 
transition kernel $\Phi_v$ from the state space $E$ of the random field to 
a measurable space $F$ in which the observations take their values.  We 
now construct the observations $Y=(Y_v)_{v\in\mathbb{Z}^d}$ such that
$$
	\mathbf{P}[Y\in dy|X] = \prod_{v\in\mathbb{Z}^d}
	\Phi_v(X_v,dy_v);
$$
that is, each site of the underlying field is observed independently with
$\mathbf{P}[Y_v\in A|X_v]=\Phi_v(X_v,A)$. 
The resulting model $(X_v,Y_v)_{v\in\mathbb{Z}^d}$ is called a 
\emph{hidden Markov random field}. 

\begin{rem}
\label{rem:edge}
For notational simplicity, we have formulated our model such that the 
observations are attached to individual sites $v\in\mathbb{Z}^d$.  One 
could also consider more general models, for example, where an 
observation $Y_{\{v,w\}}$ is attached to every edge 
$\{v,w\}\subset\mathbb{Z}^d$, $\|v-w\|=1$ with $\mathbf{P}[Y_{\{v,w\}}\in 
A|X]=\Phi_{\{v,w\}}(X_v,X_w,A)$ (cf.\ Example \ref{ex:rfbla}).  The
results of this section will continue to hold in this setting with minor 
modifications.
\end{rem}

We can now formulate the natural counterpart of the filter stability 
property in hidden Markov random fields: the model is said to be 
conditionally mixing if the conditional distribution of the underlying 
process in a finite set of sites given the observations is insensitive to 
knowledge of the configuration of the field at distant sites.

\begin{defn}
The hidden Markov random field $(X_v,Y_v)_{v\in\mathbb{Z}^d}$ is 
\emph{conditionally mixing} if
$$
	\lim_{W\subset\subset\mathbb{Z}^d}\mathbf{E}|
	\mathbf{P}[X_V\in A|X_{W^c},Y]-\mathbf{P}[X_V\in A|Y]|=0
$$
for every set $A$ and $V\subset\subset\mathbb{Z}^d$.
\end{defn}

The basic question to be addressed in this setting is therefore: 
\emph{when is the mixing property inherited by conditioning}, that is, 
when does the mixing property of the random field $X$ imply the 
conditional mixing property of $(X,Y)$? 

It will be insightful to reformulate the problem in different terms. For 
simplicity, we will assume in the sequel that the observations are locally 
nondegenerate, that is, that 
$\Phi_v(x_v,dy_v)=g_v(x_v,y_v)\,\varphi(dy_v)$ for some positive density 
$g_v(x_v,y_v)>0$ for all $x_v,y_v$ (the reference measure $\varphi(dy_v)$ 
on $F$ may be any $\sigma$-finite measure.)

\begin{prop}
\label{prop:cgibbs}
Define for every $y\in F^{\mathbb{Z}^d}$ and
$V\subset\subset\mathbb{Z}^d$ the transition kernel on $E^{\mathbb{Z}^d}$
$$
	\gamma_V^y(x,A) = \frac{
	\int \mathbf{1}_A(z)\prod_{v\in V}g_v(z_v,y_v)\,\gamma_V(x,dz)
	}{
	\int \prod_{v\in V}g_v(z_v,y_v)\,\gamma_V(x,dz)
	}.
$$
Then the following hold.
\begin{enumerate}
\item $\gamma^y=(\gamma^y_V)_{V\subset\subset\mathbb{Z}^d}$ is a 
specification for every $y\in\mathbb{Z}^d$.
\item $\mathbf{P}[X\in\,\cdot\,|Y]$ is in $\mathscr{G}(\gamma^Y)$ a.s.
\item $(X,Y)$ is conditionally mixing iff $\mathbf{P}[X\in\,\cdot\,|Y]$ is
extremal in $\mathscr{G}(\gamma^Y)$ a.s.
\end{enumerate}
\end{prop}

\begin{proof}
We begin by verifying that $\gamma^y$ is a specification.  To this end,  
let $W\subset V\subset\subset\mathbb{Z}^d$.  As 
$\gamma_V\gamma_W=\gamma_V$ and $\gamma_W(fg)=g\,\gamma_Wf$ if $g(x)$ 
depends only on $x_{W^c}$, we can write
\begin{align*}
	&\int \mathbf{1}_A(z)\prod_{v\in V}g_v(z_v,y_v)\,\gamma_V(x,dz) \\
	&=
	\int \gamma_W^y(z',A)
	\int\prod_{w\in W}g_w(z_w,y_w)\,\gamma_W(z',dz)
	\prod_{v\in V\backslash W}g_v(z_v',y_v)\,\gamma_V(x,dz') \\
	&=
	\int \gamma_W^y(z,A)\prod_{v\in V}g_v(z_v,y_v)\,\gamma_V(x,dz).
\end{align*}
Thus $\gamma_V^y\gamma_W^y=\gamma_V^y$, and the remaining 
properties of a specification hold trivially.

Next, we show that $\mathbf{P}[X\in\,\cdot\,|Y]$ is in 
$\mathscr{G}(\gamma^Y)$ a.s.  To this end, let us fix any regular 
version $\mathbf{P}^Y$ of the conditional distribution 
$\mathbf{P}[\,\cdot\,|Y]$.  We must show that for a.e.\ observation record 
$y$, we have $\mathbf{P}^y[X\in A|X_{V^c}]=\gamma_V^y(X,A)$ for all $A$, 
that is, we must show that
$$
	\mathbf{E}^y[\gamma_V^y(X,A)\mathbf{1}_B]=\mathbf{P}^y[\{X\in A\}\cap B]
	\quad\mbox{for every measurable }A\mbox{ and }B\in\sigma\{X_{V^c}\}	
$$
holds for $\mathbf{P}$-a.e.\ $y$.  Is easily seen by the definition 
of a hidden Markov random field that
$$
	\gamma_V^Y(X,A)=\mathbf{P}[X\in A|X_{V^c},Y].
$$
We therefore have
$$
	\mathbf{E}[\gamma_V^Y(X,A)\mathbf{1}_B\mathbf{1}_C]
	=\mathbf{P}[\{X\in A\}\cap B\cap C]
$$
for every $A$ and $B\in\sigma\{X_{V^c}\}$, $C\in\sigma\{Y\}$.  It 
follows by disintegration that
$$
	\mathbf{E}^Y[\gamma_V^Y(X,A)\mathbf{1}_B]=\mathbf{P}^Y[\{X\in A\}\cap B]
$$
holds $\mathbf{P}$-a.s.\ for a fixed choice of $A$, 
$B\in\sigma\{X_{V^c}\}$, and thus simultaneously for a countable family of 
sets $A$ and $B\in\sigma\{X_{V^c}\}$.  By choosing the countable family to 
be a generating class (note that all our $\sigma$-fields are countably 
generated), the above identity holds simultaneously for every $A$ and 
$B\in\sigma\{X_{V^c}\}$ by a monotone class argument. As there are only 
countably many $V\subset\subset\mathbb{Z}^d$, we have proved that 
$\mathbf{P}[X\in\,\cdot\,|Y]$ is in $\mathscr{G}(\gamma^Y)$ a.s.

Finally, we consider the conditional mixing property.  As the limit in the 
definition of (conditional) mixing is over a decreasing net (by Jensen's 
inequality), it suffices to consider the limit along any fixed cofinal 
increasing sequence $W_n\subset\subset\mathbb{Z}^d$.  Thus by the 
martingale convergence 
theorem, the conditional mixing property holds if and only if
$$
	\lim_{n\to\infty}
	\mathbf{E}[~|
	\mathbf{P}[X\in A|X_{W_n^c},Y]-\mathbf{P}[X\in A|Y]|~|Y]=0
	\quad\mbox{a.s.}	
$$
for every $V\subset\subset\mathbb{Z}^d$ and $A\in\sigma\{X_V\}$.
As we have shown that $\mathbf{P}[X\in A|X_{W_n^c},Y]=
\gamma_{W_n}^Y(X,A)=\mathbf{P}^Y[X\in A|X_{W_n^c}]$, the conditional 
mixing property is equivalent to
$$
	\lim_{n\to\infty}
	\mathbf{E}^y|
	\mathbf{P}^y[X\in A|X_{W_n^c}]-\mathbf{P}^y[X\in A]|=0
	\quad\mbox{for }\mathbf{P}\mbox{-a.e.\ }y	
$$
for every $V\subset\subset\mathbb{Z}^d$ and $A\in\sigma\{X_V\}$.
But by the martingale convergence theorem
$$
	\lim_{n\to\infty}
	\mathbf{E}^y|
	\mathbf{P}^y[X\in A|X_{W_n^c}]-\mathbf{P}^y[X\in A]|=
	\mathbf{E}^y|
	\mathbf{P}^y[X\in A|\textstyle{\bigcap_n}
	\sigma\{X_{W_n^c}\}]-\mathbf{P}^y[X\in A]|.
$$
Thus we can again use a monotone class argument as above to remove
the dependence of the $\mathbf{P}$-null set on $V$ and $A$.  Thus
$(X_v,Y_v)_{v\in\mathbb{Z}^d}$ is conditionally mixing if and only if
$$
	\lim_{W\subset\subset\mathbb{Z}^d}
	\mathbf{E}^y|
	\mathbf{P}^y[X\in A|X_{W^c}]-\mathbf{P}^y[X\in A]|=0
	\quad\mbox{for every }V\subset\subset\mathbb{Z}^d,~
	A\in\sigma\{X_V\}
$$
holds for $\mathbf{P}$-a.e.\ $y$, which is precisely the mixing
property of $\mathbf{P}[X\in\,\cdot\,|Y]$.
\end{proof}

Proposition \ref{prop:cgibbs} shows that the conditional distribution
$\mathbf{P}[X\in\,\cdot\,|Y]$ defines again a (random) Markov random 
field, and gives an explicit expression for its specification $\gamma^Y$.  
The inheritance of ergodicity can now be formulated in terms of the 
ergodic properties of the conditional field.  In particular, we can pose 
two natural questions:
\begin{enumerate}
\item If $\mathbf{P}[X\in\,\cdot\,]$ is extremal in $\mathscr{G}(\gamma)$,
when is $\mathbf{P}[X\in\,\cdot\,|Y]$ extremal in $\mathscr{G}(\gamma^Y)$ 
a.s.?
\item If $|\mathscr{G}(\gamma)|=1$, when is $|\mathscr{G}(\gamma^Y)|=1$ 
a.s.?
\end{enumerate}
The first question is evidently the direct spatial analogue of the filter 
stability problem: when is the mixing property inherited by the 
conditional distribution?  The second question is analogous, but for the 
uniform mixing property.  It is evident from Theorem \ref{thm:gibbs} that
$|\mathscr{G}(\gamma^Y)|=1$ a.s.\ implies the conditional mixing property.
The stronger conclusion $|\mathscr{G}(\gamma^Y)|=1$ a.s.\ is perhaps less 
natural from the point of view of conditional distributions, but is of 
practical relevance in its own right as it is closely connected with 
the computational complexity of MCMC methods for Bayesian image analysis 
\cite{FMS97}.

As in the filter stability problem, local nondegeneracy of the 
observations does not suffice to obtain an affirmative answer to either of 
the above questions.  In fact, we have a direct analogue of the example 
given in section \ref{sec:phtrans}.

\begin{ex}
\label{ex:rfbla}
Let $E=F=\{-1,1\}$, and define the random field $(X_v)_{v\in\mathbb{Z}^2}$
such that $X_v$ are i.i.d.\ symmetric Bernoulli random variables.  It is 
evident that this model is uniformly mixing in the most trivial sense 
(thus uniqueness and extremality both hold).  

We now attach an observation $Y_{\{v,w\}}$ to each edge 
$\{v,w\}\subset\mathbb{Z}^d$, $\|v-w\|=1$ by setting 
$Y_{\{v,w\}}=X_vX_w\xi_{\{v,w\}}$ with $\xi_{\{v,w\}}$ i.i.d.\ and 
independent of $X$ with $\mathbf{P}[\xi_{\{v,w\}}=-1]=p$.  In this manner, 
we evidently obtain a direct counterpart of the model of section 
\ref{sec:phtrans}.  While the observations in this model are defined on 
the edges rather than on the vertices as we have done in this section, a 
result that is entirely analogous to Proposition \ref{prop:cgibbs} holds 
in this setting (see also Remark \ref{rem:edge} above and Remark 
\ref{rem:sillyedge} below).

We can now proceed identically as in the proof of Theorem 
\ref{thm:phtrans} to show that there exists $0<p_\star<1/2$ such that the 
hidden Markov random field $(X,Y)$ fails to be conditionally mixing for 
$p<p_\star$.  In fact, this is precisely the idea behind the proof of 
Theorem \ref{thm:phtrans} in the first place: the model 
$(X_k^v,Y_k^v)_{k,v\in\mathbb{Z}}$ is considered as a space-time random 
field, and the problem is addressed using classical methods from 
statistical mechanics.

The present example could be considered as a toy model in image analysis. 
The underlying field $X$ represents a grid of black or white pixels of an 
image, and the observations $Y$ correspond to noisy measurements of the 
gradient of the image at each point.  Thus we see that the ability to 
reconstruct the image based on the noisy gradient information undergoes a 
phase transition at a positive signal-to-noise ratio.
\end{ex}

\begin{rem}
\label{rem:sillyedge}
The use of edge observations in Example \ref{ex:rfbla} is merely cosmetic: 
the same example can be reformulated in terms of vertex observations. 
Indeed, let us define the random field $(\tilde X_v,\tilde 
Y_v)_{v\in\mathbb{Z}^d}$ with $\tilde 
X_v\in\{-1,1\}^3$ and $\tilde Y_v\in\{-1,1\}^2$ by setting $\tilde 
X_v=(X_v,X_{v+(0,1)},X_{v+(1,0)})$ and $\tilde 
Y_v=(X_vX_{v+(0,1)}\xi_{\{v,v+(0,1)\}},X_vX_{v+(1,0)}\xi_{\{v,v+(1,0)\}})$, 
where $X_v$ and $\xi_{\{v,w\}}$ are as in Example \ref{ex:rfbla}.  Then 
$\tilde X$ is still a uniformly mixing Markov random field, the 
observations $\tilde Y$ are locally nondegenerate, and 
$\mathbf{P}[\tilde X^1\in\,\cdot\,|\tilde Y] = \mathbf{P}[X\in\,\cdot\,|Y]$.
In particular, the above conditional phase transition 
arises identically in this formulation.
\end{rem}

In view of the above, the inheritance of mixing properties of random 
fields under conditioning cannot be taken for granted.  Just as in the 
filter stability problem, however, it is natural to expect that 
conditional mixing will hold in the absence of observation symmetries.  
Such a conjecture is often implicit in work on Bayesian image analysis 
(cf.\ \cite[p.\ 6]{FMS97}).  For example, we can formulate the natural 
analogue of Conjecture \ref{conj:obs}.

\begin{conj}
\label{conj:rfobs}
Let $(X_v,Y_v)_{v\in\mathbb{Z}^2}$ be a hidden Markov field
with $E=F=\{-1,1\}$ and
$$
	Y_v = X_v\xi_v,\qquad
	(\xi_v)_{v\in\mathbb{Z}^2}\mbox{ are i.i.d.}\independent X
	\mbox{ with } \mathbf{P}[\xi_v=-1]=p.
$$
If the underlying random field $X$ is mixing, then the model is
conditionally mixing.
\end{conj}

We do not know how to prove such a conjecture in a general setting.  
However, we will presently establish the validity of such a result under 
monotonicity assumptions on the underlying field.  This provides an 
entirely different mechanism for the inheritance of ergodicity than
the observability theory that was developed in section \ref{sec:obs} 
above.

\subsection{Monotonicity}
\label{sec:mono}

The goal of this section is to prove a variant of Conjecture 
\ref{conj:rfobs} under monotonicity assumptions on the underlying field.  
In such models, the direct observation structure $Y_v = X_v\xi_v$ ensures 
that monotonicity properties are preserved under conditioning on the 
observations, which greatly facilitates the analysis of the conditional 
random field.  The arguments used here are directly inspired by the 
methods used in \cite{Gol80, Fol80, FP97} to investigate the \emph{global} 
Markov property of random fields.

In this section, we will assume that $E=F=\{-1,1\}$.  Fix a specification 
$\gamma$ and a Markov random field $X=(X_v)_{v\in\mathbb{Z}^d}$ specified 
by $\gamma$, and introduce the observations 
$$
	Y_v = X_v\xi_v,\qquad
	(\xi_v)_{v\in\mathbb{Z}^2}\mbox{ are independent}
	\independent X \mbox{ with } 
	\mathbf{P}[\xi_v=-1]=p_v.
$$
For simplicity, we will assume throughout this section that $0<p_v\le 1/2$ for all 
$v$.\footnote{
	The assumption $p_v>0$ ensures local nondegeneracy, which is 
	convenient in view of Proposition \ref{prop:cgibbs} but is not 
	essential for the results in this section.  The assumption 
	$p_v\le 1/2$ does not entail any loss of generality: if $p_v>1/2$, 
	we can reduce to $p_v<1/2$ by considering the inverted 
	observation $-Y_v$ instead of $Y_v$.}

The following monotonicity property is the key assumption of this section.
This property has various useful characterizations, cf.\ 
\cite[Theorem 2.27]{Gri06}.  Here a function 
$f:\{-1,1\}^{\mathbb{Z}^d}\to\mathbb{R}$ is called increasing if
$f(x)\ge f(z)$ whenever $x_v\ge z_v$ for all $v\in\mathbb{Z}^d$.

\begin{defn}
\label{defn:mono}
A specification $\gamma$ for a $\{-1,1\}$-valued random field is called
\emph{monotone} if for every bounded increasing function $f$ and
$V\subset\subset\mathbb{Z}^d$, the function $\gamma_Vf$ is increasing.
\end{defn}

We can now formulate the main result of this section.

\begin{thm}
\label{thm:mono}
For the model of this section, suppose that the specification $\gamma$ is 
monotone.  Then $|\mathscr{G}(\gamma)|=1$ implies that
$|\mathscr{G}(\gamma^Y)|=1$ a.s.  In particular, under the monotonicity 
assumption, uniqueness of the underlying random field implies conditional 
mixing.
\end{thm}

Under the monotonicity assumption, Theorem \ref{thm:mono} essentially 
resolves the analogue of Conjecture \ref{conj:rfobs} for uniqueness rather 
than extremality (this is the special case $p_v=p\ne 1/2$ for all $v$).  A 
partial result on the inheritance of extremality when the field is not 
unique can be deduced from the proof as well, see Remark \ref{rem:extmono} 
below.

Nonetheless, the result of Theorem \ref{thm:mono} is arguably quite 
different in spirit from Conjecture \ref{conj:rfobs}, as it does not 
require the absence of observation symmetries.  For example, we could 
choose the noise parameters $p_v$ such that $p_v=\frac{1}{2}$ for 
alternating sites in the lattice $\mathbb{Z}^d$: then there is certainly a 
large class of observation symmetries, as only every other site in the 
lattice is observed.  Thus Theorem \ref{thm:mono} is not addressing a 
symmetry-breaking phenomenon of the type that motivated the observability 
conjecture in section \ref{sec:obs}.  On the other hand, as will 
become clear in the proof of Theorem \ref{thm:mono}, the present approach 
directly implements the idea that ergodicity is inherited from the 
underlying model to the conditional distribution, in contrast with the 
observability theory of section \ref{sec:obs} which does not exploit at 
all the ergodic properties of the model.

\begin{rem}
It would be interesting to obtain a counterpart to Theorem \ref{thm:mono} 
for the filter stability problem.  Unfortunately, this does not appear to 
be possible, at least in the setting of section \ref{sec:ihmm}.  The proof 
of Theorem \ref{thm:mono} relies on the fact that the conditional 
distributions of the random field given the observations are monotone, 
which is essentially equivalent to the validity of Definition 
\ref{defn:mono}.  However, in the filtering model of section 
\ref{sec:ihmm}, the associated space-time random field must generally fail 
to be monotone in the sense of Definition \ref{defn:mono}, cf.\ 
\cite{Lig06} for a discussion in the continuous time setting.  While the 
space-time distributions of interacting particle systems with monotone 
transition probabilities do satisfy some weaker monotonicity properties, 
cf.\ \cite[section 3.1]{LMS90} or \cite{Lig06}, such weaker properties do 
not suffice for the proof of Theorem \ref{thm:mono}. 
\end{rem}

We now turn to the proof of Theorem \ref{thm:mono}.  We begin by 
establishing the inheritance of monotonicity by the conditional 
specification $\gamma^y$.

\begin{lem}
\label{lem:inhmono}
Suppose that $\gamma$ is monotone.  Then the following hold:
\begin{enumerate}
\item $\gamma^y$ is monotone for every $y\in\{-1,1\}^{\mathbb{Z}^d}$.
\item $y\mapsto \gamma^y_Vf$ is increasing for every
$V\subset\subset\mathbb{Z}^d$ and bounded increasing $f$.
\end{enumerate}
\end{lem}

\begin{proof}
It will be convenient to write 
$$
	g_v(x_v,y_v)=\mathbf{P}[Y_v=y_v|X_v=x_v]=\sqrt{p_v(1-p_v)}\,e^{\beta_v 
	y_vx_v}
$$
with $\beta_v=\log\sqrt{(1-p_v)/p_v}\ge 0$ in Proposition \ref{prop:cgibbs}.

Define for every $y\in F^{\mathbb{Z}^d}$ and
$W,V\subset\subset\mathbb{Z}^d$ the transition kernel on 
$E^{\mathbb{Z}^d}$
$$
	\gamma_{V,W}^y(x,A) = \frac{
	\int \mathbf{1}_A(z)\prod_{w\in W\cap V}e^{\beta_wy_wz_w}\,\gamma_V(x,dz)
	}{
	\int \prod_{w\in W\cap V}e^{\beta_wy_wz_w}\,\gamma_V(x,dz)
	}.
$$
Evidently $\gamma_V^y=\gamma_{V,V}^y$ and 
$\gamma_V=\gamma_{V,\varnothing}^y$, and for simplicity we will write
$\gamma_{\varnothing,W}^yf=f$.
We now prove that $(\gamma_{V,W}^yf)(x)$ is increasing in $x,y$ for 
every bounded increasing function $f$ by induction on $|V|+|W|$.  The 
claim is obviously true for $|V|+|W|=1$ by the monotonicity of $\gamma$.  
In the remainder of the proof, we suppose that the claim is true for
all $V,W$ such that $|V|+|W|=m$, and we proceed with the induction step.

Fix $V,W$ such that $|V|+|W|=m+1$.  The claim is trivially true if
$W\cap V=\varnothing$ by the monotonicity of $\gamma$.  Otherwise, fix 
$v\in W\cap V$ and a bounded increasing function $f$.  Then
$$
	\gamma_{V,W}^yf = 
	\gamma_{V,W}^y
	\gamma_{V\backslash\{v\},W}^yf
$$
holds by the same argument as in the proof of Proposition 
\ref{prop:cgibbs}.  Moreover, by the induction hypothesis, 
$(\gamma_{V\backslash\{v\},W}^yf)(x)$ is increasing in $x,y$,
and by construction $(\gamma_{V\backslash\{v\},W}^yf)(x)$ depends
only on $x_v$ and $x_{V^c},y$.  To show that
$(\gamma_{V,W}^yf)(x)$ is increasing in $x,y$, it thus suffices to 
show that $\gamma_{V,W}^y(x,A)$ is increasing in $x,y$ for 
$\mathbf{1}_A(x)=\mathbf{1}_{x_v=1}$.  But note that
$$
	\frac{\gamma^y_{V,W}(x,A)}{
	1-\gamma^y_{V,W}(x,A)} =
	e^{2\beta_v y_v}\,
	\frac{\gamma_{V,W\backslash\{v\}}^y(x,A)}{
	1-\gamma_{V,W\backslash\{v\}}^y(x,A)}.
$$
As $\gamma_{V,W\backslash\{v\}}^y(x,A)$ is increasing in $x,y$ by the 
induction hypothesis, $\gamma^y_{V,W}(x,A)$ is
also increasing in $x,y$.  Thus the induction step is established, and the 
proof is complete.
\end{proof}

A classical consequence of the monotonicity of $\gamma$ is that it implies 
that $\mathscr{G}(\gamma)$ has a maximal and a minimal element.  This 
substantially simplifies the characterization of uniqueness.  The 
following standard result (see, for example, \cite[section 4.3.3]{Bov06}) 
makes this idea precise.  In the sequel, we define the maximal and minimal 
configurations $\boldsymbol{+},\boldsymbol{-}\in \{-1,1\}^{\mathbb{Z}^d}$ 
as $\boldsymbol{+}_v=1$ and $\boldsymbol{-}_v=-1$ for all 
$v\in\mathbb{Z}^d$; moreover, a function 
$f:\{-1,1\}^{\mathbb{Z}^d}\to\mathbb{R}$ is said to be local if it 
depends only on a finite number of coordinates.

\begin{lem}
\label{lem:monogibbs}
Suppose that $\gamma$ is monotone.  There exist laws $\mathbf{P}_+,
\mathbf{P}_-$ in $\mathscr{G}(\gamma)$ such that
$$
	(\gamma_Vf)(\boldsymbol{-})\uparrow
	\mathbf{E}_-[f(X)]\quad
	\mbox{and}\quad
	(\gamma_Vf)(\boldsymbol{+})\downarrow
	\mathbf{E}_+[f(X)]\quad
	\mbox{as}\quad V\uparrow\mathbb{Z}^d
$$
for every local increasing function $f$.  Moreover,
$$
	\mathbf{E}_-[f(X)]\le\mathbf{E}[f(X)]\le\mathbf{E}_+[f(X)]
$$
for every $\mathbf{P}$ in $\mathscr{G}(\gamma)$ and local increasing 
function $f$.  In particular, $|\mathscr{G}(\gamma)|=1$ iff
$\mathbf{P}_+=\mathbf{P}_-$.
\end{lem}

\begin{proof}
Let $f$ be a bounded increasing function and let $W\subset 
V\subset\subset\mathbb{Z}^d$.  As $\gamma_Wf$ is increasing and
$\gamma_V=\gamma_V\gamma_W$, we readily obtain 
$\gamma_Wf(\boldsymbol{-})\le \gamma_Vf  \le\gamma_Wf(\boldsymbol{+})$.
Thus the net 
$(\gamma_Vf(\boldsymbol{-}))_{V\subset\subset\mathbb{Z}^d}$ is
increasing and the net
$(\gamma_Vf(\boldsymbol{+}))_{V\subset\subset\mathbb{Z}^d}$ is
decreasing.

Now note that $\{-1,1\}^{\mathbb{Z}^d}$ is compact and metrizable (for 
the product topology).  Thus 
$(\gamma_V(\boldsymbol{-},\,\cdot\,))_{V\subset\subset\mathbb{Z}^d}$ and
$(\gamma_V(\boldsymbol{+},\,\cdot\,))_{V\subset\subset\mathbb{Z}^d}$ are
precompact for the weak convergence topology.  In particular, we can 
choose a cofinal increasing sequence $V_n\subset\subset\mathbb{Z}^d$ such 
that $\gamma_{V_n}(\boldsymbol{-},\,\cdot\,)\to\mathbf{P}_-$ and
$\gamma_{V_n}(\boldsymbol{+},\,\cdot\,)\to\mathbf{P}_+$ weakly for some 
probability measures $\mathbf{P}_-$ and $\mathbf{P}_+$, respectively.
By the above monotonicity, it follows readily that 
$$
	\lim_{V\subset\subset\mathbb{Z}^d}(\gamma_Vf)(\boldsymbol{-})
	=\mathbf{E}_-[f(X)]\quad\mbox{and}\quad
	\lim_{V\subset\subset\mathbb{Z}^d}(\gamma_Vf)(\boldsymbol{+})
	=\mathbf{E}_+[f(X)]
$$
for every local increasing function $f$.  Moreover, for any measure 
$\mathbf{P}$ in $\mathscr{G}(\gamma)$, we have
$$
	\mathbf{E}_-[f(X)] =
	\lim_{V\subset\subset\mathbb{Z}^d}(\gamma_Vf)(\boldsymbol{-})
	\le
	\lim_{V\subset\subset\mathbb{Z}^d}\mathbf{E}[(\gamma_Vf)(X)]
	= \mathbf{E}[f(X)]
$$
for every local increasing function $f$, and similarly for the upper 
bound.

We now argue that $\mathbf{P}_-$ and $\mathbf{P}_+$ are in fact in 
$\mathscr{G}(\gamma)$.  To this end, fix any 
$W\subset\subset\mathbb{Z}^d$.  As $W\subset V_n$ for all $n$ sufficiently 
large (as $\{V_n\}$ is cofinal), and as $\gamma_Wf$ is local if $f$ is 
local (by the Markov property of the random field), we can write
$$
	\mathbf{E}_-[g(X_{W^c})\,(\gamma_Wf)(X)] =
	\lim_{n\to\infty}(\gamma_{V_n}\gamma_W fg)(\boldsymbol{-}) =
	\lim_{n\to\infty}(\gamma_{V_n}fg)(\boldsymbol{-}) =
	\mathbf{E}_-[g(X_{W^c})\,f(X)]
$$
for all local functions $f,g$.  Thus $\mathbf{E}_-[f(X)|X_{W^c}]=
(\gamma_Wf)(X)$ for all $W\subset\subset\mathbb{Z}^d$, which implies
that $\mathbf{P}_-$ is in $\mathscr{G}(\gamma)$.  The conclusion for
$\mathbf{P}_+$ follows identically.

Finally, we argue that $|\mathscr{G}(\gamma)|=1$ iff
$\mathbf{P}_+=\mathbf{P}_-$.  If $\mathbf{P}_+\ne\mathbf{P}_-$, then 
evidently $|\mathscr{G}(\gamma)|\ge 2$.  On the other hand, if 
$\mathbf{P}_+=\mathbf{P}_-$, then the expectation of every local 
increasing function must coincide for all $\gamma$-specified random 
fields.  But then all $\gamma$-specified random 
fields coincide, as the local increasing functions are 
measure-determining (the moment generating function 
$\mathbf{E}[e^{\lambda_1X_{v_1}+\cdots+\lambda_mX_{v_m}}]$
determines uniquely the joint law of $X_{v_1},\ldots,X_{v_m}$,
and $f(x_{v_1},\ldots,x_{v_m})=e^{\lambda_1x_{v_1}+\cdots+\lambda_mx_{v_m}}$
is local and increasing for every $\lambda_1,\ldots,\lambda_m\ge 0$).
\end{proof}

We call $\mathbf{P}_+$ and $\mathbf{P}_-$ the maximal and minimal 
elements of $\mathscr{G}(\gamma)$.  Now recall that if $\gamma$ is 
monotone, then so is $\gamma^y$.  Thus $\mathscr{G}(\gamma^y)$ also has a 
maximal and a minimal element. The key step in the proof of Theorem 
\ref{thm:mono} is the following observation due to F\"ollmer \cite{Fol80}.

\begin{lem}
\label{lem:foll}
Suppose that $\gamma$ is monotone.  Let $\mathbf{P}_+$ and 
$\mathbf{P}_-$ be the maximal and minimal element of 
$\mathscr{G}(\gamma)$, and let $\mathbf{P}^y_+$ and $\mathbf{P}^y_-$ be 
the maximal and minimal element of $\mathscr{G}(\gamma^y)$.
Then $\mathbf{P}_+^Y$ is a version of
$\mathbf{P}_+[X\in\,\cdot\,|Y]$ and $\mathbf{P}_-^Y$ is a version of
$\mathbf{P}_-[X\in\,\cdot\,|Y]$.
\end{lem}

\begin{proof}
Let us prove the result for $\mathbf{P}_+$; the conclusion for 
$\mathbf{P}_-$ follows in the same manner.  Let $f,g$ be local increasing 
functions.  It suffices to show that
$$
	\mathbf{E}_+[\mathbf{E}^Y_+[f(X)]\,g(Y)]=
	\mathbf{E}_+[f(X)g(Y)]
$$
for all local increasing functions $f,g$.  By Lemma \ref{lem:monogibbs}
$$
	\mathbf{E}_+[\mathbf{E}^Y_+[f(X)]\,g(Y)] =
	\inf_{V\subset\subset\mathbb{Z}^d}
	\mathbf{E}_+[(\gamma_V^Yf)(\boldsymbol{+})\,g(Y)].
$$
Now note that $(\gamma_V^yf)(\boldsymbol{+})$ is a local increasing 
function of $y$ by Lemma \ref{lem:inhmono}.  It is easily verified (as
$p_v\le \frac{1}{2}$ for all $v$) that 
$\mathbf{E}_+[h(Y_V)|X]=\mathbf{E}_+[h(Y_V)|X_V]$ is an increasing 
function of $X_V$ for every increasing function $h$.  Thus we obtain 
using Lemma \ref{lem:monogibbs}
\begin{align*}
	\mathbf{E}_+[\mathbf{E}^Y_+[f(X)]\,g(Y)] &=
	\inf_{V\subset\subset\mathbb{Z}^d}
	\mathbf{E}_+[\mathbf{E}_+[(\gamma_V^{Y}f)(\boldsymbol{+})
	\,g(Y)|X]] 
	\\ &=
	\inf_{V\subset\subset\mathbb{Z}^d}
	\inf_{W\subset\subset\mathbb{Z}^d}
	\int \mathbf{E}_+[(\gamma_V^{Y}f)(\boldsymbol{+})
	\,g(Y)|X=x]\,\gamma_W(\boldsymbol{+},dx)
	\\ &=
	\inf_{V\subset\subset\mathbb{Z}^d}
	\int \mathbf{E}_+[(\gamma_V^{Y}f)(\boldsymbol{+})
	\,g(Y)|X=x]\,\gamma_V(\boldsymbol{+},dx),
\end{align*}
where the last equality follows as the quantity inside the infimum is 
decreasing in both $V$ and $W$.  But note that we have for every $V$ 
sufficiently large that $g(y)=g(y_V)$
\begin{align*}
        &\int \mathbf{E}_+[(\gamma_V^{Y}f)(\boldsymbol{+})
        \,g(Y)|X=x]\,\gamma_V(\boldsymbol{+},dx) \\ &=
	\sum_{y\in\{-1,+1\}^V} (\gamma_V^{y}f)(\boldsymbol{+})
	\,g(y)\int\prod_{v\in V}g(x_v,y_v)
	\,\gamma_V(\boldsymbol{+},dx) \\ &=
	\int \sum_{y\in\{-1,+1\}^V} f(x)\,g(y)\prod_{v\in V}g(x_v,y_v)
	\,\gamma_V(\boldsymbol{+},dx) \\ &=
	\int f(x)\,\mathbf{E}_+[g(Y)|X=x]\,\gamma_V(\boldsymbol{+},dx).
\end{align*}
Taking the infimum over $V$, it follows that
$$
	\mathbf{E}_+[\mathbf{E}^Y_+[f(X)]\,g(Y)] =
	\mathbf{E}_+[f(X)\,\mathbf{E}_+[g(Y)|X]] =
	\mathbf{E}_+[f(X)g(Y)],
$$
and the proof is complete.
\end{proof}

We can now easily complete the proof of Theorem \ref{thm:mono}.

\begin{proof}[Proof of Theorem \ref{thm:mono}]
By Lemma \ref{lem:monogibbs}, $|\mathscr{G}(\gamma)|=1$
implies $\mathbf{P}_+=\mathbf{P}_-$.  But then
$$
	\mathbf{P}^Y_+ = \mathbf{P}_+[X\in\,\cdot\,|Y] =
	\mathbf{P}_-[X\in\,\cdot\,|Y] = \mathbf{P}^Y_-
	\quad\mbox{a.s.}
$$
by Lemma \ref{lem:foll}.
Thus $|\mathscr{G}(\gamma^Y)|=1$ a.s.\ by Lemma \ref{lem:monogibbs}.
\end{proof}

\begin{rem}
\label{rem:extmono}
Even when the underlying field is not unique, that is, when 
$\mathbf{P}_+\ne\mathbf{P}_-$, Lemma \ref{lem:foll} immediately shows that 
the maximal and minimal models $\mathbf{P}_+$ and $\mathbf{P}_-$ are both 
conditionally mixing.  Indeed, Lemma \ref{lem:foll} implies that the 
conditional distributions $\mathbf{P}_+[X\in\,\cdot\,|Y]$ and 
$\mathbf{P}_-[X\in\,\cdot\,|Y]$ are maximal and minimal in 
$\mathscr{G}(\gamma^Y)$ a.s., and must therefore be extremal a.s.\ (so 
that conditional mixing follows by Proposition \ref{prop:cgibbs}).

Lemma \ref{lem:monogibbs} might lead one to expect that if $\gamma$ is 
monotone, then $\mathbf{P}_+$ and $\mathbf{P}_-$ are the only extremal 
random fields in $\mathscr{G}(\gamma)$.  If a monotone specification 
$\gamma$ satisfies this property, then we have shown that mixing implies 
conditional mixing, that is, we have resolved Conjecture \ref{conj:rfobs} 
for monotone random fields.  However, it turns out that there may exist 
extremal random fields defined by monotone specifications that differ from 
$\mathbf{P}_+,\mathbf{P}_-$.  For example, for the ferromagnetic Ising 
model, the situation is as follows.  On the two-dimensional lattice 
$\mathbb{Z}^2$, the maximal and minimal fields $\mathbf{P}_+,\mathbf{P}_-$ 
are the only extremal fields.  However, in higher dimensions 
$\mathbb{Z}^d$, $d\ge 3$ there exist extremal fields other than 
$\mathbf{P}_+,\mathbf{P}_-$.  On the other hand, in any dimension, 
$\mathbf{P}_+,\mathbf{P}_-$ are the only \emph{translation-invariant} 
extremal fields.  These are highly nontrivial results, and we refer to 
\cite{CV12} and the references therein for further details.  By combining 
these facts, we can resolve several variants of Conjecture 
\ref{conj:rfobs} for the ferromagnetic Ising model.  However, the 
situation for general monotone fields remains unclear.
\end{rem}

\section{Discussion}
\label{sec:disc}

Our main aim in this paper has been to draw attention to the existence of 
surprising probabilistic phenomena in nonlinear filtering that arise in 
infinite dimension, and that are fundamentally different in nature from 
the problems that have been investigated in filtering theory to date.  In 
particular, we have exhibited the existence of conditional phase 
transitions, and we made some first steps toward a general theory. It is 
only fair to emphasize, however, that none of our general results are 
entirely satisfactory, and it appears in particular that we are far from 
understanding the issues surrounding Conjectures \ref{conj:obs} and 
\ref{conj:rfobs} at any level of generality.  It is likely that new ideas 
are needed to develop a deeper understanding of these questions.

Throughout this paper, we have repeatedly exploited connections between 
filtering in infinite dimension and problems in statistical mechanics and 
in multidimensional ergodic theory.  In this final section, we will 
briefly discuss some further connections with other probabilistic problems 
that did not arise in the previous sections.

\subsection{Measure-theoretic identities and conditional ergodicity}
\label{sec:meas}

In this paper we have formulated the filter stability problem as one of 
inheritance of ergodicity.  It is insightful, however, to consider the 
problem from a different perspective: following Kunita \cite{Kun71}, the 
filter stability problem can be equivalently rephrased in terms of the 
validity of a simple \emph{measure-theoretic} identity (see \cite{vH12} 
for further discussion and references).  In view of this connection, the 
counterexamples given above provide some new insight into the failure of 
such measure-theoretic identities.  

We discuss this issue briefly here.  Let us begin by rephrasing the filter 
stability property.  Let $(X_k,Y_k)_{k\in\mathbb{Z}}$ be a stationary 
hidden Markov model as in section \ref{sec:filt}.  Then
\begin{align*}
	&\mathbf{E}|\mathbf{P}[X_k\in A|X_0,Y_1,\ldots,Y_k]-
	\mathbf{P}[X_k\in A|Y_1,\ldots,Y_k]| = \mbox{} 
	&\mbox{(stationarity)} \\
	&\mathbf{E}|\mathbf{P}[X_0\in A|X_{-k},Y_{-k+1},\ldots,Y_0]
	-\mathbf{P}[X_0\in A|Y_{-k+1},\ldots,Y_0]| = \mbox{}
	&\mbox{(Markov property)} \\
	&\mathbf{E}|\mathbf{P}[X_0\in A|
	\mathcal{X}_{-\infty}^{-k}\vee\mathcal{Y}_{-\infty}^{0}]-
	\mathbf{P}[X_0\in A|\mathcal{Y}_{-k+1}^{0}]|
	\xrightarrow{k\to\infty} \mbox{} 
	&\mbox{(martingale cvg.)} \\
	&\mathbf{E}|\mathbf{P}[X_0\in A|
	\textstyle{\bigcap_{k}}(\mathcal{X}_{-\infty}^{-k}\vee\mathcal{Y}_{-\infty}^{0})
	]-\mathbf{P}[X_0\in A|\mathcal{Y}_{-\infty}^{0}]|,
\end{align*}
where we define here and below the $\sigma$-fields
$$
	\mathcal{X}_{m}^{n}:=\sigma\{X_m,\ldots,X_n\},\qquad\quad
	\mathcal{Y}_{m}^{n}:=\sigma\{Y_m,\ldots,Y_n\}.
$$
Thus the filter is stable if and only if
$$
	\mathbf{P}[X_0\in\cdot\,|
	\textstyle{\bigcap_{k}}(\mathcal{X}_{-\infty}^{-k}\vee\mathcal{Y}_{-\infty}^{0})
	]=\mathbf{P}[X_0\in\cdot\,|\mathcal{Y}_{-\infty}^{0}].
$$
In particular, validity of the measure-theoretic
identity
$$
	\bigcap_{k}(\mathcal{X}_{-\infty}^{-k}\vee\mathcal{Y}_{-\infty}^{0}) =
	\mathcal{Y}_{-\infty}^{0}\quad\mathop{\mathrm{mod}}\mathbf{P}
$$
is sufficient for filter stability. In many cases, this identity can also 
be shown to be necessary \cite{vH12}.
In precisely the same manner, it is not difficult to show that
the unobserved Markov chain $(X_k)_{k\in\mathbb{Z}}$ is stable if and only 
if the tail $\sigma$-field
$$
	\bigcap_k\mathcal{X}_{-\infty}^{-k}\quad
	\mbox{is }\mathbf{P}\mbox{-trivial}.
$$
Thus we can formulate a measure-theoretic version of the filter 
stability problem:
\begin{quote}
\emph{When does $\mathbf{P}$-triviality of
$\bigcap_{k}\mathcal{X}_{-\infty}^{-k}$ imply
$\bigcap_{k}(\mathcal{X}_{-\infty}^{-k}\vee\mathcal{Y}_{-\infty}^{0}) = 
\mathcal{Y}_{-\infty}^{0}$ $\mathrm{mod}\,\mathbf{P}$?}
\end{quote}
It is tempting to conclude that this is always the case: as 
$\mathcal{Y}_{-\infty}^{0}$ does not depend on $k$, one would expect that
$\bigcap_{k}(\mathcal{X}_{-\infty}^{-k}\vee\mathcal{Y}_{-\infty}^{0})
\,\stackrel{?}{=}\,
(\bigcap_{k}\mathcal{X}_{-\infty}^{-k})\vee\mathcal{Y}_{-\infty}^{0}$ 
and thus the filter stability property would automatically follow from 
stability of the underlying model.  This (incorrect) reasoning was used
by Kunita \cite{Kun71} in the proof of his main result.  The conclusion 
is already
contradicted, however, by Example \ref{ex:bla}!  The fundamental issue
is that \emph{the exchange of intersection $\cap$ and 
supremum $\vee$ of $\sigma$-fields is not permitted} in general.

An entirely analogous measure-theoretic formulation appears in the random 
field setting of section \ref{sec:condrf}.  Indeed, let 
$(X_v,Y_v)_{v\in\mathbb{Z}^d}$ be a partially observed random field model 
as in section \ref{sec:crf}, and define 
$\mathcal{X}_V:=\sigma\{X_v:v\in V\}$ and
$\mathcal{Y}_V:=\sigma\{Y_v:v\in V\}$
for $V\subseteq\mathbb{Z}^d$.  Then the question whether mixing 
implies conditional mixing can be phrased as:
\begin{quote}
\emph{When does
$$
	\bigcap_{V\subset\subset\mathbb{Z}^d}
	\mathcal{X}_{\mathbb{Z}^d\backslash V}
	\mbox{ is }\mathbf{P}\mbox{-trivial}
$$
imply
$$
	\bigcap_{V\subset\subset\mathbb{Z}^d}
	(\mathcal{X}_{\mathbb{Z}^d\backslash V}
	\vee\mathcal{Y}_{\mathbb{Z}^d}) = 
	\mathcal{Y}_{\mathbb{Z}^d}
	~\mathrm{mod}\,\mathbf{P}~?
$$}
\end{quote}
Once again, Example \ref{ex:bla} already shows this is not always the
case (even for $d=1$).

Establishing the validity of the exchange of intersection and supremum of 
$\sigma$-fields is a measure-theoretic conundrum that poses a tantalizing 
problem in a diverse range of probabilistic questions \cite{vW83,vH12}.  
As this problem evidently lies at the heart of the ergodic theory of 
nonlinear filters, it is interesting to view results in this area from a 
measure-theoretic perspective.  For example, as the inheritance of 
ergodicity has been established under nondegeneracy of the observations 
(Theorem \ref{thm:vH09c} and generalizations developed in \cite{TvH13}), 
one might ask whether this result is a manifestation of a more general 
measure-theoretic principle that enables the exchange of intersection and 
supremum of $\sigma$-fields.  Such a hope could be justified as follows.  
If $\mathcal{G}_n$ is a decreasing filtration and $\mathcal{F}$ is a 
$\sigma$-field, it is easily established that 
$\bigcap_n(\mathcal{G}_n\vee\mathcal{F})=(\bigcap_n\mathcal{G}_n) 
\vee\mathcal{F}$ $\mathrm{mod}\,\mathbf{P}$ when $\mathcal{G}_0$ and 
$\mathcal{F}$ are independent.  Using the Bayes formula, one can readily 
deduce that the conclusion still follows if $\mathbf{P}\ll\mathbf{Q}$ such 
that $\mathcal{G}_0$ and $\mathcal{F}$ are independent under $\mathbf{Q}$.  
The nondegeneracy assumption, however, implies a weaker property: it 
implies that the laws of $(X_k)_{|k|\le K}$ and $(Y_k)_{|k|\le K}$ can be 
made independent by an equivalent change of measure for every 
\emph{finite} $K$ (but not for infinite $K$ as would be required to apply 
the above principle).  Thus Theorem \ref{thm:vH09c} might suggest the 
existence of a more general \emph{local} form of the above 
measure-theoretic principle in the setting of stationary processes.

Unfortunately, this hope proves to be unfounded: a counterexample in 
\cite{vH12} shows that the exchange of intersection and supremum may fail 
even when the observations are nondegenerate (note that this does not 
contradict Theorem \ref{thm:vH09c}, as its uniform ergodicity assumption 
is strictly stronger than the stability property: thus the stronger 
hypothesis is essential to utilize nondegeneracy). However, the 
counterexample of \cite{vH12} is decidedly esoteric.  In contrast, the 
results of the present paper show that the validity of the exchange of 
intersection and supremum of $\sigma$-fields can fail in a very robust 
manner in infinite dimension.  For example, Example \ref{ex:rfbla} defines 
an almost trivial model of a stationary and ergodic random field with 
values in a finite set for which all finite-dimensional marginals admit a 
density that is bounded away from zero, while the validity of exchange of 
intersection and supremum of $\sigma$-fields exhibits a phase transition.  
These new measure-theoretic examples illustrate the highly nontrivial 
nature of the exchange of intersection and supremum problem even in very 
regular models.

\subsection{Spin glasses and conditional Gibbs measures}

We have seen in section \ref{sec:crf} that when a Markov random field is 
conditioned on partial observations, the conditional distribution again 
defines a Markov random field, albeit with a \emph{random} specification 
(through dependence on the observations whose realization is random).  
This idea has played a central 
role throughout this paper, both in section \ref{sec:condrf} where the 
objects of interest were the conditional random fields themselves, and in 
sections \ref{sec:phtrans} and \ref{sec:obs} where the space-time random 
field was a basic tool in the investigation of the corresponding filtering 
problems in infinite dimension.

Random fields whose specification depends on external random variables 
have been widely studied in the statistical mechanics of disordered 
systems.  The randomness, which is typically assumed to be i.i.d.\ across 
sites, is used here to model the irregular spatial distribution of 
impurities in materials such as magnetic alloys.  Such disordered 
materials, known as \emph{spin glasses}, exhibit many remarkable 
properties and complex behavior; we refer to \cite{Bov06} for an 
introduction to the mathematical study of spin glasses.  One of the chief 
difficulties in spin glasses is that the random interactions in the system 
do not favor a particular alignment of the configurations at different 
sites, a phenomenon known as `frustration'.  This creates a very 
complicated energy landscape that gives rise to many unusual properties of 
spin glasses, with a large number of mathematical mysteries remaining 
to be understood (see, e.g., \cite[section 6.5]{Bov06}).

In a sense, our conditional random fields could be interpreted as a form 
of spin glass models.  For example, if the underlying field $X$ is an 
Ising model (cf.\ Example \ref{ex:ising}), then the conditional random 
field in Example \ref{ex:rfbla} can be viewed as a form of \emph{random 
bond Ising model} or \emph{Ising spin glass}, while the setting of 
Conjecture \ref{conj:rfobs} leads to a form of \emph{random field Ising 
model}, cf.\ \cite{Bov06,GHM01}.  There is an essential difference, 
however, with the usual definition of a spin glass: the random variables 
that enter the conditional specification are not defined by an exogenous 
i.i.d.\ disorder, but rather by observations whose law is generated by the 
underlying random field.  Thus the statistics of the randomness is 
perfectly matched to the underlying model.  In particular, the conditional 
distribution should favor the same configurations as the underlying model, 
so that one might expect that the effect of frustration is mitigated. It 
is not clear, however, how such ideas could be made precise.  More 
generally, this highlights the fundamental challenge in investigating 
conditional random field problems from the point of view of spin glasses: 
while conditional random fields are expected to exhibit much `nicer' 
behavior than general spin glasses due to the precise matching of the 
randomness to the model, it is entirely unclear how the latter can be 
taken into account in the analysis of such problems.

\begin{rem}
It should be noted that the above intuition for the special nature of 
the problems considered in this paper fails to hold if we were to consider 
model misspecification.  For example, if the conditional random field in 
Example \ref{ex:rfbla} is computed assuming an error probability $p$ that 
is misspecified with respect to the true error probability of the 
underlying observations, then there is no reason to expect improved 
behavior with respect to general spin glasses.  The misspecified setting 
is completely unnatural from the point of view of conditional 
distributions: if $\mathbf{P}_p$ denotes the law of the model with error 
probability $p$, then $\mathbf{P}_p[X\in\cdot\,|Y]$ is not even defined 
under $\mathbf{P}_q$ for $p\ne q$ as then $\mathbf{P}_p\perp\mathbf{P}_q$.  
Such problems are therefore of a fundamentally different nature than those 
considered in this paper.  On the other hand, one can never expect to have 
perfect knowledge of the underlying model in real-world applications, so 
that the present example highlights the need for caution in the practical 
interpretation of our results.
\end{rem}

\begin{rem}
\label{rem:nishi}
Unlike in most spin glass models, the observations that enter the 
specification of the conditional random field fail to be independent in 
almost any model of interest.  However, in models that possess certain 
symmetries, such as Example \ref{ex:rfbla}, it is possible to transform 
the conditional random field to a spin glass model with i.i.d.\ randomness 
by means of a gauge transformation.  This fact was exploited implicitly in 
the proofs in section \ref{sec:phtrans} to analyze the phase transition 
behavior in this model.  The converse observation has been noted in the 
physics literature: in spin glasses that possess gauge 
symmetries, it is possible to make a special choice of the randomness so 
that various quantities become explicitly computable.  This restricts 
consideration to a distinguished subset of the parameter space of the 
model called the \emph{Nishimori line}, which corresponds precisely to the 
situation where the model describes a conditional distribution (after 
gauge transformation).  We refer to \cite{Nis07} and the references 
therein.
\end{rem}

An entirely different manner in which conditional random fields arise in 
statistical mechanics is in the study of renormalization group 
transformations.  In section \ref{sec:condrf} we have investigated Markov 
random fields where $\mathbf{P}[X_V\in\cdot|X_{V^c}]$ depends on 
$X_{\partial V}$ only.  In renormalization group analyses, one aims to 
perform computations at the level of a transformation $Y$ of the random 
field $X$ obtained, for example, by working on a sublattice. In general, 
such a transformation is no longer Markov, that is, 
$\mathbf{P}[Y_W\in\cdot|Y_{W^c}]$ depends on all sites outside 
$W\subset\subset\mathbb{Z}^d$.  However, one still expects that the 
resulting model is \emph{local} in the sense that 
$\mathbf{P}[Y_W\in\cdot|Y_{W^c}]$ only depends significantly on $Y_w$ for 
$w$ in some finite neighborhood of $W$.  This idea is formalized by the 
requirement that the functions 
$y\mapsto\mathbf{P}[Y_W\in A|Y_{W^c}=y_{W^c}]$ are quasilocal, that is, 
uniform limits of local functions, for every 
$W\subset\subset\mathbb{Z}^d$.  Such random fields are called 
\emph{Gibbsian}.  In order for the renormalization procedure to be well 
defined, one therefore faces the question whether a transformation 
$Y$ of a Markov random field $X$ is guaranteed to be Gibbsian.

Surprisingly, this natural property proves to be false in general: it may 
be that $Y$ fails to be Gibbsian even when $X$ is Markov.  Such 
\emph{renormalization group pathologies} are discussed at great length in 
the paper \cite{EFS93}.  This issue proves to be closely connected, at 
least in monotone fields, with the uniqueness property of the conditional 
random field $\mathbf{P}[X\in\cdot|Y]$, and thus with the type of 
phenomena investigated in section \ref{sec:condrf}.  We refer for further 
details to the paper \cite{FP97}, which also develops the connection with 
global Markov property \cite{Gol80, Fol80} (already mentioned in the proof 
of Theorem \ref{thm:mono}).

\subsection{Trees and other graphs}
\label{sec:graphs}

Throughout this paper, we have considered spatial variables indexed by the 
lattice $\mathbb{Z}^d$.  However, both infinite-dimensional filtering 
models and partially observed Markov random fields can be meaningfully 
formulated on any infinite graph of finite degree $G=(V,\mathcal{E})$.  
While the lattice setting is a natural prototype for the investigation of 
infinite-dimensional systems, the consideration of other graphical 
structures could provide complementary insights on the phenomena 
considered in this paper.

Of particular interest is the case where the spatial graph $G$ is a tree. 
Such models are motivated by applications in information theory, 
statistical physics, and mathematical biology, cf.\ \cite{EKPS00} and the 
references therein.  More importantly, trees play a special role in the 
theory of Markov random fields, as many problems that are exceedingly 
difficult in general graphs are amenable to exact computations on a tree 
by means of recursive formulas (see, for example, \cite[Chapter 
12]{Geo11}).  It is therefore possible that significant insight could be 
obtained on conjectures that arise from this paper by investigating these 
on trees using the special tools that are available in this setting.  We 
have not systematically pursued this avenue of investigation.  
Nonetheless, the statistical mechanics literature on spin glass models on 
Bethe lattices (infinite regular trees in which every vertex has exactly 
$q\ge 3$ neighbors) already provides some suggestive evidence towards 
questions that arise from this paper, as we will now briefly discuss.

\begin{enumerate}
\item The main conjecture of this paper is that stability properties are 
preserved by conditioning in the absence of symmetries.  However, the 
theory of section \ref{sec:obs} raises the question whether stability 
of the underlying model even plays a role in this setting: it is possible 
that the stability of conditional distributions follows already from the 
symmetry-breaking assumption alone even when the underlying model is not 
stable.  On the other hand, the theory of section \ref{sec:mono} does 
require stability of the model, suggesting that the latter 
should play a role after all.

Some insight into this question can be obtained from results 
in \cite{BRZ98} on the random field Ising model on the Bethe lattice, 
which can be interpreted as the conditional random field for the variant 
of the model of section \ref{sec:mono} where the underlying field $X$ is a 
ferromagnetic Ising model on the Bethe lattice.  It is shown in 
\cite{BRZ98} that if the underlying field $X$ is not unique and the error 
probability $p$ is sufficiently close to $\frac{1}{2}$, then the 
conditional random field is also not unique for any realization of the 
observations.  Thus the conclusion of Theorem \ref{thm:mono} fails in this 
case, even though the symmetry-breaking condition holds here and the model 
is monotone and translation invariant (in the sense appropriate to the 
Bethe lattice).  Thus, at least in the setting of trees, the 
symmetry-breaking property alone is not sufficient to ensure the 
inheritance of ergodicity in the sense of uniform mixing. This provides 
some evidence towards the necessity of the stability assumption in our 
conjectures. 

\item Section \ref{sec:condrf} introduced two different notions of 
conditional ergodicity for Markov random fields: uniform mixing 
(uniqueness) and mixing (extremality) of the conditional random field.  
It is not clear whether inheritance of the mixing property is connected to 
inheritance of the uniform mixing property: for example, is it possible 
for the conditional distribution of a uniformly mixing random field to be 
mixing but not uniformly mixing?  While we do not have any concrete 
evidence in the present setting, related results on the ferromagnetic 
Ising model on the Bethe lattice \cite{BRZ95} show that it is indeed 
possible, albeit in a different setting, that the uniqueness and 
extremality properties appear at different phase transition points.

\item The conditional random field for the analogue of Example 
\ref{ex:rfbla} on the Bethe lattice can be viewed as a random bond Ising 
model on the Bethe lattice.  This model has been systematically 
investigated in the statistical mechanics literature 
\cite{CCST86,CCST90a,CCST90b}, which leads to a detailed description of 
the phase diagram and, in particular, the exact location of the phase 
transition point on the Nishimori line (cf.\ Remark \ref{rem:nishi}).
\end{enumerate}

Another widely studied class of models in statistical mechanics are 
mean-field models, where the spatial graph is chosen to be the complete 
graph on a finite number $n$ of spatial locations (that is, every pair of 
sites interacts with each other).  The large-scale properties of such 
models do not arise at the level of a fixed graph, but rather in the limit 
as $n\to\infty$ where the strength of the interactions are scaled down 
appropriately to obtain nontrivial limiting behavior.  Mean-field 
analogues of Example \ref{ex:rfbla} arise naturally in applications in 
statistics and computer science \cite{MNS12} and in information theory 
\cite{MM09,MA13}. However, it is arguably less surprising that nontrivial 
phase transition phenomena arise in this setting, as these problems are 
concerned with the limiting behavior of a sequence of conditional 
distributions where the underlying model is scaled in a nontrivial manner.  
In contrast, in the problems we have investigated in this paper, the 
emergent phase transition phenomena arise from conditioning alone.

\subsection{Stability in time vs.\ spatial mixing}

We have discussed in this paper two different questions of inheritance of 
ergodicity: the stability property in time of dynamical filtering models, 
and the mixing property in space in the setting of static random fields.  
While we have utilized the notion of a space-time random field even in the 
dynamical setting, it is not at all clear to what extent the problems of 
inheritance of spatial and temporal ergodicity are related to one another.  
In particular, we do not know whether it is possible that there exist 
filtering models for which the filter is stable, but does not also exhibit 
a spatial mixing property.  From the practical perspective, the design of 
filtering algorithms in high dimension relies crucially both on the 
temporal stability and spatial mixing of the filter \cite{RvH13}.

Connections between spatial and temporal mixing properties arise 
frequently in the ergodic theory of interacting Markov chains.  For 
example, \cite{GKLM89} shows a correspondence in the translation-invariant 
setting between invariant measures of an infinite-dimensional Markov chain 
and random fields corresponding to the specification of the space-time 
system (related ideas in the context of the global Markov property 
can be found in \cite{GKS90, Fol88}).  Similarly, connections between 
exponential mixing in space and time are well known, see for example 
\cite{Lou04} and the references therein.  We do not know, however, whether 
similar connections are likely to arise for conditional distributions.  
Let us emphasize in this context that the arguments that were used in this 
paper to obtain filter stability (section \ref{sec:obs}) and conditional 
mixing (section \ref{sec:condrf}) are of a very different nature, but this
is likely to be a limitation of our analysis rather than a genuine 
phenomenon.


\providecommand{\bysame}{\leavevmode\hbox to3em{\hrulefill}\thinspace}
\providecommand{\MR}{\relax\ifhmode\unskip\space\fi MR }
\providecommand{\MRhref}[2]{%
  \href{http://www.ams.org/mathscinet-getitem?mr=#1}{#2}
}
\providecommand{\href}[2]{#2}


\end{document}